\documentclass[11pt]{amsart}

\usepackage{amsmath, amsthm, amssymb, amsfonts, enumerate}
\usepackage[colorlinks=true,linkcolor=blue,urlcolor=blue,breaklinks=true]{hyperref}
\usepackage{dsfont}
\usepackage{xcolor}
\usepackage{geometry}
\usepackage{soul}
\usepackage{soul}
\usepackage{natbib}
\usepackage{breakcites}

\usepackage{mathtools}
\mathtoolsset{showonlyrefs}

\def \cC{\mathcal{C}}

\def \cS{\mathcal{S}}
\def \cH{\mathcal{H}}
\def \cF{\mathcal{F}}

\def \cO{\mathcal{O}}

\def \T{\mathcal{T}}
\def \P{\mathsf P}
\def \Q{\mathsf Q}

\def \E{\mathsf E}

\def \N{\mathbb{N}}
\def \R{\mathbb{R}}

\def \ud{\mathrm{d}}

\newcommand{\eps}{\varepsilon}

\newtheorem{theorem}{Theorem}[section]
\newtheorem{lemma}[theorem]{Lemma}
\newtheorem{corollary}[theorem]{Corollary}
\newtheorem{proposition}[theorem]{Proposition}
\newtheorem{definition}[theorem]{Definition}
\newtheorem{remark}[theorem]{Remark}
\newtheorem{example}[theorem]{Example}
\newtheorem{assumption}{Assumption}

\makeatletter
\@namedef{subjclassname@2020}{%
  \textup{2020} Mathematics Subject Classification}
\makeatother

\title[Mean-field games of singular control with finite fuel]{Mean-Field Games of Finite-Fuel Capacity Expansion with Singular Controls}
\author[L. Campi, T. De Angelis, M. Ghio, G. Livieri]{Luciano Campi \and Tiziano De Angelis \and Maddalena Ghio \and Giulia Livieri}
\thanks{T. De Angelis gratefully acknowledges support from EPSRC via grant EP/R021201/1; M.~Ghio and G.~Livieri acknowledge the financial support of UniCredit Bank R$\&$D group through the \textit{Dynamical and Information Research Institute} at the Scuola Normale Superiore. We thank the Associate Editor and two Referees for useful comments that improved the quality of the paper.}
\address{L.~Campi: Department of Mathematics ``Federigo Enriques'', University of Milan, Via Saldini 50, 20133, Milan, Italy.}
\email{\href{mailto:luciano.campi@unimi.it}{luciano.campi@unimi.it}}
\address{T.~De Angelis: School of Management and Economics, Dept.\ ESOMAS, University of Turin, C.so Unione Sovietica 218bis, 10134, Torino, Italy; Collegio Carlo Alberto, P.za Arbarello 8, 10122, Torino, Italy}
\email{\href{mailto:tiziano.deangelis@unito.it}{tiziano.deangelis@unito.it}}
\address{M.~Ghio and G.~Livieri: Scuola Normale Superiore, Piazza dei Cavalieri 7, 56126 Pisa, Italy}
\email{\href{maddalena.ghio@sns.it}{maddalena.ghio@sns.it}}
\email{\href{giulia.livieri@sns.it}{giulia.livieri@sns.it}}

\date{\today}

\numberwithin{equation}{section}

\begin{document}

\begin{abstract}
We study Nash equilibria for a sequence of symmetric $N$-player stochastic games of finite-fuel capacity expansion with singular controls and their mean-field game (MFG) counterpart. We construct a solution of the MFG via a simple iterative scheme that produces an optimal control in terms of a Skorokhod reflection at a (state-dependent) surface that splits the state space into {\em action} and {\em inaction} regions. We then show that a solution of the MFG of capacity expansion induces approximate Nash equilibria for the $N$-player games with approximation error $\varepsilon$ going to zero as $N$ tends to infinity. Our analysis relies entirely on probabilistic methods and extends the well-known connection between singular stochastic control and optimal stopping to a mean-field framework.
\end{abstract}

\maketitle

\section{Introduction}\label{sec:intro}
Mean field games (MFGs) were first introduced in \cite{huang2006large} and  \cite{lasry2007mean} as an elegant and tractable way to study approximate Nash equilibria in nonzero-sum symmetric stochastic differential games for a large population of players with a mean-field interaction, i.e., each player interacts with the rest of the population  through its empirical distribution. The idea is that a large-population game of this type should behave similarly to its MFG counterpart, which
may be thought of as an infinite-player version of the game. Exploiting the underlying symmetry while passing to the limit with the number of players tending to infinity, the sequence of games converges in some sense to a problem in which the game structure is preserved in a simpler form for a ``representative player''. Such representative player responds optimally to the distribution of the population which,  at equilibrium, coincides with the distribution of the optimally controlled state variable. Once the limit problem is solved, its solution can typically be implemented in the $N$-player game and provides an approximate Nash equilibrium with vanishing error as $N$ tends to infinity.

From a practical point of view this approximation result can be very useful since tackling directly the $N$-player game is often unfeasible when $N$ is very large due the curse of dimensionality. In the past few years many authors from different mathematical backgrounds have studied this class of games. Two approaches have been adopted to address the study of MFGs: an analytic approach as in, e.g., the initial paper \cite{lasry2007mean}, and a probabilistic one. Here we follow the latter, which has been developed in a series of papers by Carmona, Delarue, and their co-authors (see, e.g., \cite{carmona2013mean,carmona2013probabilistic,carmona2015probabilistic,carmona2016mean}). We refer the reader to the lecture notes by \cite{cardaliaguet2012notes} and the two-volume monograph by \cite{carmona2018probabilistic} for  comprehensive presentations of MFG theory and its applications from analytic and probabilistic perspective, respectively.

The literature on MFGs is rapidly growing. Most of the papers deal with games where players use ``regular controls" in order to optimise their payoffs. Here by regular controls we mean those having a bounded impact on the velocity of the underlying dynamics.
Only few papers have studied the case of MFGs with singular controls, which is a larger class of controls allowing for unbounded changes in the velocity of the underlying process and possible discontinuities in the state trajectories. More specifically, \cite{fu2017mean} established an abstract existence result for solutions of a general MFG with singular control, using the notion of relaxed solutions. The same approach was also applied in \cite{fu2019extended} to extend the previous results to MFG with interaction through the controls as well. In both papers, the issue of finding approximate Nash equilibria in the $N$-player games is left aside. \cite{guo2018class} propose and analyse a class of $N$-player stochastic games that also includes so-called `finite fuel' controls, which will be introduced in detail later in the paper. To the best of our knowledge, only the works of \cite{cao2017approximation}, \cite{guo2019stochastic} and \cite{cao2020mfgs} tackle simultaneously MFGs and $N$-player games with singular controls. Their analysis is based on verification theorems and quasi-variational inequalities specifically designed for their settings and not amenable to simple extensions. For completeness, we also mention the two papers \cite{hu2014singular}, \cite{zhang2012relaxed}, which use a maximum principle approach to solve singular control problems with mean-field dynamics for the state variables. A class of controls closely related to singular controls is that of impulses, which has also received attention recently within MFG theory. We mention the two papers \cite{basei2019nonzero} and \cite{zhou2017mean}, where MFGs with impulse controls are considered and solved using an approach based on quasi-variational inequalities and exploiting the stationarity properties of  their settings. Finally, the article \cite{bertucci2020fokker} proves existence and uniqueness of a MFG equilibrium with impulse controls and it provides a variational characterization of the density of jumping particles arising from such problem.

\subsection{Model description}\label{sec:model} In this article, we study Nash equilibria for a class of symmetric $N$-player stochastic differential games, for large $N$, and we characterize the solutions of the associated MFG. Specifically, we consider a class of finite-fuel capacity expansion games with \textit{singular controls}. The control exerted by each agent (the ``capacity expansion'') is modelled via a bounded variation process and the ``finite-fuel'' condition prescribes that the total variation of the control be bounded by a deterministic constant (i.e., the maximum capacity expansion is capped). Moreover, each agent incurs costs proportional to the amount of fuel being used.

In order to set out our main results, we provide here a short description of the $N$-player game of capacity expansion (see Section \ref{subsec:Nplayer} for a full account). The game is set over a finite-time horizon $T$ given and fixed. We consider a complete probability space $(\Omega,\cF,\bar\P)$ equipped with a right-continuous filtration $\mathbb F:=(\cF_t)_{t\in[0,T]}$ which is augmented with all the $\bar\P$-null sets. There are $N$ players in the game and the $i$-th player, $i=1,\ldots, N$, chooses a strategy $\xi^{N,i} = (\xi^{N,i} _t)_{t \in [0,T]}$ from the set of all right-continuous non-decreasing adapted processes, affecting the player's own private state variables $(X^{N,i},Y^{N,i})$. Given a drift $a\,:\,\mathbb{R} \times [0,1] \rightarrow \mathbb{R}$ and a volatility $\sigma\,:\,\mathbb{R}\rightarrow \mathbb{R}_{+}$, the private states have dynamics

\begin{equation}
\begin{split}
& X_{t}^{N,i}= X_0 ^{i} + \int_{0}^{t}\,a(X_s^{N,i},m_s^{N})\,\ud s + \int_0^t\,\sigma(X_s^{N,i})\,\ud W_s^{i},\\
& Y_{t}^{N,i} = Y_{0-}^{i}+ \xi^{N,i}_t,\quad\quad t \in [0,T],
\end{split}
\label{eq:dynamics_Nplayer_0}
\end{equation}
where $(W^{1},\ldots,W^{N})$ is a $N$-dimensional Brownian motion. The initial conditions $(X_0^{i}, Y_{0-}^{i})$ are i.i.d.~random variables with common distribution $\nu \in \mathcal{P}(\Sigma)$, where $\mathcal{P}(\Sigma)$ is the space of all probability measures on $\Sigma := \mathbb{R} \times [0,1]$. The players interact through the mean-field term $m_t^{N}$ appearing in the drift and given by
\begin{equation}
m_t^{N} = \frac{1}{N}\sum_{i = 1}^{N} Y_t^{N,i} = \int_{\Sigma}\,y\,\mu_t^{N}(\ud x,\ud y), \quad t \in [0,T],\label{emp-mean}
\end{equation}
where $\mu_t^{N} = \frac{1}{N}\sum_{i = 1}^{N} \delta_{(X_t^{N,i}, Y_{t}^{N,i})}$ denotes the empirical probability measure of the players' states with $\delta_{z}$ the Dirac delta mass at $z \in \Sigma$. In \eqref{eq:dynamics_Nplayer_0}, each $\xi^{N,i}$ represents the investment in additional capacity made by the $i$-th player. Each player aims at maximizing an expected payoff of the form

\begin{equation}
J^{N,i} := \bar\E \left[\int_{0}^{T} e^{-rt} f(X_t^{N,i}, Y_t^{N,i})\,\ud t - \int_{[0,T]} e^{-r t}\,c_0\,\ud\xi_{t}^{N,i}\right],
\label{eq:cost_Nplayer_0}
\end{equation}
for a fixed discount rate $r \ge 0$, some cost $c_0 > 0$ and some running payoff $f$ (the same for all players). The optimisation is subject to the so-called \emph{finite-fuel constraint}: $Y_{0-} + \xi_t ^{N,i} \in [0,1]$ for all $t \in [0,T]$ and all $i=1,\ldots, N$. We are interested in computing (approximate) Nash equilibria of this $N$-player game via the MFG approach. This requires to pass to the limit as $N \to \infty$ and to identify the limiting MFG. The latter must be solved (as explicitly as possible) and the associated optimal control is then implemented in the $N$-player game for sufficiently large $N$, as a proxy for the equilibrium strategy.\vspace{+5pt}
 
Singular control problems with finite (and infinite) fuel find numerous applications in the economic literature and originated from the engineering literature in the late 60's (see \cite{bather1967sequential} for a seminal paper and, for example, \cite{benevs1980some,karatzas1985, karoui1988} for early contributions to the finite fuel case). Game versions of these problems are a natural extension of the single agent set-up and allow to model numerous applied situations. Here in particular we make assumptions on the structure of the interaction across players that are suitable to model the so-called {\em goodwill} problem (see, e.g., \cite{marinelli2007stochastic,jack2008singular,ferrari2017continuous} in a stochastic environment and \cite{buratto2002new} in a deterministic one). Specifically, players can be interpreted as firms that produce the same good (e.g., mobile phones) and must decide how to advertise it over a finite time horizon. The $i$-th firm's product has a market price that depends on the particular type/brand (e.g., Apple, Huawei, etc.) and we model that by the process $X^{N,i}$. Each firm can invest in marketing strategies in order to raise the profile of their product and its popularity. The $i$-th firm's cumulative amount of investment that goes towards advertising is modelled by the process $Y^{N, i}$, where the finite-fuel feature naturally incorporates the idea that firms set a maximum budget for advertising over the period $[0,T]$. All firms measure their performance in terms of future discounted revenues: they use a running profit function $(x,y)\mapsto f(x,y)$ and deduct the (proportional) cost of advertising $c_0 \ud \xi$. A typical example is $f(x,y)=x \cdot y^\alpha$, $\alpha\in(0,1)$, where profits are linear in the product's price $x\ge 0$ and increasing and concave as function of the total investment $y$ made towards advertising. 

From the point of view of the $i$-th firm, investing $\Delta \xi^{N,i}>0$ has a cost $c_0\Delta \xi^{N,i}$ and produces two effects. First of all it increases the popularity of the $i$-th firm's product, hence increasing the running profit to the level $f(x,y+\Delta\xi^{N,i})$ (we are tacitly assuming $y\mapsto f(x,y)$ increasing). Secondly, it has a broader impact on the visibility of the type of product (e.g., mobile phones) and will stimulate the public's demand for that good. This has a knock-on effect on the trend of the prices of all the firms that produce the same good. We model this fact via the interaction term $m^N_t$ in the price dynamics and we will assume that the drift function increases with the average spending in advertising across all companies, i.e., $m\mapsto a(x,m)$ is non-decreasing.

\subsection{Our contribution and methodology} We focus on the construction of approximate Nash equilibria for the $N$-player game through solutions of the corresponding MFG. First, we formulate the MFG of capacity expansion, i.e., the limit model corresponding to the above $N$-player games as $N\rightarrow\infty$ (Section \ref{sec:setup}). Then, under mild assumptions on the problem's data we construct a solution in feedback form of the MFG of capacity expansion (Section \ref{sec:existence}). Our constructive approach, based on an intuitive iterative scheme, allows us to determine the optimal control in the MFG in terms of an optimal boundary $(t,x)\mapsto c(t,x)$ that splits the state space $[0,T]\times\Sigma$ into an {\em action} region and an {\em inaction} region; see Theorem \ref{teo:existenceSolMFG} in Section \ref{sec:setup}. The optimal strategy prescribes to keep the controlled dynamics underlying the MFG inside the closure of the inaction region by Skorokhod reflection. Finally, whenever the optimal boundary in the MFG is Lipschitz continuous in its second variable we can show that it induces a sequence of approximate $\varepsilon_N$-Nash equilibria for the $N$-player games with vanishing approximation error at rate $O(1/\sqrt{N})$ as $N$ tends to infinity; see Theorem \ref{teo:approximation} in Section \ref{sec:approximation}. While Lipschitz regularity of optimal boundaries is in general a delicate issue, we provide sufficient conditions on our problem data that guarantee such regularity. Since our conditions are far from being necessary, in Section \ref{subsec:gbm} we also illustrate an example with a clear economic interpretation which violates those conditions and yet yields a Lipschitz boundary.

The proof of Theorem \ref{teo:existenceSolMFG} on the existence of a feedback solution for the limit MFG hinges on a well-known connection between singular stochastic control and optimal stopping (e.g., \cite{baldursson1996irreversible,karatzaShreve1984,karatzasShreve1985}), which we apply in the analysis of our iterative scheme. This approach allows us to overcome the usual difficulties arising from fixed-point arguments often employed in the literature on MFGs. Ideas around iterative schemes for the solution of MFGs date back to \cite{lasry2007mean} and were then adopted in various forms by several other authors including, e.g., \cite{gueant2012mean}, \cite{cardaliaguet2017learning} and \cite{dianetti2019submodular}. Besides offering a {\em constructive} solution method, iterative schemes may be interpreted as {\em learning procedures}. Intuitively, these schemes consist of three steps that begin with a flow of measures obtained from an educated guess based on past iterations. In each iteration the three steps are: (i) the representative agent derives an optimal strategy for the control problem associated to the initial flow of measures, (ii) using such optimal strategy the agent computes certain quantities of interest that she will use in the third step of the iteration, (iii) the agent updates the original guess for the flow of measures using the information acquired in step (ii). The procedure is then repeated.

In our case the MF interaction is via the limit function $m(t)$ obtained by the empirical average of the control $m^N_t$ when letting $N\to\infty$. Our iteration, starts with a constant function $m^{[-1]}(t)\equiv 1$ which is updated at each subsequent step and denoted by $m^{[k]}$ for $k\ge 0$. The learning procedure we adopt is to simply compute $m^{[k]}$ as the average of the optimal control used in the $(k\!-\!1)$-th step (as best response to the previous one). This approach is in a similar spirit to the one used in \cite{gueant2012mean} in an analytical study of MFG equations with quadratic Hamiltonian. Other learning procedures in the existing literature prescribe, for example, to take the average flow of measures over the past iterations \citep[see e.g.][]{cardaliaguet2017learning}.

As a byproduct of our approach we also obtain that a connection between singular control problems of capacity expansion and problems of optimal stopping holds in the setting of our MFG. The finite-fuel condition is not a structural condition and it is clear that our choice of $y\in[0,1]$ is not restrictive: indeed, we could equally consider $y\in[0,\bar y]$ for any $\bar y>0$ (see Remark \ref{rem:fuel}). This is suggestive that our results may be extended to the infinite-fuel setting by considering sequences of problems with increasing fuel and limiting arguments. However, since this extension is non-trivial, we leave it for later work and here focus on the finite-fuel case.

\subsection{Organization of the paper} In Section \ref{sec:setup}, we formulate the MFG of capacity expansion, we state the standing assumptions on the coefficients in the underlying dynamics and on the profit function, and we state the existence result for the MFG (Theorem \ref{teo:existenceSolMFG}). In Section \ref{sec:existence} we prove Theorem \ref{teo:existenceSolMFG} by constructing a solution of the MFG in feedback form via an iterative scheme. In particular, we characterise the optimal control in terms of the solution to a Skorokhod reflection problem at an optimal boundary (surface) that splits the state space $[0,T]\times\Sigma$ into {\em action} and {\em inaction} regions. Finally, in Section \ref{sec:approximation} we formally introduce the symmetric $N$-player game of capacity expansion and we show that the solution of the MFG found in Section \ref{sec:existence} induces approximate Nash equilibria for the $N$-player games, with vanishing error of order $O(1/\sqrt{N})$ as $N\to\infty$.

\subsection{Frequently used notations}\label{sec:notation} We conclude the Introduction with a summary of notations that will be used throughout the paper. 

Let $\Sigma := \mathbb{R} \times [0,1]$ and let $\mathcal{P}(\Sigma)$ denote the set of all probability measures on $\Sigma$ equipped with the Borel $\sigma$-field $\mathcal B(\Sigma)$. Let $\mathcal P_2 (\Sigma)$ be the subset of $\mathcal P(\Sigma)$ of probability measures with finite second moment. The set $\Sigma$ and the $N$-fold product space $\Sigma^N$ are the state spaces for the controlled processes $(X,Y)$ and $(X^{N},Y^{N})$ that are underlying the MFG and the $N$-player game, respectively. Since our problems are set on a finite-time horizon, we also consider ``time" as a state variable and will use the state space $[0,T]\times\Sigma$. Given a set $A\subset [0,T]\times\Sigma$ we denote its closure by $\overline A$. In some cases we will also need $\Sigma':=\R\times(0,1]$ and $\Sigma^\circ:=\R\times(0,1)$. 

Given a filtered probability space $\Pi:=(\Omega, \mathcal{F}, \mathbb{F} =(\mathcal{F}_{t})_{t \geq 0}, \bar \P)$ satisfying the usual conditions and a $\cF_0$-measurable random variable $Z\in[0,1]$, we denote 
\begin{align*}
\Xi^\Pi(Z):= \big \{ \xi:\:&\text{$(\xi_t)_{t\ge 0}$ is $\mathbb F$-adapted, non-decreasing,  right-continuous,}\nonumber\\
&\text{ with $\xi_{0-}=0$ and $Z+\xi_t\in [0,1]$ for all $t\in[0,T]$, $\bar\P$-a.s.} \big \}.
\end{align*}
The set $\Xi^\Pi(Z)$ will be the set of admissible strategies for the players in our games. The random variable $Z$ will be replaced by the initial value of the process $Y$ (for the MFG) or of the process $Y^{N,i}$ (for the $i$-th player in the $N$-player game). Often we will drop the dependence of $\Xi$ on the probability space $\Pi$ and the random variable $Z$, as no confusion shall arise. Notice that the controls in $\Xi^\Pi(Z)$ are of open-loop type, i.e., they are general progressively measurable maps $(t,\omega)\mapsto \xi_t(\omega)$. We show in Theorem \ref{teo:existenceSolMFG} that the MFG admits an optimal control in feedback form which is then used to construct $\eps$-Nash equilibria for the $N$-player games in Theorem \ref{teo:approximation}. Those equilibria are in closed-loop feedback form in the sense that each player's control depends in a non-anticipative way on the state variables.

\section{The mean-field game: setting and main results}\label{sec:setup}

In this section we set-up the mean-field game associated with the $N$-player game described above and we state the main result concerning the existence and structure of the optimal control for this game; see Theorem \ref{teo:existenceSolMFG}. Later, in Section \ref{sec:approximation} we will link the MFG to the $N$-player game. Below we use the notation introduced in Section \ref{sec:notation}.

\subsection{The mean-field game}\label{subsec:mfg}

Let $\Pi=(\Omega, \mathcal{F}, \mathbb{F} =(\mathcal{F}_{t})_{t \geq 0}, \bar{\P})$ be a filtered probability space satisfying the usual conditions and supporting a one-dimensional $\mathbb F$-Brownian motion $W$. Notice that the initial $\sigma$-field $\mathcal F_0$ is not necessarily trivial. 

Let $(X_0, Y_{0-})$ be a two-dimensional $\mathcal F_0$-measurable random variable with joint law $\nu \in \mathcal{P}(\Sigma)$ independent of the Brownian motion, and let $\xi\in\Xi^\Pi(Y_{0-})$ be an admissible strategy. Then, given a bounded Borel measurable function $m : [0,T] \rightarrow [0,1]$, for all $t \in [0,T]$ we define the 2-dimensional, degenerate, controlled dynamics
\begin{equation}\label{eq:dynamics_mfg_01}
\begin{split}
X_t &= X_0 + \int_{0}^{t} a(X_s, m(s))\,\ud s +   \int_{0}^{t} \sigma(X_s) \ud W_s,\\
Y^\xi_t &= Y_{0-} + \xi_t.
\end{split}
\end{equation}
The goal of the ``representative player'' consists of maximizing over the set of all admissible strategies $\xi \in \Xi^\Pi(Y_{0-})$ the following objective functional
\begin{equation}\label{eq:J0}
J(\xi)=\bar \E\left[\int_0^{T}e^{-rt}f(X_t,Y^\xi_t)\ud t-\int_{[0,T]} e^{-rt}c_0\ud\xi_t\right],
\end{equation}
where $\bar \E$ is the expectation under $\bar \P$, $f$ is a running payoff function, $c_0 > 0$ is a cost and $r\ge 0$ a discount rate. Assumptions on all the coefficients appearing in the state variables' dynamics and in the objective functional will be given below. The integral with respect to the positive random measure $\ud \xi$ includes possible atoms at the initial and terminal time (corresponding to possible jumps of $\xi$). 
The value of the optimization problem for the representative player is denoted by
\[
V^{\nu} = \sup_{\xi \in \Xi(Y_{0-})}\,J(\xi).
\]
\begin{remark}
We observe that a slightly more general running cost $f(X_t,Y_t,m(t))$, depending also on the function $m$, could be considered. In that case, in addition to Assumption \ref{ass:f} we would need to require $m\mapsto \partial_y f(x,y,m)$ (continuous) non-decreasing. 
In particular, that extra assumption would allow us to prove Proposition \ref{prop:OSvalueFunc}--(iii) while the rest of our analysis would remain unchanged.
\end{remark}
\begin{remark}\label{rem:fuel}
The choice $Y \in [0,1]$ in the definition of the set $\Xi$ of admissible strategies is with no loss of generality and we could equally consider $Y \in [0,\bar y]$ for $\bar y >0$. The assumption of finite fuel is consistent with real-world applications, where a firm would set aside a certain budget to be spent over a given period $[0,T]$. 
\end{remark}

Since we are interested in the MFG that arises from the $N$-player game \eqref{eq:dynamics_Nplayer_0}-\eqref{eq:cost_Nplayer_0}, in the limit as $N\to \infty$, it is natural to seek for an admissible optimal strategy $\xi$ (given $m$) such that the following consistency condition holds
\begin{align}\label{eq:cons0}
m(t)=\bar \E[\,Y^\xi_t],\quad t\in[0,T].
\end{align}
The consistency condition will appear in the precise definition of MFG solution which will be given in Definition \ref{def:solMFG} below.
 
In order to develop our methodology, it is convenient to state a version of the MFG parametrized by a triple $(t, x, y)$, where $t \in [0,T]$ indicates the initial time and $(x, y) \in \Sigma$ denotes any realization of the states $(X_t, Y_{t-})$. Let us start by denoting 
\[
\P_{0,x,y}(\,\cdot\,):=\bar \P(\,\cdot\,|X_0=x,Y_{0-}=y)
\] 
and recall that $(X_0,Y_{0-})\sim\nu$. Then 
\begin{align}\label{eq:pbar}
\bar{\P}(\,\cdot\,)=\int_{\Sigma} \P_{0,x,y}(\,\cdot\,)\,\nu(\ud x, \ud y)\quad\text{and}\quad \bar{\E}[\,\cdot\,]=\int_{\Sigma} \E_{0,x,y}[\,\cdot\,]\,\nu(\ud x, \ud y).
\end{align}
The dynamics \eqref{eq:dynamics_mfg_01} conditional on the initial data $(t,x,y) \in [0,T]\times\Sigma$ read 
\begin{equation}\label{eq:dynamics_mfg_03}
\begin{split}
&X_{t+s}^{t,x} = x + \int_{0}^{s} a(X_{t+u}^{t,x}, m(t+u))\,\ud u +   \int_{0}^{s} \sigma(X_{t+u}^{t,x}) \ud W_{t+u},\\
&Y_{t+s}^{t,x,y;\,\xi} = y + (\xi_{t+s}-\xi_{t-}),\quad s \in [0,T-t], 
\end{split}
\end{equation}
where we notice for future reference that $\ud W_{t+u}=\ud (W_{t+u}-W_t)$.
Since the increments of the control $\xi\in\Xi^\Pi(Y_{0-})$, after time $t$, may in general depend on $(t,x,y)$, we account for that dependence by denoting $Y^{t,x,y;\,\xi}$ (and $\xi^{t,x,y}$ if necessary).
Instead, given a bounded measurable function $m$, the dynamics of $X$ only depends on the initial condition $X_t=x$, which motivates the use of the notation $X^{t,x}$. For the original case of the process started at time zero (i.e., $t=0$), we use the simpler notation $(X^x_s,Y^{x,y;\,\xi}_s)_{s\in[0,T]}$. It is worth noticing that the dynamics of $X$ also depends on the choice of the function $m$ appearing in the drift but we omit this dependence in our notation as it should cause no confusion in the rest of the paper.

Consistently with the notation introduced so far we will often use $\P_{t,x,y}(\,\cdot\,)=\bar \P(\,\cdot\,|X_t=x,Y_{t-}=y)$ for simplicity. For any $(\cF_{t+s})_{s\in[0,T-t]}$-stopping time $\tau\in[0,T-t]$ (i.e., such that $\{t+\tau\le s\}\in\cF_{t+s}$) and any bounded measurable function $g$ we have
\[
\bar\E\left[g(t+\tau,X^{t,x}_{t+\tau},Y^{t,x,y;\,\xi}_{t+\tau})\right]=\E_{t,x,y}\left[g(t+\tau,X_{t+\tau},Y^{\xi}_{t+\tau})\right],
\]
and, moreover, we use $\P_{x,y}=\P_{0,x,y}$ for the special case $t=0$. When no confusion shall arise we drop the subscript from $\P_{t,x,y}$ and $\E_{t,x,y}$ and simply use $\P$ and $\E$.

It is clear that given $\xi\in\Xi^\Pi(Y_{0-})$ the process $\hat \xi_{t+s}:=\xi_{t+s}-\xi_{t-}$ is right continuous, non-decreasing and adapted with $\hat \xi_{t-}=0$. Moreover, $y+\hat\xi\in[0,1]$, $\P_{t,x,y}$ a.s. (i.e., conditionally on $(X_t,Y^\xi_{t-})=(x,y)$) because $\xi\in\Xi^\Pi(Y_{0-})$. Motivated by this observation, for $t\in(0,T]$ it is useful to introduce the set 
\begin{equation}\label{eq:Xitxy}
\begin{aligned}
&\Xi^\Pi_{t,x}(y):= \big \{ \xi:\: \text{$\xi_{u}=0$ for each $u \in [0,t)$,}\\
&\qquad\qquad\qquad\:\text{$(\xi_{t+s})_{s\ge 0}$ is $(\mathcal F_{t+s})_{s\ge 0}$-adapted, non-decreasing, right-continuous,}\\
&\qquad\qquad\qquad\:\:\text{with $y+\xi_{t+s}\in [0,1]$ for all $s\in[0,T-t]$, $\P_{t,x,y}$-a.s.} \big \}.
\end{aligned}
\end{equation}
With a slight abuse of notation $\Xi^\Pi_{0,x}(y)=\Xi^\Pi_x(y)$. Here $\Pi$ is fixed, so we can drop the superscript in the definition of the set of admissible controls.

Assuming that the mapping $(x,y)\mapsto \E_{x,y}[Y^{\xi}_t]$ is measurable for any admissible $\xi$, we can express the consistency condition \eqref{eq:cons0} as
\[
m(t) =\int_\Sigma\E_{x,y}[Y^{\xi}_t]\nu(\ud x,\ud y) = \int_\Sigma \int_{\Sigma} y' \mu^{x,y;\,\xi}_{t}(\ud x', \ud y') \nu(\ud x,\ud y),
\]
where $\mu^{x,y;\,\xi}_t := \mathcal{L}(X^{x}_t, Y^{x,y;\,\xi}_t) \in \mathcal{P}(\Sigma)$ is the law of the pair $(X^{x}_t, Y^{x,y;\,\xi}_t)$ and the integral with respect to $\nu(\ud x,\ud y)$ accounts for the fact that $(X_0,Y_{0-})\overset{d}\sim\nu$.

Turning our attention to the optimisation problem, we have that the maximal expected payoff associated with a condition $(t,x,y) \in  [0,T]\times\Sigma$ is given by 
\begin{equation}\label{eq:P0}
\begin{split}
v(t,x,y) &:= \sup_{\xi \in \Xi_{t,x}(y)}\,J(t,x,y ; \xi)\quad\text{with}\\
J(t,x, y; \xi) &:= \E_{t,x,y}\left[\int_{0}^{T-t}\!\! e^{- r s} f(X_{t+s}, Y_{t+s}^{\xi})\,\ud s -\! \int_{[0,T-t]}  e^{-r s} c_0 \ud\xi_{t+s}\right].
\end{split}
\end{equation}
The function $v$ depends also on $m$ via the dynamics of $X$ in \eqref{eq:dynamics_mfg_03}, so we should write $v(t,x,y;m)$. However, we drop the dependence on $m$ as this should cause no confusion. We will show that for each $(x,y)\in\Sigma$ we can find an optimal control $\xi^*$, jointly measurable in $(\omega, x,y)$, for the problem above such that $(x,y)\mapsto J(0,x, y; \xi^*)$ is measurable. Then the initial objective function in \eqref{eq:J0} and the optimization problem in \eqref{eq:P0} are easily linked by averaging the latter over the initial condition $(X_0 , Y_{0-})\overset{d}{\sim}\nu \in \mathcal{P}(\Sigma)$. That is:
\begin{equation}\label{Vnu}
V^{\nu} = \sup_{\xi \in \Xi(Y_{0-})}\,J(\xi) = \sup_{\xi \in \Xi(Y_{0-})}\int_{\Sigma} J(0,x,y; \xi)\nu(\ud x, \ud y)=\int_{\Sigma} v(0,x,y)\nu(\ud x, \ud y).
\end{equation}
Indeed, on the one hand we have 
\begin{align*}
V^{\nu} &=  \sup_{\xi \in  \Xi(Y_{0-})} \int_{\Sigma} J(0, x, y; \xi)\nu(\ud x, \ud y) 
\leq  \int_{\Sigma} v(0, x, y)\,\nu(\ud x, \ud y) = \bar{\E}[v(0, X_0, Y_0)].
\end{align*}
On the other hand we have
\begin{align*}
\bar{\E}[v(0, X_0, Y_0)]& = \int_{\Sigma} J(0, x, y ; \xi^{*})\,\nu(\ud x, \ud y) \\
& = \int_{\Sigma} \E_{0, x, y}\left[\int_{0}^{T} f(X_s, Y_s^{\xi^{*}})\,\ud s - \int_{0}^{T} c_0\,\ud \xi^{*}_s\right]\,\nu(\ud x, \ud y)\\
& = \bar{\E} \left[\int_{0}^{T}  f(X_s, Y_s^{\xi^{*}})\,\ud s - \int_{0}^{T} c_0\,\ud \xi^{*}_s \right] \leq V^{\nu},
\end{align*}
where we used \eqref{eq:pbar} in the third equality and, for the final inequality, $\xi^{*}\in\Xi^\Pi_x(y)$ conditionally on $(X_0,Y_{0-})=(x,y)$ for all $(x,y)\in\Sigma$.

The problem formulation \eqref{eq:P0} can be interpreted as a game between a continuum of agents (see, e.g., \cite{lacker2019mean} and \cite{cardaliaguet2018mean}) indexed by $(x,y)\in\Sigma$. Each agent is assigned a ``type'' $(x,y)$ based on the realization of i.i.d.\ random variables $(X_0, Y_{0-})$. The game starts after the random assignment and $v(0, x, y)$ is the value of the optimization problem faced by an agent of ``type'' $(x,y)$ (i.e., the value of the game {\em ex-post}). Instead, the optimization problem specified by \eqref{eq:dynamics_mfg_01}-\eqref{eq:J0} is the one faced by any of the agents before the random assignment of $(X_0, Y_{0-})$. In this interpretation $V^\nu$ is the value of the game before the players know their ``type'' (i.e., the value of the game {\em ex-ante}).

Now we define solutions of the MFG of capacity expansion.
\begin{definition}[Solution of the MFG of capacity expansion]\label{def:MFGsol} 
A solution of the MFG of capacity expansion with initial condition $(X_0,Y_{0-})\sim\nu\in\mathcal P_2(\Sigma)$ is a pair $(m^* , \xi^*)$ with $m^* :[0,T]\rightarrow[0,1]$ a measurable function and $\xi^* \in \Xi(Y_{0-})$ such that:
\begin{enumerate}
\item[(i)] {\em(Optimality property)}. $\xi^*$ is optimal, i.e.,
\begin{equation}
J(\xi^* )=V^{\nu}=\sup_{\xi \in\Xi}\bar{\E}\left[\int_0^{T}e^{-rt}f(X^*_t,Y^\xi_t)\ud t-\int_{[0,T]} e^{-rt}c_0\ud \xi_t\right],
\nonumber
\end{equation}
where $(X^*,Y^\xi)$ is a solution of Eq.\ \eqref{eq:dynamics_mfg_01} associated to $(m^*,\xi)$.
\item[(ii)] {\em(Mean-field property)}. Letting $(X^* , Y^* )$ be the solution of Eq.\ \eqref{eq:dynamics_mfg_01} associated to $(m^*, \xi^*)$, the consistency condition holds, i.e., 
\[ m^* (t)=\int_{\Sigma}\E_{x,y}[Y^{*}_t]\nu(\ud x,\ud y),\] 
for each $t\in[0,T]$.
\end{enumerate}
We will say that a solution $\xi^*$ of the MFG is in \emph{feedback form} if we have $\xi^*_t=\eta(t,X^*,Y_{0-})$, $t\in[0,T]$, for some non-anticipative mapping 
\[
\eta:[0,T]\times C([0,T]; \mathbb R)\times[0,1] \to [0,1]
\]
(i.e., such that $\eta(t,X^*,Y_{0-})=\eta(t,(X^*_{s\wedge t})_{s\in[0,T]},Y_{0-})$).
\label{def:solMFG}
\end{definition}
We observe that the definition of MFG solution above mimics the structure of a Nash equilibrium (NE) in classical game theory. Indeed, for a NE we first need to compute the best response of each player while keeping the strategies of her competitors fixed, and then we obtain the equilibrium as a fixed point of the best response map. Likewise, the optimality condition (i) corresponds to computing the best response against a given behaviour of the population described by $m^*$; condition (ii) is a fixed point condition, stating that $m^*$ has to be consistent with the best response of the ``representative player".

\subsection{Assumptions and main result}\label{sub:assumption}
Before stating our main result regarding the existence and structure of the solution to the MFG, we list below the assumptions needed in our approach.

\begin{assumption}[Coefficients of the SDE]\label{ass:SDE}
For the functions $a:\Sigma\rightarrow\R$ and $\sigma:\R\rightarrow(0,\infty)$ the following holds:
\begin{itemize}
\item[(i)] $a$ and $\sigma$ are Lipschitz continuous with constant $L>0$, i.e., for all $x,x'\in\R$ and $m,m'\in[0,1]$, we have
\begin{equation*}
\left|a(x,m)-a(x',m')\right|+\left|\sigma(x)-\sigma(x')\right|\leq L(\left| x-x' \right|+\left| m- m' \right|).
\end{equation*}
\item[(ii)] The mapping $m\mapsto a(x,m)$ is non-decreasing on $[0,1]$ for all $x\in\mathbb R$. 
\end{itemize}
\end{assumption}
Part (i) of the assumption guarantees that given any Borel measurable function $m:[0,T]\to [0,1]$ the first equation in \eqref{eq:dynamics_mfg_03} admits a unique strong solution (see, e.g., \cite{karatzasShreve}, Theorem 5.2.9). Moreover, by a well-known application of Kolmogorov-Chentsov's continuity theorem, there exists a modification $\tilde X$ of $X$ which is continuous as a random field, i.e., $(t,x,s)\mapsto \tilde X^{t,x}_{t+s}$ is continuous $\P$-a.s.~(see, e.g., \cite{karatzasShreve}, pp.~397-398, or \cite{baldi}, Theorem 9.9). From now on we tacitly assume that we always work with such modification and we denote it again by $X$.

Part (i) of the assumption could be relaxed but at the cost of additional technicalities in the proofs. In principle we only need sufficient regularity on the coefficients to guarantee existence of a unique strong solution for $X$ which is also continuous with respect to its initial datum $(t,x)$. Part (ii) instead is instrumental in our construction of the optimal control in the MFG and will be used later for a comparison result (Lemma \ref{lem:comparisonTime}). Notice that (ii) is well-suited for the application to the {\em goodwill} problem described in Section \ref{sec:model} in the Introduction. Typical examples that we have in mind for the drift are $a(x,m)=(m-x)$ (mean-reverting), $a(x,m)=m x$ (geometric Brownian motion) and $a(x, m)=m$ (arithmetic Brownian motion).

Next we give assumptions on the running profit appearing in the optimisation problem and let $\Sigma'=\R\times(0,1]$.
\begin{assumption}[Profit function]\label{ass:f}
The running profit $f:\Sigma\rightarrow[0,\infty)$ is continuous and the partial derivatives $\partial_y f$ and $\partial_{xy} f$ exist and are continuous on $\Sigma'$. Furthermore, we have
\begin{itemize}
\item[(i)] Monotonicity: $x\mapsto f(x,y)$, $y\mapsto f(x,y)$ and $x\mapsto\partial_y f(x,y)$ are non-decreasing, with
\begin{equation} \label{ass:partial-y-f}
\lim_{x\rightarrow -\infty}\partial_yf(x,y)< rc_0 < \lim_{x\rightarrow +\infty}\partial_yf(x,y);
\end{equation}
\item[(ii)] Concavity: $y\mapsto f(x,y)$ is strictly concave for all $x\in\mathbb{R}$.
\item[(iii)] The mixed derivative is strictly positive, i.e., $\partial_{xy}f> 0$ on $\mathbb R\times(0,1)$.
\end{itemize} 
\end{assumption}
The set of assumptions above is in line with the literature on irreversible investment and is fulfilled for example by profit functions of Cobb-Douglas type (i.e., $f(x,y)=x^\alpha y^\beta$ with $\alpha\in[0,1]$, $\beta\in(0,1)$ and $x>0$); see examples in Section \ref{sec:approximation} and Remark \ref{rem:statespace}.\\

We conclude with some standard integrability conditions that guarantee that the problem is well-posed and will allow us to use the dominated convergence theorem in some of the technical steps in the proofs.
\begin{assumption}[Integrability]\label{ass:int}
There exists $p>1$ such that, given any Borel measurable $m:[0,T]\to [0,1]$ and letting $X$ be the associated solution of the SDE \eqref{eq:dynamics_mfg_03}, we have 
\begin{eqnarray*}
\E_{t,x,y}\left[\int_0^{T-t}e^{-rs}\Big(\left|f(X_{t+s},y)\right|^{p}+\left|\partial_yf(X_{t+s},y)\right|^{p}\Big)\ud s\right]<\infty,
\end{eqnarray*}
for all $(t,x,y)\in[0,T]\!\times\! \Sigma'$. Finally, $\nu\in\mathcal P_{2}(\Sigma)$.
\end{assumption}
\begin{remark}[State space]\label{rem:statespace}
For specific applications it may be convenient to restrict the state space of the process $X$ to the positive half-line $[0,\infty)$ or to a generic (possibly unbounded) interval $(\underline x,\overline x)$. In those cases the assumptions above and the further ones we will make in the next sections can be adapted in a straightforward manner. In particular the limits in \eqref{ass:partial-y-f} are amended by letting $x$ tend to the endpoints of the relevant domain. If the end-points of the domain are inaccessible to the process $X$ all our arguments of proof continue to hold up to trivial changes in the notation. For a more general boundary behaviour of the process some tweaks may be needed on a case by cases basis.
\end{remark}

We are now ready to state the main results concerning the MFG described above. The proof requires a number of technical steps and hinges on a iterative method whose details are provided in Section \ref{sec:existence}.

\begin{theorem}[Solution of the MFG of capacity expansion]\label{teo:existenceSolMFG} Suppose Assumptions \ref{ass:SDE}, \ref{ass:f} and \ref{ass:int} hold. Then, there exists a upper-semi continuous function $c: [0,T] \times \R \to [0,1]$, with $t\mapsto c(t,x)$ and $x\mapsto c(t,x)$ both non-decreasing, such that the pair $(m^*, \xi^*)$ with
\[ \xi_t ^* := \sup_{0\le s\le t} (c(s,X^* _s)-Y_{0-})^+, \quad m^* (t) := \int_\Sigma \E_{x,y} \left[Y_t ^* \right] \nu(\ud x,\ud y), \quad t \in [0,T],\]
is a solution of the MFG as in Definition \ref{def:solMFG}.
\end{theorem}

Notice that the iterative scheme that we devise for the proof of the theorem suggests a procedure to actually construct the optimal boundary numerically.

The second key result in this paper shows that the optimal control $\xi^*$ solution of the MFG can be used (under mild additional assumptions) to construct an $\eps$-Nash equilibrium in the $N$-player game. The statement and proof of this fact are given in Section \ref{sec:approximation} below, whereas in the next section we prove Theorem \ref{teo:existenceSolMFG}. 

\section{Construction of the solutions to the MFG}\label{sec:existence}
In this section, we provide the complete proof of Theorem \ref{teo:existenceSolMFG} together with an intuitive description of the iterative scheme that underpins it. Some of the auxiliary results used along the way can be found in the Appendix as indicated.

\subsection{Description of the iterative scheme}
The idea is to start an iterative scheme based on singular control problems that are analogue to the one in the MFG but without consistency condition in the mean-field interaction. 

We initialise the scheme by setting $m^{[-1]}(t)\equiv 1$, for $t \in [0,T]$. At the $n$-th step, $n \ge 0$, assume a non-decreasing, right-continuous function $m^{[n-1]}:[0,T]\to [0,1]$ is given and fixed and consider the dynamics
\begin{equation}\label{eq:Xn}
\begin{aligned}
X_{t+s}^{[n];t,x}&= x+\int_0^{s} a(X_{t+u}^{[n];t,x},m^{[n-1]}(t+u))\ud u+\int_0^{s}\sigma(X_{t+u}^{[n];t,x}) \ud W_{t+u},\\
Y_{t+s}^{[n];t,x,y}&=y+ \xi_{t+s},
\end{aligned}
\end{equation}
for $(x,y)\in\Sigma$, $s\in[0,T-t]$, $t\in[0,T]$ and where $\xi\in \Xi_{t,x}(y)$ (cf.\ \eqref{eq:Xitxy}).  We define the singular control problem $\textbf{SC}^{[n]}_{t,x,y}$ as: 
\begin{align}
v_n(t,x,y):= & \sup_{\xi\in\Xi_{t,x}(y)}J_n(t,x,y;\xi)\qquad\text{with}\label{SCn-1}\\
J_n(t,x,y;\xi):= &\,\E_{t,x,y}\left[\int_0^{T-t}e^{-rs}f(X^{[n]}_{t+s},y + \xi_{t+s})\ud s-\int_{[0,T-t]}e^{-rs}c_0\ud\xi_{t+s}\right].\label{SCn-2}
\end{align}
Now, in order to define the $(n+1)$-th step of the algorithm, let us assume that we can find an optimal control $\xi^{[n]*}$ for problem $\textbf{SC}^{[n]}_{0,x,y}$ for each $(x,y)\in\Sigma$. Set $Y^{[n]*}:=y+\xi^{[n]*}$ and assume that $(x,y)\mapsto\E_{x,y}\big[ Y^{[n]*}_t\big]$ is measurable for all $t\in[0,T]$. Then, we define
\[
m^{[n]}(t):=\int_\Sigma\E_{x,y}\left[Y^{[n]*}_t\right]\nu(\ud x,\ud y).
\] 
The map $t\mapsto m^{[n]}(t)$ is non-decreasing and right-continuous (by dominated convergence) with values in $[0,1]$, so we can use it to define $(X^{[n+1]}, Y^{[n+1]})$ and $v_{n+1}$ by iterating the above construction. 

It is well-known in singular control theory that since $y\mapsto f(x,y)$ is concave and the dynamics of $X^{[n]}$ is independent of the control $\xi$, then the $y$-derivative of $v_n(t,x,y)$ corresponds to the value function of an optimal stopping problem. While we will re-derive this fact in Proposition \ref{th:theoremConnection} for completeness, here we state the optimal stopping problem that should be associated to $\textbf{SC}^{[n]}_{t,x,y}$ above. For $(t,x,y)\in[0,T]\times\Sigma'$
we define the stopping problem $\textbf{OS}^{[n]}_{t,x,y}$ as
\begin{align}
u_n(t,x,y):= & \inf_{\tau\in\T_t}U_n(t,x,y;\tau)\qquad\text{with}\label{eq:un}\\
U_n(t,x,y;\tau):= &\,\E_{t,x}\left[\int_0^{\tau}e^{-rs}\partial_yf(X_{t+s}^{[n]},y)\ud s+c_0e^{-r\tau}\right],\quad\text{for $\tau\in\T_t$}\label{eq:Un}
\end{align}
and where $\T_t$ is the set of stopping times for the filtration generated by the Brownian motion $(W_{t+s}-W_t)_{s\ge 0}$ appearing in \eqref{eq:Xn} with values in $[0,T-t]$ (in particular $\{t+\tau\le s\}\in\cF_{t+s}$, since $(\cF_{t+s})_{s\in[0,T-t]}$ is an even larger filtration).  Due to Assumption \ref{ass:f} it may be $\partial_yf(x,0)=+\infty$ (e.g., for Cobb-Douglas profit). In that case, by convention $\tau=0$ is an optimal stopping time and $u_n(t,x,0)=c_0$. With this convention $u_n$ is well-defined on $[0,T]\times\Sigma$. 

The stopping problem above is standard (see, e.g., \cite{peskirShyriyaev}, Chapter I, Section 2, Theorem 2.2): thanks to Assumption \ref{ass:int} and continuity of the gain process
\[
u\mapsto \int_0^{u}e^{-rs}\partial_yf(X_{t+s}^{[n]},y)\ud s+c_0e^{-ru}
\]
we know that the smallest optimal stopping time is
\begin{align}\label{tau*n}
\tau^{[n]}_*(t,x,y)=\inf\{s\in[0,T-t]: u_n(t+s,X^{[n];t,x}_{t+s},y)=c_0\}.
\end{align}
Letting
\begin{align}
Z^{[n]}_s:=e^{-rs}u_n(t+s,X^{[n]}_{t+s},y)+\int_0^se^{-ru}\partial_y f(X^{[n]}_{t+u},y)\ud u
\end{align}
we have that, under $\P_{t,x,y}$, 
\begin{align}\label{eq:mart}
\text{$(Z^{[n]}_s)_{s\in[0,T-t]}$ is a submartingale and $\big(Z^{[n]}_{s\wedge\tau^{[n]}_*}\big)_{s\in[0,T-t]}$ is a martingale}.
\end{align}
Accordingly, we define the continuation region, $\mathcal{C}^{[n]}$, and the stopping region, $\mathcal{S}^{[n]}$, of the optimal stopping problem as
\begin{align*} 
\mathcal{C}^{[n]}:= &\lbrace (t,x,y) \in [0,T] \times \Sigma\,:\,u_n(t,x,y)<c_0 \rbrace,\\
\mathcal{S}^{[n]}:= &\lbrace (t,x,y) \in [0,T] \times \Sigma\,:\,u_n(t,x,y)=c_0 \rbrace.
\end{align*}
Finally, we introduce an auxiliary set which will be used in our analysis:
\begin{equation}
\mathcal{H}:=\lbrace(x,y)\in\R\times[0,1]:\,\partial_yf(x,y)-rc_0<0\rbrace. \label{def-H}
\end{equation}
Notice that condition \eqref{ass:partial-y-f} in Assumption \ref{ass:f} implies that $\mathcal H$ is not empty. This will be needed to prove that the continuation and stopping regions are not empty either.

Since $W_{t+u}-W_{t}= W_u$ in law, it is possible (and convenient for some steps in the analysis of the stopping problems) to use always the same Brownian motion in the dynamics of the process $X^{[n];t,x}$, irrespectively of $t\in[0,T]$. More formally, on a probability space $(\widehat \Omega,\widehat\cF,\widehat \P)$ let us consider a Brownian motion $(\widehat W_t)_{t\ge 0}$ with its natural filtration $(\widehat\cF_t)_{t\ge 0}$. For each $n\ge 0$ we can define a family of processes $\widehat X^{[n];t,x}$ parametrized by couples $(t,x)\in[0,T]\times\R$ so that for each $(t,x)$ we have
\begin{equation}\label{eq:hatXn}
\begin{aligned}
\widehat X_{t+s}^{[n];t,x}&= x+\int_0^{s} a(\widehat X_{t+u}^{[n];t,x},m^{[n-1]}(t+u))\ud u+\int_0^{s}\sigma(\widehat X^{[n];t,x}_{t+u}) \ud \widehat{W}_{u}.
\end{aligned}
\end{equation}
Clearly $\widehat X^{[n];t,x}_{t+s}$ is $\widehat \cF_{s}$-measurable, irrespectively of $(t,x)$, and
\[
\text{Law}_{\bar{\P}}\big((X^{[n];t,x}_{t+s})_{s\in[0,T-t]}\big)=\text{Law}_{\widehat\P}\big((\widehat{X}^{[n];t,x}_{t+s})_{s\in[0,T-t]}\big),
\]
since our SDEs for $X^{[n]}$ and $\widehat X^{[n]}$ admit a unique strong solution.

Let $\widehat{\T}_t$ be the class of $(\widehat \cF_s)_{s\ge 0}$-stopping times bounded by $T-t$. As explained above, one only needs to look for optimal stopping times in the form of \eqref{tau*n}. Then the equality in law of $X^{[n];t,x}$ and $\widehat X^{[n];t,x}$ guarantees that 
\begin{equation}
u_n(t,x,y)= \inf_{\tau\in\widehat{\T}_t}\widehat{U}_n(t,x,y;\tau),\label{eq:hatun}
\end{equation}
where
\[
\widehat{U}_n(t,x,y;\tau):= \widehat{\E}_{t,x}\left[\int_0^{\tau}e^{-rs}\partial_yf(\widehat{X}_{t+s}^{[n]},y)\ud s+c_0e^{-r\tau}\right].
\]
The representation \eqref{eq:hatun} for the value function of our stopping problem is convenient because $\widehat{\T}_{t_2}\subset \widehat{\T}_{t_1}$ for $t_1<t_2$. This fact and Lemma \ref{lem:comparisonTime} below will be used to prove monotonicity of the mapping $t\mapsto u_n(t,x,y)$ in Proposition \ref{prop:OSvalueFunc}--(iii).

The rest of our algorithm of proof for Theorem \ref{teo:existenceSolMFG} goes as follows:
\begin{itemize}
\item[{\em Step 1}.] Using a probabilistic approach we study in detail continuity and monotonicity of the value function $u_n$, for a generic $n\ge 0$. 
\item[{\em Step 2}.] Thanks to the results in \textit{Step 1} we construct a (unique) solution to $\textbf{OS}^{[n]}_{t,x,y}$ by determining the geometry of the stopping region $\cS^{[n]}$. In particular we need to prove regularity properties of the optimal stopping boundary $\partial\cC^{[n]}$ that guarantee that we can construct a process $Y^{[n]*}$ so that the state $(t,X^{[n]},Y^{[n]*})$ is bound to evolve in the closure $\overline \cC^{[n]}$ of the continuation set, by Skorokhod reflection.
\item[{\em Step 3}.] We confirm that $Y^{[n]*}$ is the unique optimal control in problem $\textbf{SC}^{[n]}_{t,x,y}$ and that $v_n$ can be constructed by integrating $u_n$ with respect to $y$ (as already shown in the existing literature).
\item[{\em Step 4}.] We prove that the sequence $(u_n)_{n\ge 0}$ is decreasing and use this fact to prove that the iterative scheme converges to the MFG, in the sense that $(X^{[n]},Y^{[n]*},m^{[n]})$ converges to $(X^*,Y^*,m^*)$ from Definition \ref{def:solMFG} and that $(Y^*,m^*)$ are expressed as in Theorem \ref{teo:existenceSolMFG}.
\end{itemize}

\subsection{Solution of the $n$-th stopping problem}\label{subsec:OCOS}
Here we construct the solution to problem $\textbf{OS}^{[n]}_{t,x,y}$ for a generic $n\ge 0$. In particular, $t \mapsto m^{[n-1]}(t)$ is a given right-continuous, non-decreasing function bounded between zero and one. 
First we state a simple but useful comparison result.
\begin{lemma}[Comparison]\label{lem:comparisonTime} Let Assumption \ref{ass:SDE} hold and recall that $m^{[n-1]}:\,[0,T]\rightarrow[0,1]$ is non-decreasing. Then, for any $t\leq t'$ we have
\begin{equation}\label{eq:comp} 
\widehat\P\left(\widehat{X}^{[n];t,x}_{t+s}\leq\widehat{X}^{[n];t',x}_{t'+s},\:\: \forall s\in[0,T-t']\right)=1,
\end{equation}
under the dynamics in \eqref{eq:hatXn}.
\end{lemma}
\begin{proof}
Since the SDEs for $\widehat{X}^{[n];t,x}$ and $\widehat{X}^{[n];t',x}$ are driven by the same Brownian motion, it suffices to compare their drift coefficients and then apply the comparison result in \cite[Proposition 5.2.18]{karatzasShreve} --- although it is mentioned in its statement, the proof of Proposition 5.2.18 does not use time-continuity of the drift. 

Set $A(x,s):= a(x,m(t+s))$ and $A'(x,s):= a(x,m(t'+s))$. Since both $t\mapsto m(t)$ and $m\mapsto a(x,m)$ are non-decreasing (Assumption \ref{ass:SDE}-(ii)), we have  $A(x,s)\leq A'(x,s)$ for all $(x,s)\in\R\times[0,T-t']$. Therefore, applying \cite[Proposition 5.2.18]{karatzasShreve} we obtain \eqref{eq:comp}. 
\end{proof}
Next we prove continuity and monotonicity of the value function. Recall our convention that $u_n(t,x,0)=c_0$ if $\partial_y f(x,0)=\infty$, and the set $\Sigma'=\R\times(0,1]$.
\begin{proposition}[Value function of $\textbf{OS}^{[n]}$ ]\label{prop:OSvalueFunc} 
Let Assumptions \ref{ass:SDE}--\ref{ass:int} hold. Then the value function of the optimal stopping problem {\em $\textbf{OS}^{[n]}_{t,x,y}$} has the following properties:
\begin{itemize}
\item[(i)] $0 \le u_n(t,x,y)\leq c_0$;
\item[(ii)] the map $x\mapsto u_n(t,x,y)$ is non-decreasing for each fixed $(t,y)\in[0,T]\times[0,1]$ and $y\mapsto u_n(t,x,y)$ is non-increasing for each $(t,x)\in[0,T]\times\R$;
\item[(iii)] the map $t\mapsto u_n(t,x,y)$ is non-decreasing for each fixed $(x,y)\in\Sigma$;
\item[(iv)] the value function is continuous, i.e., $u_n\in C([0,T]\times\Sigma')$.
\end{itemize}
\end{proposition}
\begin{proof}
{\em (i)}. The upper bound is due to $u_n(t,x,y)\le U_n(t,x,y;0)=c_0$. For the lower bound it is enough to recall that $\partial_y f\ge 0$ by Assumption \ref{ass:f}-(i).
\vspace{+4pt}
 
{\em (ii)}. Fix $(t,y)\in[0,T]\times[0,1]$. Let $x_2>x_1$ and set $\tau_2:= \tau^{[n]}_*(t,x_2,y)$  as in \eqref{tau*n}, which is optimal in $u_n(t,x_2,y)$. Then
\begin{equation}
u_n(t,x_2,y)-u_n(t,x_1,y)\geq \E\left[\int_0^{\tau_2}e^{-rs}\left(\partial_yf(X^{[n];t,x_2}_{t+s},y)-\partial_yf(X^{[n];t,x_1}_{t+s},y)\right)\ud s\right]\geq 0
\nonumber
\end{equation}
because $X^{[n];t,x_2}_{t+s}\geq X^{[n];t,x_1}_{t+s}$ by uniqueness of the solution to \eqref{eq:Xn} and $x\mapsto\partial_y f(x,y)$ is increasing by Assumption \ref{ass:f}-(i). By a similar argument we also obtain monotonicity in $y$, since $y\mapsto\partial_y f(x,y)$ is decreasing by Assumption \ref{ass:f}-(ii).
\vspace{+4pt}

{\em (iii)}. For this part of the proof we use Lemma \ref{lem:comparisonTime} and the representation \eqref{eq:hatun} of the value function. Fix $(x,y)\in\Sigma$ and take $t_2>t_1$ in $[0,T]$. Let 
\[
\tau_2= \widehat{\tau}^{[n]}_*(t_2,x,y)=\inf\{s\in[0,T-t_2]:u_n(t_2+s,\widehat{X}^{[n];t_2,x}_{t_2+s},y)=c_0\},
\] 
which is optimal in $u_n(t_2,x,y)$, and notice that the stopping time is also admissible for $u_n(t_1,x,y)$ because $\widehat{\T}_{t_2}\subset\widehat{\T}_{t_1}$. Then
\begin{equation}
u_n(t_2,x,y)-u_n(t_1,x,y)\geq \widehat\E\left[\int_0^{\tau_2}e^{-rs}\left(\partial_yf(\widehat{X}^{[n];t_2,x}_{t_2+s},y)-\partial_yf(\widehat{X}^{[n];t_1,x}_{t_1+s},y)\right)\ud s\right]\geq 0,
\nonumber
\end{equation}
where the final inequality uses that $\widehat{X}^{[n];t_2,x}_{t_2+s}\geq \widehat{X}^{[n];t_1,x}_{t_1+s}$ for $s\in[0,T-t_2]$, $\P$-a.s.~by Lemma \ref{lem:comparisonTime} and $x\mapsto \partial_yf(x,y)$ is non-decreasing by Assumption \ref{ass:f}-(i).
\vspace{+4pt}

{\em (iv)}. Joint continuity of the value function can be deduced by separate continuity in each variable and monotonicity (see, e.g., \cite{kruse}). Thanks to (ii) and (iii), it suffices to show that $u_n$ is continuous separately in each variable.

Fix $(t,x,y)\in[0,T]\times\Sigma'$. Let $x_k\to x$ as $k\to \infty$ and let $\tau_*=\tau^{[n]}_*(t,x,y)$ be optimal for $u_n(t,x,y)$. First we show right-continuity of $u_n(t,\,\cdot\,,y)$ and assume that $x_k\downarrow x$. For each $k$, using monotonicity proven in (ii) we have
\begin{align}
0\leq &u_n(t,x_k,y)-u_n(t,x,y)\label{RHS}\\
\leq &\E\left[\int_0^{\tau_*}e^{-rs}\left(\partial_yf(X^{[n];t,x_k}_{t+s},y)-\partial_yf(X^{[n];t,x}_{t+s},y)\right)\ud s\right]\nonumber\\
\leq &\E\left[\int_0^{T-t}e^{-rs}\left|\partial_yf(X^{[n];t,x_k}_{t+s},y)-\partial_yf(X^{[n];t,x}_{t+s},y)\right|\ud s\right].\nonumber
\end{align}
Taking limits as $k\to\infty$, Assumption \ref{ass:int} allows us to use dominated convergence so that we only need
\[ 
\lim_{k\to\infty}\left|\partial_yf(X^{[n];t,x_k}_{t+s},y)-\partial_yf(X^{[n];t,x}_{t+s},y)\right|= 0,\quad\P-a.s. 
\]
The latter holds by continuity of $\partial_y f$ and continuity of the flow $x\mapsto X^{[n];t,x}$ (which is guaranteed by Assumption \ref{ass:SDE}).

We can prove left-continuity by analogous arguments. Letting $x_k\uparrow x$ and, for each $k$, selecting the stopping time $\tau_k=\tau^{[n]}_*(t,x_k,y)$ which is optimal for $u_n(t,x_k,y)$ we get
\begin{align*}
0\leq &u_n(t,x,y)-u_n(t,x_k,y)\nonumber\\
\leq &\E\left[\int_0^{\tau_k}e^{-rs}\left(\partial_yf(X^{[n];t,x}_{t+s},y)-\partial_yf(X^{[n];t,x_k}_{t+s},y)\right)\ud s\right].
\end{align*}
Then we can conclude as in \eqref{RHS}. Completely analogous arguments allow to prove continuity of the value function with respect to $y$ and we omit them here for brevity.

Continuity in time only requires a small adjustment to the argument above. Let $t_k\to t$ as $k\to\infty$, with $(t,x,y)\in[0,T]\times\Sigma'$ fixed. First let us consider $t_k\downarrow t$ and set $\tau_*=\tau^{[n]}_*(t,x,y)$, which is optimal for $u_n(t,x,y)$. Then $\tau_*\wedge(T-t_k)$ is admissible for $u_n(t_k,x,y)$ and, by the monotonicity proven in (iii), we have
\begin{align}
0 \leq& u(t_k,x,y)-u(t,x,y)\nonumber\\
 \leq &\E\left[\int_0^{\tau_*\wedge(T-t_k)}e^{-rs}\left(\partial_yf(X^{[n];t_k,x}_{t_k+s},y)-\partial_yf(X^{[n];t,x}_{t+s},y)\right)\ud s\right]\nonumber\\
& \quad +\E\left[\int_{\tau_*\wedge(T-t_k)}^{\tau_*}e^{-rs}\partial_yf(X^{[n];t,x}_{t+s},y)\ud s\right]\nonumber\\
\leq& \E\left[\int_0^{T-t_k}e^{-rs}\left|\partial_yf(X^{[n];t_k,x}_{t_k+s},y)-\partial_yf(X^{[n];t,x}_{t+s},y)\right|\ud s\right] \nonumber\\
& \quad +\E\left[\int_{T-t_k}^{T-t}e^{-rs}\left|\partial_yf(X^{[n];t,x}_{t+s},y)\right|\ud s\right].
\end{align}
Now we can let $k\to\infty$ and use dominated convergence (thanks to Assumption \ref{ass:int}), continuity of the stochastic flow $t\mapsto X^{t,x}_{t+\cdot}$ and continuity of $\partial_yf$  (Assumption \ref{ass:f}) to obtain right-continuity of $u_n(\,\cdot\,,x,y)$. An analogous argument allows to prove left-continuity as well.
\end{proof}

Continuity of $u_n$ on $[0,T]\times\Sigma$ is immediate if $\partial_y f$ is continuous on $\Sigma$. If $\partial_y f$ is not well-defined for $y=0$, one can still prove continuity of $u_n$ on $[0,T]\times\Sigma$ in some cases, upon further specifying the exact form of $f$. We leave further details aside as they are not needed in the rest of our analysis.

Thanks to the properties of the value function we can easily determine the shape of the continuation region $\cC^{[n]}$, whose boundary $\partial\cC^{[n]}$ turns out to be a surface with `nice' monotonicity properties, that we will subsequently use to obtain a solution of the singular control problem $\textbf{SC}^{[n]}$. Part of the proof is based on the following equivalent representation of the value function:
\begin{equation}
u_n(t,x,y)=c_0+ \inf_{\tau\in\T_t}\E_{t,x}\left[\int_0^{\tau}e^{-rs}\left(\partial_yf(X^{[n]}_{t+s},y)-rc_0\right)\ud s\right].
\label{eq:valueFuncRepresetation2}
\end{equation}
\begin{proposition}[Optimal boundary]\label{prop:OSboundary} 
Under Assumptions \ref{ass:SDE}--\ref{ass:int}, the continuation and stopping regions, $\cC^{[n]}$ and $\mathcal{S}^{[n]}$, are non-empty. The boundary of $\cC^{[n]}$ can be expressed as a function $c_n:[0,T]\times \R\rightarrow[0,1]$, such that
\begin{align*}
&\cC^{[n]}=\{(t,x,y)\!\in\![0,T]\!\times\!\Sigma: y>c_n(t,x) \},
&\cS^{[n]}=\{(t,x,y)\!\in\![0,T]\!\times\!\Sigma: y\le c_n(t,x) \}.
\end{align*}
The map $(t,x)\mapsto c_n(t,x)$ is upper semi-continuous with $t\mapsto c_n(t,x)$ and $x\mapsto c_n(t,x)$ non-decreasing (hence $c_n(\,\cdot\,,x)$ and $c_n(t,\,\cdot\,)$ are right-continuous).
\end{proposition}
\begin{proof}
Thanks to (ii) in Proposition \ref{prop:OSvalueFunc}, for any $(t,x)\in[0,T]\times \R$ we can define
\begin{align}\label{def:cn}
c_n(t,x):=\inf\{y\in[0,1]:u_n(t,x,y)<c_0\}=\inf\{y\in[0,1]:(t,x,y)\in\cC^{[n]}\}
\end{align}
with the convention that $\inf\varnothing=1$. Since $x\mapsto u_n(t,x,y)$ and $t\mapsto u_n(t,x,y)$ are non-decreasing we have, for any $\eps>0$
\[
(t,x,y)\in\cS^{[n]}\implies (t,x+\eps,y)\in\cS^{[n]}
\]
and 
\[
(t,x,y)\in\cS^{[n]}\implies (t+\eps,x,y)\in\cS^{[n]}.
\]
Then, $c_n$ is non-decreasing in both $t$ and $x$. 

To show upper semi-continuity we fix $(t,x)$ and take a sequence $(t_k,x_k)_{k\ge 1}$ that converges to $(t,x)$. Then $(t_k,x_k,c_n(t_k,x_k))\in\cS^{[n]}$ for all $k$'s and, since the stopping region is closed, in the limit we get
\[
\limsup_{k\to\infty}\,(t_k,x_k,c_n(t_k,x_k))=(t,x,\limsup_{k\to\infty}c_n(t_k,x_k))\in\cS^{[n]}.
\]
Then, by definition of $c_n$ it must be
\[
\limsup_{k\to\infty}c_n(t_k,x_k)\le c_n(t,x).
\]
It only remains to show that $\cC^{[n]}$ and $\cS^{[n]}$ are both non-empty. A standard argument implies that $[0,T)\times\mathcal{H}\subset \cC^{[n]}$ with $\mathcal H$ the open set in \eqref{def-H}. Indeed, starting from $(t,x,y)\in[0,T)\times \mathcal H$ and taking the suboptimal strategy 
\[
\tau_{\mathcal{H}}:=\inf\{s\in[0,T-t]:(X^{[n];t,x}_{t+s},y)\notin\mathcal H\}
\]
we easily obtain $u_n(t,x,y)\le U_n(t,x,y;\tau_{\mathcal H})<c_0$ by continuity of paths of $X^{[n]}$ and since $\P_{t,x,y}(\tau_{\mathcal H}>0)=1$. So $\cC^{[n]}\neq\varnothing$ because $\mathcal H\neq\varnothing$ thanks to \eqref{ass:partial-y-f} in Assumption \ref{ass:f}. We conclude with an argument by contradiction. Assume that $\mathcal{S}^{[n]}=\varnothing$. Then, given any $(t,x,y)\in[0,T)\times\Sigma$ we have
\begin{equation}
u_n(t,x,y)=c_0+\E\left[\int_0^{T-t}e^{-rs}\left(\partial_yf(X^{[n];t,x}_{t+s},y)-rc_0\right)\ud s\right],
\nonumber
\end{equation}
thanks to \eqref{eq:valueFuncRepresetation2}. Taking limits as $x\rightarrow\infty$ and using monotone convergence to pass it under the expectation and the integral (Assumption \ref{ass:f}-(i)) we get
\begin{equation}
\lim_{x\rightarrow\infty}u_n(t,x,y)-c_0=\E\left[\int_0^{T-t}e^{-rs}\left(\lim_{x\rightarrow\infty}\partial_yf(X^{[n];t,x}_{t+s},y)-rc_0\right)\ud s\right]>0
\nonumber
\end{equation}
thanks to \eqref{ass:partial-y-f}. This contradicts $u_n(t,x,y)\leq c_0$, hence $\cS\neq\varnothing$.
\end{proof}
The optimal boundary is unique and thus it unambiguously characterises the minimal optimal stopping time $\tau_*$. It will also be shown in the proof of Lemma \ref{lem:ST} (see Eq. \eqref{usc}) that $\tau_*$ is in fact the largest optimal stopping time, hence it is unique.

\subsection{Solution of the $n$-th singular control problem}\label{sec:nSC} 
Here we follow a well-trodden path to show that the boundary $c_n$ obtained in the section above is actually all we need to construct the optimal control in the singular control problem $\textbf{SC}^{[n]}$. First we provide the candidate optimal control in the next lemma.
\begin{lemma}\label{lem:SK}
Fix $(t,x,y)\in[0,T]\times\Sigma$ and let $\xi^{[n]*}$ be defined $\P_{t,x,y}$-almost surely as
\[
\xi^{[n]*}_{t+s}:=\sup_{0\le u\le s}\left(c_n(t+u,X^{[n]}_{t+u})-y\right)^+\quad\text{with}\quad \xi^{[n]*}_{u}=0,\: u\in[0,t).
\]
Then, $\xi^{[n]*}\in\Xi_{t,x}(y)$ and realises $\P_{t,x,y}$-almost surely the Skorokhod reflection of the process $(t+s, X_{t+s} ^{[n]},Y_{t+s} ^{[n]*})_{s\in [0,T-t]}$ inside the continuation region $\cC^{[n]}$, where $Y^{[n]*}=y+\xi^{[n]*}$. That is, $\P_{t,x,y}$-almost surely we have
\begin{itemize}
\item[(i)] $(t+s, X^{[n]}_{t+s},Y^{[n]*}_{t+s})\in\overline\cC^{[n]}$ for all $s\in[0,T-t]$ (recall that $\overline\cC^{[n]}$ is the closure of $\cC^{[n]}$);
\item[(ii)] Minimality condition:
\begin{align}\label{eq:minSK}
\int_{[t,T]}\mathbf{1}_{\{Y^{[n]*}_{s-}>c_n(s,X^{[n]}_s)\}}\ud\xi^{[n]*}_{s}=\sum_{t<s\le T}\int_{Y^{[n]*}_{s-}}^{Y^{[n]*}_{s}}\mathbf{1}_{\{Y^{[n]*}_{s-}+z>c_n(s,X^{[n]}_s)\}}\ud z=0.
\end{align}
\end{itemize}
\end{lemma}
\begin{proof}
Clearly $\xi^{[n]*}$ is non-decreasing, adapted and bounded by $1-y$. So if we prove that it is also right-continuous we have shown that it belongs to $\Xi_{t,x}(y)$. The proof of right-continuity uses ideas as in \cite{de2017optimal}. For any $\eps>0$ we have
\[
\xi^{[n]*}_{t+s}\le \xi^{[n]*}_{t+s+\eps}=\xi^{[n]*}_{t+s}\vee\sup_{0< u\le \eps}\left(c_n(t+s+u,X^{[n]}_{t+s+u})-y\right)^+.
\]
By upper semi-continuity of the boundary and continuity of the trajectories of $X^{[n]}$ we have
\begin{align*}
&\lim_{\eps\to 0}\sup_{0< u\le \eps}\left(c_n(t+s+u,X^{[n]}_{t+s+u})-y\right)^+\\
&=\limsup_{u\to 0}\left(c_n(t+s+u,X^{[n]}_{t+s+u})-y\right)^+\le\left(c_n(t+s,X^{[n]}_{t+s})-y\right)^+ \leq \xi^{[n]*}_{t+s}
\end{align*}
Then, combining the above expressions we get $\xi^{[n]*}_{t+s}=\lim_{\eps\to 0} \xi^{[n]*}_{t+s+\eps}$ as needed.

Next we show the Skorokhod reflection property. By construction we have 
\[
Y^{[n]*}_{t+s}=y+\xi^{[n]*}_{t+s}\ge c_n(t+s,X^{[n]}_{t+s})
\]
so that $(t+s, X^{[n]}_{t+s},Y^{[n]*}_{t+s})\in\overline\cC^{[n]}$ for all $s\in[0,T-t]$ as claimed in (i). For the minimality condition (ii) fix $\omega\in\Omega$ and let $s\in[t,T]$ be such that $Y^{[n]*}_{s-}(\omega)>c_n\big(s,X^{[n]}_s(\omega)\big)$. Then by definition of $Y^{[n]*}$ we have
\begin{align}\label{eq:SKpr}
\sup_{t\le u <s}\left(c_n\big(u,X^{[n]}_u(\omega)\big)-y\right)^+> c_n\big(s,X^{[n]}_s(\omega)\big)-y,
\end{align}
which implies $Y^{[n]*}_{s-}(\omega)=Y^{[n]*}_{s}(\omega)$. The latter and \eqref{eq:SKpr} imply that there exists $\delta>0$ such that  
\begin{align}\label{eq:sup}
\left(c_n\big(s,X^{[n]}_s(\omega)\big)-y\right)^+\le \sup_{t\le u \le s}\left(c_n\big(u,X^{[n]}_u(\omega)\big)-y\right)^+-\delta.
\end{align}
By upper semi-continuity of $s\mapsto c_n\big(s,X^{[n]}_s(\omega)\big)$ there must exist $s'>s$ such that 
\[
\left(c_n\big(u,X^{[n]}_u(\omega)\big)-y\right)^+ \le \left(c_n\big(s,X^{[n]}_s(\omega)\big)-y\right)^++\frac{\delta}{2}
\]
for all $u\in[s,s')$. The latter and \eqref{eq:sup} imply $Y^{[n]*}_{s-}(\omega)=Y^{[n]*}_{u}(\omega)$ for all $u\in[s,s')$. Hence $\ud \xi^{[n]*}(\omega)=0$ on $[s,s')$ as needed to show that the first term in \eqref{eq:minSK} is zero. For the second term, it is enough to notice that by the explicit form of $\xi^{[n]*}$ we easily derive $\{\Delta\xi^{[n]*}_s>0\}=\{Y^{[n]*}_{s-}< c(s,X^{[n]}_s)\}$ for any $s\in[t,T]$. Therefore
\begin{align*}
Y^{[n]*}_{s-}+\Delta\xi^{[n]*}_s=&Y^{[n]*}_{s-}+\xi^{[n]*}_{s-}\vee\left(c_n(s,X^{[n]}_s)-y\right)^+-\xi^{[n]*}_{s-}\\
=&Y^{[n]*}_{s-}+\left(c_n(s,X^{[n]}_s)-Y^{[n]*}_{s-}\right)^+=Y^{[n]*}_{s-}\vee c_n(s,X^{[n]}_s),
\end{align*}
as needed (i.e., any jump of the control $\xi^{[n]*}$ will bring the controlled process to the boundary of the continuation set).
\end{proof}
Using the lemma we can now establish optimality of $\xi^{[n]*}$ and obtain $v_n$ as the integral of $u_n$. The proof of the next proposition follows very closely the proof of Theorem 5.1 in \cite{de2017optimal}, except that here we have a finite-fuel problem (see also \cite{baldursson1996irreversible} and \cite{karoui1991new} for earlier similar proofs). So we move it to the appendix for completeness.
\begin{proposition}[Value function of $\textbf{SC}^{[n]}$] \label{th:theoremConnection}
Let Assumptions \ref{ass:SDE}--\ref{ass:int} hold.
For any $(t,x,y)\in[0,T]\times\Sigma$ we have
\begin{align}
v_n(t,x,y)=\Phi_n(t,x)-\int_y^1u_n(t,x,z)\ud z,
\end{align}
with
\[
\Phi_n(t,x):=\E_{t,x}\left[\int_0^{T-t}e^{-rs}f(X^{[n]}_{t+s},1)\ud s\right].
\]
Moreover, $\xi^{[n]*}$ as in Lemma \ref{lem:SK} is optimal, i.e., $v_n(t,x,y)=J_n(t,x,y;\xi^{[n]*})$.
\end{proposition}
It turns out that the optimal control found above is also unique due to the strict concavity of $f(x,\,\cdot\,)$ (Assumption \ref{ass:f}-(ii)). A proof of uniqueness is given in Proposition \ref{prop:ex-uniq} below in a similar context and we omit it here to avoid repetitions.

\subsection{Limit of the iterative scheme}
Now that we have characterised the solution of the $n$-th singular control problem, we turn to the study of convergence of the iterative scheme. First we show monotonicity of the scheme in terms of the sequence of value functions $(u_n)_{n\ge 0}$ of the stopping problems.
\begin{proposition}[Monotonicity of the iterative scheme]\label{teo:monotonicity}
Under Assumptions \ref{ass:SDE}--\ref{ass:int} we have $u_n\ge u_{n+1}$ on $[0,T]\times\Sigma$ and $c_n\ge c_{n+1}$ on $[0,T]\times\mathbb R$. Moreover, for any $(t,x,y)\in[0,T]\times\Sigma$ we also have 
\begin{align}\label{eq:comp-Xn}
\text{$X^{[n]}_{t+s}\ge X^{[n+1]}_{t+s}$ and $Y^{[n]*}_{t+s}\ge Y^{[n+1]*}_{t+s}$ for $s\in[0,T-t]$, $\P_{t,x,y}$-a.s.}
\end{align}
Finally, $m^{[n]}\ge m^{[n+1]}$ on $[0,T]$.
\end{proposition}
\begin{proof}
We argue by induction and assume that for some $n\ge 0$ we have $m^{[n-1]}\ge m^{[n]}$ on $[0,T]$. Then, by monotonicity of the drift coefficient (Assumption \ref{ass:SDE}-(ii)), we have $a(x,m^{[n]}(t))\le a(x,m^{[n-1]}(t))$ for all $(t,x)\in[0,T]\times \mathbb R$. It follows from comparison results for SDEs \citep[see, e.g.,][Proposition 5.2.18]{karatzasShreve} and \eqref{eq:Xn} that $X^{[n]}_{t+s}\ge X^{[n+1]}_{t+s}$ for all $s\in[0,T-t]$, $\P_{t,x}$-a.s., for all $(t,x)\in[0,T]\times \mathbb R$.
By monotonicity of the profit function (Assumption \ref{ass:f}-(i)) we have $\partial_yf(X^{[n+1]}_{t+s},y)\le \partial_y f(X^{[n]}_{t+s},y)$ and therefore \eqref{eq:un} and \eqref{eq:Un} imply $u_{n+1}\le u_n$ on $[0,T]\times\Sigma$. The latter and the definition of the optimal boundary in \eqref{def:cn} give us $c_{n+1}\le c_n$ on $[0,T]\times\mathbb R$. Now, using the definition of the optimal control in Lemma \ref{lem:SK} we have $\P_{t,x,y}$-a.s.
\begin{align*}
\xi^{[n+1]*}_{t+s}=&\sup_{0\le u\le s}\left(c_{n+1}(t+u,X^{[n+1]}_{t+u})-y\right)^+\le\sup_{0\le u\le s}\left(c_{n}(t+u,X^{[n+1]}_{t+u})-y\right)^+\\
\le &\sup_{0\le u\le s}\left(c_{n}(t+u,X^{[n]}_{t+u})-y\right)^+=\xi^{[n]*}_{t+s},
\end{align*}
where the first inequality is due to $c_n\ge c_{n+1}$ and the second one to $X^{[n]}\ge X^{[n+1]}$, since $x\mapsto c_n(t,x)$ is non-decreasing (Proposition \ref{prop:OSboundary}). Monotonicity of the optimal controls implies monotonicity of the optimally controlled processes $Y^{[n]*}_{t+s}\ge Y^{[n+1]*}_{t+s}$ for all $s\in[0,T-t]$ and from the latter we obtain
\[
m^{[n+1]}(t)=\int_\Sigma \E_{x,y}\left[Y^{[n+1]*}_t\right]\nu(\ud x,\ud y)\le\int_\Sigma \E_{x,y}\left[Y^{[n]*}_t\right]\nu(\ud x,\ud y)=m^{[n]}(t). 
\]
So the argument is complete once we show that we can find $n\ge 0$ such that $m^{[n-1]}\ge m^{[n]}$ on $[0,T]$. The latter is true in particular for $n=0$ since $m^{[-1]}\equiv 1$ and $m^{[0]}\le 1$ on $[0,T]$.
\end{proof}

It is clear that by construction $0\le c_{n}(t,x)\le 1$ and $0\le m^{[n]}(t)\le 1$ for all $(t,x)\in[0,T]\times \mathbb R$ and all $n\ge 0$. Moreover, $a(x,0)\le a(x,m^{[n]}(t))\le a(x,1)$ for all $(t,x)\in[0,T]\times \mathbb R$ and all $n\ge 0$, so that by the comparison principle $\bar X^0_{t+s}\le X^{[n]}_{t+s}\le X^{[0]}_{t+s}$, for all $s\in[0,T-t]$, $\P_{t,x,y}$-a.s.~for all $n\ge 0$ and with $\bar X^0$ the solution of \eqref{eq:Xn} associated to $a(x,0)$.

By monotonicity of the sequences $(u_n)_{n\ge 0}$, $(c_n)_{n\ge 0}$ and $(m^{[n]})_{n\ge 0}$ we can define the functions 
\begin{align}\label{eq:lims}
u(t,x,y):=&\,\lim_{n\to\infty}u_n(t,x,y),\quad c(t,x):=\lim_{n\to\infty}c_n(t,x)\\
&\:\:\text{and}\:\: \widetilde m(t):=\lim_{n\to\infty}m^{[n]}(t),\notag
\end{align}
for all $(t,x,y)\in[0,T]\times\Sigma$. Pointwise limit preserves the monotonicity of $\widetilde{m}$, $c$ and $u$ with respect to $(t,x,y)$. Moreover, since $u_n$ is continuous and $c_n$, $m^{[n]}$ are upper semi-continuous for all $n\ge 0$ we have that
\begin{align}\label{eq:usc-lim}
\textit{the functions $u$, $\widetilde{m}$ and $c$ are upper semi-continuous}
\end{align}
on their respective domains as decreasing limit of upper semi-continuous functions. Since $\widetilde{m}$ is also non-decreasing, then it must be right-continuous.

Notice that for each $n\ge 0$ the null set in \eqref{eq:comp-Xn} depends on $n$ and $(t,x,y)$ so we denote it by $N^n_{t,x,y}$. Then we can define a universal null set $N_{t,x,y}:=\cup_{n\ge 0}N^n_{t,x,y}$ and for any $(t,x,y)\in[0,T]\times\Sigma$ and all $\omega\in\Omega\setminus N_{t,x,y}$ we define the processes $\widetilde X$ and $\widetilde \xi$ as
\begin{align}\label{eq:lim-Xi}
\widetilde X_{t+s}(\omega):=\lim_{n\to\infty}X^{[n]}_{t+s}(\omega)\quad\text{and}\quad \widetilde\xi_{t+s}(\omega):=\lim_{n\to\infty}\xi^{[n]*}_{t+s}(\omega),
\end{align} 
for all $s\in[0,T-t]$. We can then set $\widetilde X\equiv0$ and $\widetilde\xi\equiv0$ on $N_{t,x,y}$ and recall that the filtration is completed with $\P_{t,x,y}$-null sets, so that the limit processes are adapted. Of course we also have 
\[
\widetilde Y_t:=y+\widetilde{\xi}_t=\lim_{n\to\infty} Y^{[n]*}_t
\]
and thanks to monotone convergence we can immediately establish
\begin{align}\label{eq:m-lim}
\widetilde{m}(t)=\lim_{n\to\infty}\int_\Sigma\E_{x,y}\big[Y^{[n]*}_t\big]\nu(\ud x,\ud y)=\int_\Sigma\E_{x,y}\big[\,\widetilde Y_t\,\big]\nu(\ud x,\ud y).
\end{align}
Notice that here we are using that $(x,y)\mapsto\E_{x,y}[\xi^{[n]*}_t]$ is measurable, thanks to the explicit expression of $\xi^{[n]*}$ and measurability of $c_n$. Therefore $(x,y)\mapsto\E_{x,y}[\,\widetilde \xi_t\,]$ is measurable too as pointwise limit of measurable functions.

We now derive the dynamics of $\widetilde X$ and show that $\widetilde \xi\in\Xi_{t,x}(y)$.
\begin{lemma}[Limit state processes]\label{lem:convXY} 
Suppose Assumptions \ref{ass:SDE}--\ref{ass:int} hold. For any $(t,x,y)\in[0,T]\times\Sigma$ the process $\widetilde X$ is the unique strong solution of
\begin{equation}
\widetilde X_{t+s}=x+\int_0^s a\big(\widetilde X_{t+u},\widetilde{m}(t+u)\big)\ud u+\int_0^s\sigma\big(\widetilde X_{t+u}\big) \ud W_{t+u},\quad s\in[0,T-t], \label{limit-SDE}
\end{equation}
and the process $\widetilde \xi$ belongs to $\Xi_{t,x}(y)$.
\end{lemma}
\begin{proof}
Fix $(t,x,y)\in[0,T]\times\Sigma$. The first observation is that $\widetilde X$ and $\widetilde \xi$ are $(\cF_{t+s})_{s\ge 0}$-adapted processes as pointwise limit of adapted processes on $\Omega\setminus N_{t,x,y}$ and by $\P_{t,x,y}$-completeness of the filtration. Since $\widetilde\xi$ is decreasing limit of right-continuous non-decreasing processes (hence upper semi-continuous), then it is also non-decreasing and upper-semi continuous. The latter two properties imply right-continuity of the limit process $\widetilde\xi$ as well. Since $\xi^{[n]*}_{u}=0$ for $u\in[0,t)$ and $\xi^{[n]*}_{T}\le 1-y$, for all $n\ge 0$, we also have $\widetilde\xi_{u}=0$ for $u\in[0,t)$ and $\widetilde\xi_T\le 1-y$. Hence $\widetilde \xi\in\Xi_{t,x}(y)$.

Let us now prove \eqref{limit-SDE}. Denote by $X'$ the unique strong solution of \eqref{limit-SDE} and let us show that $\widetilde X=X'$. By standard estimates and using Lipschitz continuity of the drift $a(\,\cdot\,)$ (Assumption \ref{ass:SDE}-(i)) we have
\begin{align*}
&\E_{t,x}\left[\sup_{0\le s\le T-t}\left|X^{[n]}_{t+s}-X'_{t+s}\right|^2\right]\\
&\le 2\,\E_{t,x}\left[L^2\cdot T\int_0^{T-t}\left(\big|X^{[n]}_{t+s}-X'_{t+s}\big|^2+\big|m^{[n]}(t+s)-\widetilde{m}(t+s)\big|^2\right)\ud s\right]\\
&\quad\quad+2\,\E_{t,x}\left[\sup_{0\le s\le T-t}\left|\int_0^{s}\big(\sigma(X^{[n]}_{t+u})-\sigma(X'_{t+u})\big)\ud W_{t+u}\right|^2\right].
\end{align*}
Since $\sigma$ enjoys linear growth and $X^{[n]}$ and $X'$ are solutions of SDEs with Lipschitz coefficients, then
\[
s\mapsto \int_0^{s}\big(\sigma(X^{[n]}_{t+u})-\sigma(X'_{t+u})\big)\ud W_{t+u}
\] 
is a martingale on $[0,T-t]$ and we can use Doob's inequality to get
\begin{align*}
&\E_{t,x}\left[\sup_{0\le s\le T-t}\left|\int_0^{s}\big(\sigma(X^{[n]}_{t+u})-\sigma(X'_{t+u})\big)\ud  W_{t+u}\right|^2\right]\\
&\le 4\,\E_{t,x}\left[\int_0^{T-t}\big(\sigma(X^{[n]}_{t+s})-\sigma(X'_{t+s})\big)^2\ud s\right]\le 4 L^2 \E_{t,x}\left[\int_0^{T-t}\big|X^{[n]}_{t+s}-X'_{t+s}\big|^2\ud s\right],
\end{align*}
Combining the estimates above and using Gronwall's inequality we obtain
\[
\E_{t,x}\left[\sup_{0\le s\le T-t}\left|X^{[n]}_{t+s}-X'_{t+s}\right|^2\right]\le c\int_0^{T-t}\big|m^{[n]}(t+s)-\widetilde{m}(t+s)|^2\ud s,
\]
for some constant $c>0$. Letting $n\to\infty$ and using bounded convergence and the definition of $\widetilde{m}$ we conclude. 
\end{proof}

Next we connect $u(\,\cdot\,)$ and $c(\,\cdot\,)$ with an optimal stopping problem for $\widetilde X$. Recall that $u$ and $c$ are upper semi-continuous by \eqref{eq:usc-lim} and enjoy the same monotonicity properties of $u_n$ and $c_n$. Recall also the convention $u_n(t,x,0)=c_0$ if $\partial_y f(x,0)=\infty$.
\begin{lemma}[Limit optimal stopping problem]\label{lem:limitOS} 
Suppose Assumptions \ref{ass:SDE}--\ref{ass:int} hold. Then, for all $(t,x,y)\in[0,T]\times\Sigma$ we have
\begin{equation}
\label{eq:uOS}
\begin{split}
u(t,x,y)&=\inf_{\tau\in\mathcal T_t}U(t,x,y;\tau)\qquad \text{with}\\
U(t,x,y;\tau)&:= \E_{t,x}\left[\int_0^\tau e^{-rs}\partial_y f(\widetilde X_{t+s},y)\ud s+c_0e^{-r\tau}\right]
\end{split}
\end{equation}
and
\[
c(t,x)=\inf\{y\in[0,1]:u(t,x,y)<c_0\}\quad\text{with $\inf\varnothing = 1$}.
\]
In particular $c$ is the boundary of the set
\begin{align}\label{cC}
\cC:=\{(t,x,y)\in[0,T]\times\Sigma: u(t,x,y)<c_0\}
\end{align}
and, moreover, both $\cC$ and $\cS:=([0,T]\times\Sigma)\setminus \cC$ are not empty.
\end{lemma}
\begin{proof}
Since $X^{[n]}\ge \widetilde X$ for all $n\ge 0$ and $x\mapsto\partial_y f(x,y)$ is non-decreasing, for any $\tau\in\T_t$ we have
$U_n(t,x,y;\tau)\ge U(t,x,y;\tau)$
and therefore 
\[
u(t,x,y)=\lim_{n\to\infty}\inf_{\tau\in\T_t}U_{n}(t,x,y;\tau)\ge\inf_{\tau\in\T_t}U(t,x,y;\tau).
\]
Now, given $\eps>0$ we can find a stopping time $\tau_\eps\in\T_t$ such that
\[
\inf_{\tau\in\T_t}U(t,x,y;\tau)+\eps\ge U(t,x,y;\tau_\eps). 
\]
Moreover, by dominated convergence (Assumption \ref{ass:int}) and continuity of $\partial_y f$ we have
\[
U(t,x,y;\tau_\eps)=\E_{t,x}\left[\int_0^{\tau_\eps}e^{-r s}\lim_{n\to\infty}\partial_y f(X^{[n]}_{t+s},y)\ud s + c_0 e^{-r \tau_{\varepsilon}}\right]= \lim_{n\to\infty} U_n(t,x,y;\tau_\eps).
\]
So combining the above we get
\[
\inf_{\tau\in\T_t}U(t,x,y;\tau)+\eps\ge\lim_{n\to\infty} U_n(t,x,y;\tau_\eps)\ge \lim_{n\to\infty}u_n(t,x,y)=u(t,x,y)
\]
and since $\eps>0$ was arbitrary we conclude 
\[
u(t,x,y)\le \inf_{\tau\in\T_t}U(t,x,y;\tau)
\]
as needed for the first claim. Notice that the expression in \eqref{eq:uOS} is readily verified also if $\partial_y f(x,0)=\infty$, with $u_n(t,x,0)=u(t,x,0)=c_0$.

Let us next prove that $c$ coincides with the optimal stopping boundary for the limit problem. Since $u\le u_n$ for all $n\ge 0$ we have 
\[
c_n(t,x)=\inf\{y\in[0,1]: u_n(t,x,y)<c_0\}\ge \inf\{y\in[0,1]: u(t,x,y)<c_0\} 
\]
so that 
\[
c(t,x)\ge \inf\{y\in[0,1]: u(t,x,y)<c_0\}.
\]
For the reverse inequality, let us fix $(t,x)\in[0,T]\times \R$, take $\eta\in[0,1]$ such that 
\begin{align}\label{eta}
\eta > \inf\{y\in[0,1]: u(t,x,y)<c_0\}.
\end{align}
Then there must be $\delta>0$ such that $u(t,x,\eta)\le c_0-\delta$. By pointwise convergence, there exists $n_\delta\ge 0$ such that $u_n(t,x,\eta)\le u(t,x,\eta)+\delta/2$ for all $n\ge n_\delta$ and therefore, $u_n(t,x,\eta)\le c_0-\delta/2$ for all $n\ge n_\delta$. Hence, $\eta>c_n(t,x)$ for all $n\ge n_\delta$ and $\eta>c(t,x)$ too. The result holds for any $\eta\in[0,1]$ such that \eqref{eta} is true and therefore 
\[
c(t,x)\le  \inf\{y\in[0,1]: u(t,x,y)<c_0\}.
\]
Since $y\mapsto u(t,x,y)$ is decreasing it is clear that $c$ is the boundary of the set $\cC$ defined in \eqref{cC}. 

The exact same arguments as in the proof of Proposition \ref{prop:OSboundary} apply to the stopping problem with value $u$ and allow us to show that $\cC\neq\varnothing$ and $\cS\neq\varnothing$ thanks to \eqref{ass:partial-y-f} in Assumption \ref{ass:f}.
\end{proof}

Thanks to the probabilistic representation of $u$ we can use the same arguments as in the proof of Proposition \ref{prop:OSvalueFunc} to show that $u$ indeed fulfils the same properties as $u_n$.
\begin{corollary}\label{cor:reg-u}
Under Assumptions \ref{ass:SDE}-\ref{ass:int} the function $u$ satisfies (i)--(iv) in Proposition \ref{prop:OSvalueFunc}.
\end{corollary}

In what follows we let
\begin{align}\label{eq:tau*}
\tau_*(t,x,y)=\inf\{s\in[0,T-t]: u(t+s,\widetilde X^{t,x}_{t+s},y)=c_0\},
\end{align}
which is optimal for the limit problem with value $u(t,x,y)$.
Continuity of the value function allows a simple proof of convergence of optimal stopping times. The result is of independent interest and might be used for numerical approximation of the optimal stopping rule $\tau_*$. We state the result here but put its proof in the appendix as it will not be needed in the rest of the paper.
\begin{lemma}\label{lem:tau-n}
For all $(t,x,y)\in[0,T]\times\Sigma$ we have $\tau^{[n]}_*\uparrow \tau_*$, $\P_{t,x,y}$-a.s., as $n\to\infty$.
\end{lemma}

Since the dynamics of $(\widetilde X_t)_{t\in[0,T]}$ is fully specified and we have obtained a solution of the optimal stopping problem with value $u$ (Lemma \ref{lem:limitOS}), we can state a result similar to Proposition \ref{th:theoremConnection}.
\begin{proposition}\label{prop:ex-uniq}
Let Assumptions \ref{ass:SDE}--\ref{ass:int} hold and let $\widetilde X$ be specified as in Lemma \ref{lem:convXY}. Define
\begin{equation}\label{eq:hat-v}
\begin{split}
\hat v(t,x,y)&:=  \sup_{\xi\in\Xi_{t,x}(y)}\hat J(t,x,y;\xi)\qquad\text{with}\\
\hat J(t,x,y;\xi)&:= \E_{t,x}\left[\int_0^{T-t}e^{-rs}f(\widetilde X_{t+s},y + \xi_{t+s})\ud s-\int_{[0,T-t]}e^{-rs}c_0 \ud\xi_{t+s}\right].
\end{split}
\end{equation}
Then, for any $(t,x,y)\in[0,T]\times\Sigma$ we have
\begin{align}
\hat v(t,x,y)=\Phi(t,x)-\!\int_y^1\!u(t,x,z)\ud z\quad\text{with}\quad\Phi(t,x):=\E_{t,x}\left[\int_0^{T-t}\!\!e^{-rs}f(\widetilde X_{t+s},1)\ud s\right].
\end{align}
Moreover, 
\begin{align}\label{eq:xi*}
\xi^{*}_{t+s}:=\sup_{0\le u\le s}\left(c(t+u,\widetilde X_{t+u})-y\right)^+\quad\text{with}\quad \xi^{*}_{u}=0,\:u\in[0,t),
\end{align}
is the {\em unique} optimal control in \eqref{eq:hat-v}, up to indistinguishability.
\end{proposition}
\begin{remark}
Before passing to the proof we would like to emphasise that at this stage we are not claiming that $(\widetilde{m},\xi^*)$ is a solution of the MFG because $\widetilde{m}$ is specified in \eqref{eq:m-lim} and a priori the consistency condition (Definition \ref{def:MFGsol}-(ii)) may not hold. Hence, a priori $\hat v$ is not the function $v$ defined in \eqref{eq:P0}. Of course uniqueness of the optimal control also holds in Proposition \ref{th:theoremConnection} albeit not stated there.
\end{remark}
\begin{proof}
We only need to prove uniqueness of the optimal control as all remaining claims are obtained by repeating verbatim the proof of Proposition \ref{th:theoremConnection}. As usual, uniqueness follows by strict concavity of the map $y\mapsto f(x,y)$, convexity of the set $\Xi_{t,x}(y)$ of admissible controls and an argument by contradiction. 

For notational simplicity and with no loss of generality we take $t=0$. Assume that $\eta\in \Xi_{0,x}(y)$ is another optimal control. Since, $\eta$ and $\xi^*$ are both right-continuous they are indistinguishable as soon as they are modifications, i.e. if $\P_{x,y}(\xi^*_{s}=\eta_{s})=1$ for all $s\in[0,T]$. Arguing by contradiction assume there exists $s_0\in[0,T)$ such that $3 p:=\P_{x,y}(\xi^*_{s_0}\neq \eta_{s_0})>0$. Then, there also exists $\eps>0$ such that $\P_{x,y}(|\xi^*_{s_0}- \eta_{s_0}|\ge \eps)\ge 2p$ and, by right-continuity of the paths, there exists $s_1>s_0$ such that 
\[
\P_{x,y}\left(\inf_{s_0\le u\le s_1}|\xi^*_{u}- \eta_{u}|\ge \eps\right)\ge p>0.
\]
Let us denote 
\begin{align}\label{A0}
A_{0}:=\left\{\omega:\inf_{s_0\le u\le s_1}|\xi^*_{u}(\omega)- \eta_{u}(\omega)|\ge \eps\right\}.
\end{align}

For any $\lambda\in(0,1)$, since $\eta$ and $\xi^*$ are both optimal, we have 
\begin{align}
\hat v(0,x,y)=&\lambda \hat J(0,x,y;\eta)+(1-\lambda)\hat J(0,x,y;\xi^*)\\
=&\E_{x}\left[\int_0^{T}e^{-rs}\left[\lambda f(\widetilde X_{s},y + \eta_{s})+
(1-\lambda)f(\widetilde X_{s},y + \xi^*_{s})\right]\ud s\right]\\
&-\E_{x}\left[\int_{[0,T]}e^{-rs}c_0(\lambda\ud \eta_s+(1-\lambda)\ud\xi^*_s)\right].
\end{align}
Now, letting $\zeta^\lambda:=\lambda\eta+(1-\lambda)\xi^*$ it is immediate to check that $\zeta^\lambda\in\Xi_{0,x}(y)$ 
and, by strict concavity of $y\mapsto f(x,y)$ (and joint continuity of $f$), we have
\begin{align*}
&\mathds{1}_{A_0}\left[\lambda f(\widetilde X_{s},y + \eta_{s})+(1-\lambda)f(\widetilde X_{s},y + \xi^*_{s})\right]<\mathds{1}_{A_0} f( \widetilde X_{s},y + \zeta^\lambda_{s}),\quad\text{for $s\in[s_0,s_1]$}
\end{align*}
with $\mathds 1_{A_0}$ the indicator of the event in \eqref{A0}. For all times $s\in[0,T]$ and on $\Omega\setminus A_0$ the same inequality holds with `$\le$' by concavity. Since $\P_{0,x}(A_0)>0$ and $s_0<s_1$, the strict inequality holds for the expected values as well. Hence we reach the contradiction
\[
\hat v(0,x,y)<\hat J(0,x,y;\zeta^\lambda),
\]
which concludes the proof.
\end{proof}

\subsection{Solution of the MFG}

In this section we first show that $\widetilde \xi$ obtained in the previous section (see \eqref{eq:lim-Xi}) is optimal for the control problem in Proposition \ref{prop:ex-uniq} and then conclude that $(\widetilde m,\widetilde \xi)$ solves the MFG.
\begin{proposition}
Let Assumptions \ref{ass:SDE}--\ref{ass:int} hold, take $\widetilde \xi$ as in \eqref{eq:lim-Xi}, $\widetilde m$ as in \eqref{eq:m-lim} and $\widetilde X$ as in Lemma \ref{lem:convXY}. Then 
\[
\hat v(t,x,y)=\hat J(t,x,y;\widetilde \xi),\quad \text{for any $(t,x,y)\in[0,T]\times\Sigma$}
\]
and $\widetilde \xi$ is indistinguishable from $\xi^*$ as in \eqref{eq:xi*}.
\end{proposition}
\begin{proof}
We only need to prove optimality of $\widetilde \xi$ as the rest follows by uniqueness of the optimal control (Proposition \ref{prop:ex-uniq}). 

Recall the value function $v_n$ of $\textbf{SC}^{[n]}$ (see \eqref{SCn-1}--\eqref{SCn-2}) and its expression from Proposition \ref{th:theoremConnection}. Using dominated convergence we obtain
\begin{align}
&\lim_{n\to\infty}\left(\Phi_n(t,x)-\int_y^1 u_n(t,x,z)\ud z\right)\\
&=\E_{t,x}\left[\int_0^{T-t}e^{-rs}\lim_{n\to\infty}f(X^{[n]}_{t+s},1)\ud s\right]-\int_y^1 \lim_{n\to\infty}u_n(t,x,z)\ud z=\hat v(t,x,y),
\end{align}
where the final equality is due to \eqref{eq:lims}, \eqref{eq:lim-Xi} and Proposition \ref{prop:ex-uniq}. Therefore we have  
\[
\lim_{n\to\infty}v_n(t,x,y)=\hat v(t,x,y).
\] 
Since $v_n(t,x,y)=J_n(t,x,y;\xi^{[n]*})$, if we can show that
\[
\lim_{n\to\infty}J_n(t,x,y;\xi^{[n]*})=\hat J(t,x,y;\widetilde \xi),
\]
the proof is complete. The latter is not difficult, indeed by integration by parts and dominated convergence we have
\begin{align*}
&\lim_{n\to\infty}J_n(t,x,y;\xi^{[n]*})\\
&=\lim_{n\to\infty}\E_{t,x}\left[\int_0^{T-t}\!\!e^{-r s}f(X^{[n]}_{t+s},y+\xi^{[n]*}_{t+s})\ud s-c_0e^{-r(T-t)}\xi^{[n]*}_T-rc_0\int_0^{T-t}\!\!e^{-rs}\xi^{[n]*}_{t+s}\ud s\right]\\
&=\E_{t,x}\bigg[\int_0^{T-t}\!\!e^{-r s}\lim_{n\to\infty}f(X^{[n]}_{t+s},y+\xi^{[n]*}_{t+s})\ud s\\
&\qquad\qquad-c_0e^{-r(T-t)}\lim_{n\to\infty}\xi^{[n]*}_T-rc_0\int_0^{T-t}\!\!e^{-rs}\lim_{n\to\infty}\xi^{[n]*}_{t+s}\ud s\bigg]\\
&=\E_{t,x}\left[\int_0^{T-t}\!\!e^{-r s}f(\widetilde X_{t+s},y+\widetilde \xi_{t+s})\ud s-c_0e^{-r(T-t)}\widetilde \xi_T-rc_0\int_0^{T-t}\!\!e^{-rs}\widetilde \xi_{t+s}\ud s\right]\\
&=\hat J(t,x,y;\widetilde{\xi}\,),
\end{align*}
where the penultimate equality comes from \eqref{eq:lim-Xi} and the final one is obtained by undoing the integration by parts.
\end{proof}
By construction $\widetilde Y$ and $\widetilde m$ fulfill the consistency condition \eqref{eq:m-lim} hence we have a simple corollary.
\begin{corollary}
The pair $(\widetilde m,\widetilde \xi)$ is a solution of the MFG as in Definition \ref{def:solMFG}. Since $\widetilde\xi$ is indistinguishable from $\xi^*$ in \eqref{eq:xi*} then Theorem \ref{teo:existenceSolMFG} holds with 
\[
X^*=\widetilde X,\quad Y^*=Y_{0-}+\widetilde\xi=Y_{0-}+\xi^* \quad\text{and}\quad m^*=\widetilde{m}.
\] 
\end{corollary}
As a byproduct of this result and of Proposition \ref{prop:ex-uniq} we also have that the classical connection between singular stochastic control and optimal stopping still holds in our specific mean-field game.

\begin{remark}
It is interesting to notice that, upon a close inspection, the existence of a solution for our MFG could have been derived following ideas as in \cite{dianetti2019submodular} based on submodularity of the game's structure and the use of Tarski's fixed point theorem. However, that approach does not reveal any information on the structure of the solution, which makes its implementation in the $N$-player (pre-limit) games a very difficult task. 
Thanks to our constructive method, we have obtained a more explicit form of a MFG's solution, which we use in the next section to also obtain approximate Nash equilibria in the $N$-player games for $N$ large enough (with a convergence rate).

A clearer parallel at the technical level seems in order. In our set-up the function $f$ is independent of the measure flow so that the submodularity is trivially satisfied. We also have an indirect dependence of the players' objectives on the measure flow via the state $X$ (similarly to Subsection 4.4 in \cite{dianetti2019submodular}). Indeed, we recall that the drift of $X$ depends on the average of the state $Y$ in a monotone fashion (Assumption \ref{ass:SDE}-(ii)). Denoting $Y_t ^{m,*} := Y_{0-} + \xi_t ^{m,*}$ the optimally controlled process $Y$ (where we emphasise its dependence on $m(t)$), it is not difficult to show as in Proposition \ref{teo:monotonicity} that $m \mapsto \mathbb E[Y_t ^{m,*}]$ is non-increasing with respect to the natural order $m' \ge m$. The latter is defined as $m'(t) \ge m(t)$ for all $t\in [0,T]$ over the set of all measurable functions $m:[0,T]\to [0,1]$. Such set is a complete lattice under that order, so Tarski's fixed point theorem applies leading to the existence of MFG solutions in our case.
\end{remark}


\section{Approximate Nash equilibria for the $N$-player game}\label{sec:approximation}

\subsection{The $N$-player game: setting and assumptions}\label{subsec:Nplayer}
Here we start with a formal description of the $N$-player game sketched in the introduction.

Let $\Pi:=(\Omega, \mathcal{F}, \bar{\mathbb{F}}  = (\mathcal{F}_t)_{t \geq 0}, \bar \P)$ be a filtered probability space satisfying the  usual conditions, supporting an infinite sequence of independent one-dimensional $\bar{\mathbb{F}}$-Brownian motions $(W^{i})_{i=1}^\infty$, as well as i.i.d. $\mathcal{F}_0$-measurable initial states $(X_0^{i}, Y_{0-}^{i})_{i = 1}^\infty$ with common distribution $\nu \in \mathcal{P}(\Sigma)$, independent of the Brownian motions. With no loss of generality we assume that $\Pi$ is the same probability space that also accommodates the Brownian motion and the random initial conditions used in the setting of the MFG in Section \ref{subsec:mfg}. For each $N \ge 1$, define $\mathbb{F}^N  = (\mathcal F_0\vee\mathcal{F}^N_t)_{t \geq 0}$, where $(\mathcal F^N_t)_{t\ge 0}$ is the augmented filtration generated by the Brownian motions $(W^{i})_{i=1}^N$. Then the filtered probability spaces $\Pi^N:=(\Omega, \mathcal{F}, \mathbb{F}^N, \bar \P)$ satisfy the usual conditions. These are the spaces on which we define strong solutions for the SDEs appearing in the $N$-player systems. In what follows the classes of admissible strategies associated to the probability spaces $\Pi^N$ are denoted by $\Xi^{\Pi^N}$ with the same meaning as in Section \ref{sec:notation} but with $\Pi^N$ instead of $\Pi$.

Each player $i = 1, \ldots, N$ observes/controls her own private state process $(X^{N,i}, Y^{N,i})$, whose dynamics is
\begin{equation}
\begin{split}
& X_{t}^{N,i}= X_0^{i} + \int_{0}^{t}\,a(X_s^{N,i},m_s^{N})\,\ud s + \int_0^t\,\sigma(X_s^{N,i})\,\ud W_s^{i},\\
& Y_{t}^{N,i} = Y_{0-}^{i} + \xi^{N,i}_t,\quad\quad t \in [0,T],
\end{split}
\label{eq:dynamics_Nplayer_01}
\end{equation}
where $\xi^{N,i}\in\Xi^{\Pi^N}(Y^{i}_{0-})$ is the strategy chosen by the $i$-th player, while $m^{N}$ is the mean-field interaction term given by
\begin{equation} \label{emp-mean2}
m_t^{N} = \frac{1}{N}\sum_{i = 1}^{N} Y_t^{N,i} = \int_{\Sigma}\,y\,\mu_t^{N}(\ud x,\ud y),\quad \mu_t^{N} = \frac{1}{N}\sum_{i = 1}^{N} \delta_{(X_t^{N,i}, Y_{t}^{N,i})}.
\end{equation}
The process $\mu_t ^N$ above denotes the empirical distribution of the players' states, with $\delta_{z}$ the Dirac delta mass at $z \in \Sigma$. 

In the rest of this section we use the notations $\xi^N:=(\xi^{N,i})_{i=1}^N$ and
\[
\Xi^N(Y_{0-})=\{\xi^N:\:\xi^{N,i}\in\Xi^{\Pi^N}(Y^{i}_{0-})\:\:\text{for all $i=1,\ldots N$ }\},
\]
where $Y_{0-} = (Y_{0-}^{1}, \ldots, Y_{0-}^{N})$.
We will also consider the dynamics of $(X^N,Y^N)$ conditionally on specific initial conditions $({\bf x},{\bf y}):=(x^{i}, y^{i})_{i = 1}^{N}$ drawn independently from the common initial distribution $\nu$. Analogously to \eqref{eq:pbar} we have
\[
\bar\P (\,\cdot\,)=\int_{\Sigma^N} \P_{{\bf x},{\bf y}} (\,\cdot\,)\nu^N (\ud {\bf x},\ud{\bf y}),\quad \bar\E (\,\cdot\,)=\int_{\Sigma^N} \E_{{\bf x},{\bf y}} (\,\cdot\,)\nu^N (\ud {\bf x},\ud{\bf y}),
\]
where $\nu^N$ is the $N$-fold product of the measure $\nu$. The dynamics of the state variables under $\P_{{\bf x},{\bf y}}$ reads 
\begin{equation*}
\begin{split}
& X_{t}^{N,i}= x^{i} + \int_{0}^{t}\,a(X_s^{N,i},m_s^{N})\,\ud s + \int_0^t\,\sigma(X_s^{N,i})\,\ud W_s^{i},\\
& Y_{t}^{N,i} = y^{i} + \xi^{N,i}_t,\quad\quad t \in [0,T].
\end{split}
\label{eq:dynamics_Nplayer_02}
\end{equation*}
Accordingly, since the initial conditions $({\bf x},{\bf y})=(x^{i}, y^{i})_{i=1}^N$ are drawn from $\nu^N$,  the expected payoff of the $i$-th player is given by
\begin{equation}
J^{N,i}(\xi^{N}) := \int_{\Sigma^{N}} J^{N,i}({\bf x}, {\bf y}; \xi^{N}) \nu^{N}(\ud {\bf x}, \ud {\bf y}),
\nonumber
\end{equation}  
where 
\[
J^{N,i}({\bf x}, {\bf y}; \xi^{N}):=\E_{{\bf x},{\bf y}}\left[\int_0^T e^{-rt}f\left(X^{N,i}_t,Y^{N,i}_t\right)\ud t-\int_{[0,T]}e^{-r t}c_0\ud \xi^{N,i}_t\right].
\]

\begin{definition}[$\varepsilon$-Nash equilibrium for the $N$-player game] Given $\varepsilon \ge 0$, an admissible strategy vector $\xi^\eps\in\Xi^N(Y_{0-})$ is called  \textit{$\varepsilon$-Nash equilibrium} for the $N$-player game of capacity expansion if for every $i = 1,\ldots,N$ and for every admissible individual strategy $\xi^{i}\in\Xi^{\Pi^N}(Y^{i}_{0-})$, we have
\begin{equation}
J^{N,i}(\xi^\eps)\geq J^{N,i}([\xi^{\eps, -i},\xi^i])-\varepsilon,\quad
\nonumber
\end{equation}
where $[\xi^{\eps,-i},\xi^i]$ denotes the $N$-player strategy vector that is obtained from $\xi^{\eps}$ by replacing 
the $i$-{th} entry with $\xi^i$.
\label{def:nashEq}
\end{definition}

In order to construct $\eps$-Nash equilibria using the optimal control obtained in the MFG it is convenient to make an additional set of assumptions on the profit function. 
\begin{assumption}\label{ass:f2}
The running payoff $f$ is locally Lipschitz, i.e.
\begin{equation}\label{eq:local-Lipschitz}
|f(x,y)-f(x',y')|\leq \Lambda(x,x')(|x-x'|+|y-y'|),\quad (x,y),(x',y')\in\Sigma.
\end{equation}
Moreover, there exists $q>1$ such that the function $\Lambda:  \mathbb{R} \times \mathbb{R}\rightarrow\R_+$ satisfies the integrability condition
\begin{equation}\label{eq:int-Lambda}
C(\Lambda,q):=\sup_{N\in\N}\,\sup_{\xi^{N,1}}\bar\E\left[\int_0^T\Lambda^q(X_t^{N,1},X^*_t)\ud t\right]<\infty 
\end{equation}
where $X^{N,1}$ is the solution of \eqref{eq:dynamics_Nplayer_01}, $ X^*=\widetilde X$ is the solution of \eqref{eq:dynamics_mfg_01} obtained in the MFG (see also \eqref{limit-SDE}) and the supremum is taken over all admissible controls $\xi^{N,1}\in\Xi^{\Pi^N}(Y^1_{0-})$ and all $N\in\N$.
\end{assumption}

The assumption above is of technical nature and it is needed in the proof of Theorem \ref{teo:approximation} in order to use dominated convergence in some steps of the construction of $\varepsilon$-Nash equilibria. In Section \ref{sec:Lip} we will provide two examples of running profit of Cobb-Douglas type that satisfy that assumption.
The integrability condition \eqref{eq:int-Lambda} is redundant if $f$ is Lipschitz continuous. Since $Y^{N,1}\in[0,1]$ there is no loss of generality in taking $\Lambda$ independent of $y$ and the supremum over $\xi^{N,1}$ is not restrictive either. If $\Lambda$ has polynomial growth of order $p\ge 1$, then \eqref{eq:int-Lambda} holds thanks to the Lipschitz continuity of the coefficients $a(x,m)$ and $\sigma(x)$ as soon as $\E[(X_0)^{p\cdot q}]<\infty$. Later in Section \ref{subsec:gbm} we will consider and example where $\Lambda$ has exponential growth  and \eqref{eq:int-Lambda} holds.

The next is an assumption on the optimal boundary found in Theorem \ref{teo:existenceSolMFG}.
\begin{assumption}\label{ass:Lip-c}
The optimal boundary $(t,x)\mapsto c(t,x)$ of Theorem \ref{teo:existenceSolMFG} is uniformly Lipschitz continuous in $x$ with constant $\theta_c>0$, i.e.
\[
\sup_{0\le t\le T}|c(t,x)-c(t,x')|\le \theta_c|x-x'|,\quad x,x'\in\R.
\] 
\end{assumption}

The study of regularity of free boundaries in singular stochastic control and in optimal stopping has a long tradition. 
Early work by, e.g., \cite{VM75} and \cite{K74} address the question in optimal stopping of one-dimensional diffusions with finite-time horizon. In those cases the free boundary is a function of time only and it is shown to be continuously differentiable away from the terminal time in the optimisation. Methods from those papers (and subsequent ones in analogous settings) do not extend easily to free boundaries which are functions of several variables. 

In singular stochastic control some of the main contributions were given by \cite{soner1991free}, \cite{soner1989} and \cite{williams1994regularity}, who adopt an approach based on variational inequalities with gradient constraint. The latter two papers consider problems with an infinite-time horizon, which lead to elliptic variational problems. The first paper is more closely related to our set-up as it considers problems with finite-time horizon and $d$-dimensional dynamics (with controls acting on a single spatial coordinate). In \cite{soner1991free} the free boundary is shown to be a Lipschitz continuous function of time and of $d\!-\!1$ spatial coordinates. Differently from our set-up the state dynamics in all three papers has diffusive component in all its spatial coordinates, thus leading to uniformly elliptic second order differential operators. In our case there is no diffusive component in the $y$-coordinate, so those PDE techniques cannot be employed directly.

Free boundary problems are also widely studied in the PDE literature, often beyond their applications in stochastic control theory. Self-contained expositions of fundamental results and complementary methods are contained, for example, in the monographs \cite{F88} and \cite{caffarelli2005geometric}. Also in these references the regularity of the free boundary is analysed when the second order differential operator is the Laplacian or a uniformly elliptic operator with suitable coefficients. Extensions to degenerate settings like ours are not straightforward.

A probabilistic approach developed in \cite{deangelisStabile} instead does not require uniform non-degeneracy of the state dynamics. In Section \ref{sec:Lip} we show how those ideas can be used in our framework to prove that Assumption \ref{ass:Lip-c} indeed holds in a large class of examples.

\subsection{Approximate Nash equilibria}
Here we prove that the MFG solution constructed in Theorem \ref{teo:existenceSolMFG} induces approximate Nash equilibria in the $N$-player game of capacity expansion, when $N$ is large enough. 
\begin{theorem}[Approximate Nash equilibria for the $N$-player game]\label{teo:approximation} Suppose Assumptions \ref{ass:SDE}--\ref{ass:Lip-c} hold. Recall the solution $(m^*,\xi^*)$ of the MFG of capacity expansion constructed in Theorem \ref{teo:existenceSolMFG}. Notice that the control $\xi^*$ is in feedback form and reads:
\begin{equation}
\xi^*_t= \eta^*(t,X^*,Y_{0-}),\quad t\in[0,T],
\nonumber
\end{equation}
with $\eta^*:[0,T]\times C([0,T]; \mathbb R)\times[0,1]\to [0,1]$ the non-anticipative mapping defined by
\begin{equation}
\eta^*(t,\varphi,y):= \sup_{0\le s\le t}\Big(c(s,\varphi(s))-y\Big)^+
\label{eq:eta*}
\end{equation}
and with $X^*$ the dynamics in \eqref{eq:dynamics_mfg_01} associated to $m^*$. 

Setting $\hat \xi^{N,i}_t:= \eta^*(t,X^{N,i},Y^i_{0-})$, the vector $\hat \xi^{N}$ is a $\varepsilon_N $-Nash equilibrium for the $N$-player game of capacity expansion according to Definition \ref{def:nashEq} with $\varepsilon_N \rightarrow 0$ as $N\rightarrow\infty$. Further, if $q\ge 2$ in \eqref{eq:int-Lambda} of Assumption \ref{ass:f2}, then the rate of convergence is at least of order $N^{-1/2}$.
\end{theorem}
\begin{proof}
For each Brownian motion $W^i$ we introduce the following auxiliary dynamics with $m^*$ as in Theorem \ref{teo:existenceSolMFG} and with $\eta^*$ as defined in \eqref{eq:eta*}:
\begin{equation}\label{X-limN}
\begin{split}
& X^i_t=X^i_0+\int_0^t a(X^i_s,m^*(s))ds+\int_0^t\sigma(X^i _s) \ud W^i_s,\\
& Y^i_t=Y^i_{0-}+\zeta^i_t:= Y_{0-}+\eta^*(t,X^i,Y^i_{0-}),\quad t\in[0,T],\quad i\in\lbrace 1,\ldots,N\rbrace.
\end{split}
\end{equation}
These are the analogues of the solution $(X^*,Y^*)$ of \eqref{eq:dynamics_mfg_01}.

Notice that the initial conditions above are the same as in the dynamics of $(X^{N,i},Y^{N,i})$. Moreover, $(X^i_t,Y_t^i)_{i=1}^\infty$ is a sequence of i.i.d.~random variables with values in $\mathbb R\times [0,1]$, so that in particular the law of large numbers (LLN) holds.
The rest of the proof is structured in three steps:
\begin{itemize}
\item[(i)] We prove that $J^{N,1}(\hat \xi^N)\rightarrow J(\xi^*)$ as $N\rightarrow\infty$.
\item[(ii)] Recalling the notation $[\hat \xi^{N,-1},\xi]=(\xi,\hat \xi^{N,2},\ldots,\hat \xi^{N,N})$ introduced in Definition \ref{def:nashEq} we prove 
\begin{equation}
\limsup_{N\to\infty} \sup_{\xi\in\Xi^{\Pi^N}(Y^1_{0-})} J^{N,1}([\hat \xi^{N,-1},\xi])\leq J(\xi^*)=V^\nu.
\nonumber
\end{equation}
\item[(iii)] Combining (i) and (ii), for any $\varepsilon>0$ there exists $N_{\varepsilon}\in\N$ such that
\[
J^{N,1}(\hat \xi^N)\geq \sup_{\xi \in \Xi^{\Pi^N}(Y^{1}_{0-})}J^{N,1}([\hat\xi^{N,-1}, \xi])-\varepsilon 
\] 
for all $N\geq N_{\varepsilon}$. 
\end{itemize} 
In the three steps above we singled out the first player with no loss of generality since the $N$-player game is symmetric.
\vspace{+4pt}

(i) Let us start by observing that $(X^*,Y^*,\xi^*)$ from Theorem \ref{teo:existenceSolMFG} and $(X^1,Y^1,\zeta^1)$ defined above have the same law, so that 
\[
J(\xi^*)= \bar{\E}\left[\int_0^T e^{-r s}f(X^1_t,Y^1_t)\ud t-c_0\int_{[0,T]}e^{-r t}\ud \zeta^1_t\right].
\]
By triangular inequality we get 
\begin{align}\label{eq:J-J}
|J^{N,1}(\hat \xi^N)- J(\xi^*)| & \le \bar{\E}\left[\int_0^Te^{-rt}\Big|f(\hat X_t^{N,1},\hat Y_t^{N,1})-f(X_t^1,Y_t^1)\Big| \ud t \right] \\
& \quad + c_0 \bar{\E}\left[\left|\int_{[0,T]}e^{-rt}\ud \big(\hat \xi^{N,1}_t-\zeta^1_t\big)\right|\right],\notag
\end{align}
where we use $(\hat X^{N,i},\hat Y^{N,i})$ for the state process of the $i$-th player when all players use the control vector $\hat \xi^N$. Similarly we denote by $\hat m^N$ the empirical average of the processes $\hat Y^{N,i}$, that is
\[
\hat m^N_t=\frac{1}{N}\sum_{i=1}^N \hat Y^{N,i}_t.
\]
We estimate the first term on the right-hand side using Assumption \ref{ass:f2} and obtain
\begin{align*}
&\bar{\E}\left[\int_0^T e^{-rt}\Big|f(\hat X_t^{N,1},\hat Y_t^{N,1})-f(X_t^1,Y_t^1)\Big|\ud t \right]\nonumber\\
&\leq \bar{\E}\left[\int_0^Te^{-rt}\Lambda(\hat X_t^{N,1},X_t^1)\big(|\hat X_t^{N,1}-X_t^1|+|\hat Y_t^{N,1}-Y_t^1|\big)\ud t \right]\nonumber\\
&\leq C_{1} \bar{\E}\left[\int_0^T\!\!\Lambda^q(\hat X_t^{N,1},X_t^1)\ud t\right]^{\frac{1}{q}} \bar{\E}\left[\int_0^T\big(|\hat X_t^{N,1}-X_t^1|^{p}+|\hat Y_t^{N,1}-Y_t^1|^{p}\big)\ud t\right]^{\frac{1}{p}}\nonumber\\
&\leq C_{1} C(\Lambda,q) \bar{\E}\left[\int_0^T \big(|\hat X_t^{N,1}-X_t^1|^{p}+|\hat Y_t^{N,1}-Y_t^1|^{p}\big)\ud t\right]^{\frac{1}{p}},
\end{align*}
for some positive constant $C_{1}=C_1(T,q)$, using H\"older's inequality with $p=q/(q-1)$ and $q>1$ as in Assumption \ref{ass:f2}. For the remaining term in \eqref{eq:J-J} we use integration by parts and $\hat \xi^{N,1}_{0-}=\zeta^1_{0-}=0$ to obtain
\begin{align}
\bar{\E}\left[\left|\int_{[0,T]}e^{-rt}\ud \big(\hat \xi^{N,1}_t-\zeta^1_t\big)\right|\right]=\bar{\E}\left[\left|\int_0^T e^{-rt}\big(\hat \xi^{N,1}_t-\zeta_t^1\big)\ud t \right|\right]\le \bar{\E}\left[\int_0^Te^{-rt}\Big|\hat \xi^{N,1}_t-\zeta_t^1\Big|\ud t\right].
\nonumber
\end{align}
Recall that $\hat \xi^{N,1}_t=\eta^*(t,\hat X^{N,1},Y^{1}_{0-})$ and $\zeta^1_t=\eta^*(t,X^{1},Y^{1}_{0-})$. Then using Assumption \ref{ass:Lip-c} we obtain
\begin{align}\label{lip0}
\Big|\hat \xi^{N,1}_t-\zeta_t^1\Big|\le \sup_{0\le s\le t}\big|c(s,\hat X^{N,1}_s)-c(s,X^1_s)\big|\le \theta_c \sup_{0\le s\le t}\big|\hat X^{N,1}_s-X^1_s\big|
\end{align}
and the same bound also holds for $|\hat Y^{N,1}_t-Y^1_t|$. Then, combining the above estimates we arrive at
\begin{align}\label{eq:JJN}
|J^{N,1}(\hat \xi^N)- J(\xi^*)|\le&\, C_{1} C(\Lambda,q)T(1+\theta_c) \bar{\E}\Big[\sup_{0\le t\le T}|\hat X_t^{N,1}-X_t^1|^{p}\Big]^{\frac{1}{p}}\\
&+c_0\theta_c T\,\bar{\E}\Big[\sup_{0\le t\le T}\big|\hat X^{N,1}_t-X^1_t\big|\Big].\notag
\end{align}
Since $p>1$ it remains to show that 
\begin{align}\label{limXXN}
\lim_{N\to\infty} \bar{\E}\Big[\sup_{0\le t\le T}|\hat X_t^{N,1}-X_t^1|^{p}\Big]=0.
\end{align}
Repeating the same estimates as those in the proof of Lemma \ref{lem:convXY} but with $(\hat X^{N,1},X^1)$ instead of $(X^{[n]},X')$ and with $(\hat m^N,m^*)$ instead of $(m^{[n]},\widetilde m)$ we obtain 
\begin{align}\label{XXN-0}
\bar{\E} \Big[\sup_{0\le t\le T}|\hat X_t^{N,1}-X_t^1|^{p}\Big]\le C\,\bar{\E}\left[\int_0^T\left|\hat m^N_t-m^*(t)\right|^p\ud t\right],
\end{align}
for some constant $C>0$ depending on $p$, $T$ and the Lipschitz constant of the coefficients $a(\,\cdot\,)$ and $\sigma(\,\cdot\,)$. 

In order to estimate the right-hand side of \eqref{XXN-0} we first observe that for $Y^i$ introduced in \eqref{X-limN} we have
\begin{equation}\label{eq:epsN}
\varepsilon_{p,N}:= \int_0^T \bar{\E} \left[\,\left|\frac{1}{N}\sum_{i=1}^N Y^i_t-m^*(t)\right|^{p}\,\right]\ud t\rightarrow 0, \quad\text{as $N \to \infty$,}
\end{equation}
by the LLN and the bounded convergence theorem, since $(Y^i_t)_{i=1}^N$ are i.i.d.~with mean $m^*(t)$ (recall that $\eta^*$ is the feedback map of the optimal control in the MFG).
Hence, we have
\begin{align}\label{eq:limN}
\bar{\E}\left[\int_0^T|\hat m^N_t-m^*(t)|^{p}\ud t\right] \le& 2^{p-1}\int_0^T \bar{\E}\left[\,\left|\frac{1}{N}\sum_{i=1}^N(\hat Y^{N,i}_t-Y^i_t)\right|^{p}\,\right]\ud t+2^{p-1}\varepsilon_{p,N}\nonumber\\
\leq &2^{p-1}\int_0^T\frac{1}{N}\sum_{i=1}^N \bar{\E}\left[|\hat Y^{N,i}_t-Y^i_t|^{p}\right]\ud t+2^{p-1}\varepsilon_{p,N}\\
=&2^{p-1}\int_0^T \bar{\E} \left[\,\left|\hat Y^{N,1}_t-Y^1_t\right|^{p}\,\right]\ud t +2^{p-1}\varepsilon_{p,N}\nonumber\\
\le& 2^{p-1}\theta^p_c \int_0^T \bar{\E}\left[\,\sup_{0\le s\le t}\left|\hat X^{N,1}_s-X^1_s\right|^{p}\,\right]\ud t +2^{p-1}\varepsilon_{p,N}\nonumber
\end{align}
where the first inequality uses $|a+b|^p\le 2^{p-1}(|a|^p+|b|^p)$, the second inequality follows by Jensen's inequality ($p>1$), the equality by the fact that the processes $(\hat Y^{N,i}-Y^i)_{i=1}^N $ are exchangeable and the final inequality uses \eqref{lip0} applied to $|\hat Y^{N,1}_t-Y^1_t|$.

Plugging the latter estimate back into \eqref{XXN-0} and applying Gronwall's lemma we obtain
\[
\bar{\E}\Big[\sup_{0\le t\le T}|\hat X_t^{N,1}-X_t^1|^{p}\Big]\le C' \varepsilon_{p,N},
\]
for a suitable constant $C'>0$ depending on $T$ and the other constants above. Thanks to \eqref{eq:epsN} we obtain \eqref{limXXN}.
\vspace{+4pt}

(ii). This part of the proof is similar to the above but now the first player deviates by choosing a generic admissible control $\xi$ while all remaining players pick $\hat \xi^{N,i}$, $i=2,\ldots, N$; we denote this strategy vector $\beta^N=[\hat \xi^{N,-1},\xi]$. In particular we notice that the empirical average associated to this strategy reads
\[ 
\frac{1}{N}\left(Y^{1}_{0-}+\xi_t+\sum_{i=2}^N(Y^i_{0-}+\hat \xi^{N,i}_t)\right)=\bar m^N_t+\frac{1}{N}(\xi_t-\hat \xi^{N,1}_t),
\]
where $\bar m^N_t:=N^{-1}\sum_{i=1}^N(Y^i_{0-}+\hat \xi^{N,i}_t)$. One should be careful here that $\bar m^N$ is different to $\hat m^N$ used in the proof of (i) above, because the deviation of player 1 from the strategy vector $\hat \xi^N$ causes a knock-on effect on the dynamics of $\hat \xi^{N,i}$ for all $i$'s through the non-anticipative mapping $\eta^*(t,X^{N,i;\beta},Y^i_{0-})$. To keep track of this subtle aspect we  use the notations $\hat \xi^{N,i;\beta}_t=\eta^*(t,X^{N,i;\beta},Y^i_{0-})$ and $\bar Y^{N,i;\beta}_t=Y^i_{0-}+\hat \xi^{N,i;\beta}_t$, for $i=1,\ldots, N$, in the calculations below.
Accordingly, the state process of the first player reads
\begin{align*}
&X^{N,1;\beta}_t=X^{1}_0+\int_0^t a\big(X^{N,1;\beta}_s,\bar m^N_s+N^{-1}(\xi_s-\hat \xi^{N,1;\beta}_s)\big)\ud s+\int_0^t\sigma(X^{N,1;\beta}_s)\ud W^1_s\\
&Y^{N,1;\beta}_t=Y^{1}_{0-}+\xi_t,\qquad t\in[0,T].
\end{align*}

Using the above expression for $X^{N,1;\beta}$ and the same arguments as in \eqref{XXN-0} we obtain
\begin{align*}
& \bar{\E}\Big[\sup_{0\le t\le T}|X_t^{N,1;\beta}-X_t^1|^{p}\Big]\\
&\le C\, \bar{\E}\left[\int_0^T\left|\bar m^N_t+N^{-1}(\xi_t-\hat \xi^{N,1;\beta}_t)-m^*(t)\right|^p\ud t\right]\\
&\le 2^{p-1}C\, \bar{\E}\left[\int_0^T\left|\bar m^N_t-m^*(t)\right|^p\ud t\right]+2^p 2^{p-1} C T N^{-p},
\end{align*}
where the final inequality uses $|a+b|^p\le 2^{p-1}(|a|^p+|b|^p)$ and $|\xi_t-\hat \xi^{N,1;\beta}_t|\le 2$ (by the finite-fuel condition), and $C>0$ is a suitable constant. Repeating the same steps as in \eqref{eq:limN} we have
\begin{align}
\bar{\E}\left[\int_0^T|\bar m^N_t-m^*(t)|^{p}\ud t\right] \le& 2^{p-1}\int_0^T \bar{\E}\left[\,\left|\frac{1}{N}\sum_{i=1}^N(\bar Y^{N,i;\beta}_t-Y^i_t)\right|^{p}\,\right]\ud t+2^{p-1}\varepsilon_{p,N}\nonumber\\
\le&2^{p-1}\int_0^T \bar{\E}\left[\,\left|\bar Y^{N,1;\beta}_t-Y^1_t\right|^{p}\,\right]\ud t +2^{p-1}\varepsilon_{p,N}\nonumber\\
\le& 2^{p-1}\theta^p_c \int_0^T \bar{\E}\left[\,\sup_{0\le s\le t}\left| X^{N,1;\beta}_s-X^1_s\right|^{p}\,\right]\ud t +2^{p-1}\varepsilon_{p,N}\nonumber
\end{align}
where we have used that $(\bar Y^{N,i;\beta}-Y^i)_{i=1}^N$ are exchangeable by construction. Combining the two estimates above and using Gronwall's inequality we obtain a bound which is uniform with respect to $\xi\in\Xi^{\Pi^N}(Y^1_{0-})$. In particular we have  
\begin{align}\label{limXb}
\lim_{N\to\infty}\sup_{\xi\in\Xi^{\Pi^N}(Y^1_{0-})} \bar{\E}\Big[\sup_{0\le t\le T}|X_t^{N,1;\beta}-X_t^1|^{p}\Big]\le C'\lim_{N\to \infty}\big(\eps_{p,N}+N^{-p}\big)  =0,
\end{align}
where $C'>0$ is the constant appearing from Gronwall's inequality. 
Since any $\xi\in\Xi^{\Pi^N}(Y^1_{0-})$ is admissible but suboptimal in the MFG with state process $X^1$ as in \eqref{X-limN} we get 
\begin{align*}
&\sup_{\xi\in\Xi^{\Pi^N}(Y^1_{0-})}J^{N,1}([\hat \xi^{N,-1},\xi])-V^\nu\\
&\le \sup_{\xi\in\Xi^{\Pi^N}(Y^1_{0-})}\left(J^{N,1}([\hat \xi^{N,-1},\xi])-J(\xi)\right)\\
&\le \sup_{\xi\in\Xi^{\Pi^N}(Y^1_{0-})} \bar{\E}\left[\int_0^T e^{-r s}\Big(f\big(X_t^{N,1;\beta}, Y^{1}_{0-}+\xi_t\big)-f\big(X_t^{1}, Y^{1}_{0-}+\xi_t\big)\Big)\ud t\right]\\
&\le \sup_{\xi\in\Xi^{\Pi^N}(Y^1_{0-})} \bar{\E}\left[\int_0^Te^{-rt}\Lambda\big(X^{N,1;\beta}_t,X^1_t\big)\big|X^{N,1;\beta}_t-X^{1}_t\big|\ud t\right]
\end{align*}
where in the final inequality we used Assumption \ref{ass:f2}. Now, arguing as in (i) and using \eqref{limXb} and \eqref{eq:int-Lambda} we obtain
\begin{align*}
\limsup_{N\to\infty}\sup_{\xi\in\Xi^{\Pi^N}(Y^1_{0-})}J^{N,1}\big([\hat \xi^{N,-1},\xi]\big)\le V^\nu=J(\xi^*),
\end{align*}
where the final equality holds by optimality of $\xi^*$ in the MFG.
\vspace{+4pt}

(iii). This step follows from the previous two. With no loss of generality we consider only the first player as the game is symmetric. Given $\eps>0$, thanks to (ii) there exists $N_{\eps}>0$ sufficiently large that for any $\xi\in\Xi^{\Pi^N}(Y^{1}_{0-})$
\[
J^{N,1}([\hat \xi^{N,-1},\xi])\le V^\nu +\frac{\eps}{2}\quad\text{for all $N>N_{\eps}$}.
\]
From (i), with no loss of generality we can also assume $N_{\eps}>0$ sufficiently large that
\[
J^{N,1}(\hat \xi^N)\ge V^\nu-\frac{\eps}{2}\quad\text{for all $N>N_{\eps}$}.
\]
Combining the two inequalities above we obtain that for all $\xi\in\Xi^{\Pi^N}(Y^1_{0-})$ it holds
\[
J^{N,1}(\hat \xi^N)\ge J^{N,1}([\hat \xi^{N,-1},\xi])-\eps\quad\text{for all $N>N_{\eps}$}.
\]

The final claim on the speed of convergence can be verified by taking $q=p=2$ in the above estimates (for $q>2$ the result clearly continues to hold). The leading term in the convergence of \eqref{eq:JJN} is $\sqrt{\eps_{2,N}}$ (see \eqref{limXb}). Equation \eqref{eq:epsN} reads
\[
\varepsilon_{2,N}= \int_0^T\text{Var}\left(\frac{1}{N}\sum_{i=1}^N Y^i_t\right)\ud t=\frac{1}{N}\int_0^T\text{Var}(Y^1_t)\,\ud t,
\]
upon noticing that $\E\big[N^{-1}\sum_{i=1}^N Y^i_t\big]=\E[Y^1_t]=m^*(t)$ since $(Y^i)_{i=1}^N$ are i.i.d. Then the claim follows.
\end{proof}

\subsection{Conditions for a Lipschitz continuous optimal boundary}\label{sec:Lip}
Here we complement results from \cite{deangelisStabile} to provide sufficient conditions under which Assumption \ref{ass:Lip-c} holds. Notice that our problem is parabolic and degenerate as there is no diffusive dynamics in the $y$-direction. Therefore, as explained above, classical PDE results cannot be applied. Moreover, we extend \cite{deangelisStabile} by considering non-constant diffusion coefficients in the the dynamics of $X^*$. Thanks to Lemma \ref{lem:limitOS}, the question reduces to finding sufficient conditions on the data of the optimal stopping problem \eqref{eq:uOS} that guarantee a Lipschitz stopping boundary. In \eqref{eq:uOS} the dynamics of $\widetilde X$ was obtained from Lemma \ref{lem:convXY} and it corresponds to the dynamics of $X^*$ in the MFG. In the rest of this section we always use such $X^*$.

We make some additional assumptions on the coefficients of the SDE.
\begin{assumption}\label{ass:dX}
We have $x\mapsto a(x,m)$ and $x\mapsto \sigma(x)$ continuously differentiable with $\partial_x\sigma(x)\ge 0$ and $\partial_x a(x,m)\le \bar a$ for some $\bar a>0$.
\end{assumption} 
Thanks to this assumption we have that the stochastic flow $x\mapsto X^{*;\,t,x}(\omega)$ is continuously differentiable. The dynamics of $Z^{t,x}:=\partial_xX^{*;\,t,x}$ is given by \citep[see][Chapter V.7]{protter2005stochastic}
\begin{align}
Z_{t+s}^{t,x} = 1 + \int_{0}^{s} \!\partial_{x} a(X_{t+u}^{*;\,t,x}, m^*(t\!+\!u))Z_{t+u}^{t,x}\ud u+\!\int_{0}^{s}\! \partial_x \sigma(X_{t+u}^{*;\,t,x})Z_{t+u}^{t,x}\,\ud W_{t+u},
\label{eq:dynamic_derivative}
\end{align}
for all $(t,x)\in[0,T]\times\R$ and $s\in[0,T-t]$. The solution of \eqref{eq:dynamic_derivative} is explicit in terms of $X^*$ and it reads
\begin{align}
Z_{t+s}^{t,x} = \exp\Big[\int_{0}^{s}\!\Big(\partial_{x} a(X_{t+u}^{*;\,t,x}, m^*(t+u))\! -\! \frac{1}{2}\partial_{x}\sigma(X_{t+u}^{*;\,t,x})^2\Big)\,\ud u + \int_{0}^{s}\!\partial_{x}\sigma(X_{t+u}^{*;t,x})\,\ud W_{t+u} \Big],
\label{eq:representation_of_Z}
\end{align} 
for $(t,x)\in[0,T]\times\R$ and $s\in[0,T-t]$. Thanks to this explicit formula we can deduce that $(t,x)\mapsto Z^{t,x}$ is a continuous flow, by continuity of the flow $(t,x)\mapsto X^{*;\,t,x}$.

Later on we will perform a change of measure using $Z$ and for that we also require:
\begin{assumption}\label{ass:Z}
For all $(t,x)\in[0,T]\times\R$ we have
\begin{align}\label{ass:Z2}
\E_{t,x}\left[\int_0^{T-t}\Big(\partial_x \sigma(X^*_{t+u})Z_{t+u}\Big)^2\ud u\right]<+\infty.
\end{align}
\end{assumption}
Then, letting $Z_T=Z^{0,x}_T$,
\begin{align}\label{eq:Q}
\frac{\ud \Q}{\ud \P}\Big|_{\cF_T}:=Z_T \exp\left(-\int_{0}^{T}\partial_{x} a(X^*_{t}, m^*(t))\ud t\right)
\end{align}
defines the Radon-Nikodym derivative of the absolutely continuous change of measure from $\P$ to $\Q$. 

Next we assume some extra conditions on the profit function. Let $\Sigma^\circ=\R\times(0,1)$.
\begin{assumption}\label{ass:f3}
We have $f\in C^2(\Sigma^\circ)$ and either $\sigma(x)=\sigma$ is constant or we have $x\mapsto \partial_{xy}f(x,y)$ non-increasing. Moreover, the integrability condition below holds:
\begin{eqnarray*}
\sup_{(t,x,y)\in K}\E_{t,x}\left[\int_0^{T-t}\!\!\!e^{-rs}\Big(\left|\partial_{yy}f(X^*_{t+s},y)\right|+(1+Z_{t+s})\left|\partial_{xy}f(X^*_{t+s},y)\right|\Big)\ud s\right]<\infty,
\end{eqnarray*}
for any compact $K\subset [0,T]\times\Sigma^\circ$.
\end{assumption}
Notice that $f(x,y)=x^\alpha y^\beta$ with $\alpha\in(0,1]$, $\beta\in(0,1)$ and $x>0$ satisfies Assumption \ref{ass:f3} combined with Assumption \ref{ass:SDE}.
The next proposition provides sufficient conditions for Lipschitz continuity of the optimal boundary.
\begin{proposition}\label{prop:Lip}
Let Assumptions \ref{ass:SDE}--\ref{ass:int} and Assumptions \ref{ass:dX}--\ref{ass:f3} hold. If either of the two conditions below holds:
\begin{itemize}
\item[(i)] there exist $\alpha,\gamma>0$ such that 
\[
|\partial_{yy}f|\ge \alpha>0\quad\text{and}\quad|\partial_{xy}f|\le \gamma (1+|\partial_{yy}f|)\quad\text{on $\Sigma^\circ$};
\]
\item[(ii)] there exists $\gamma>0$ such that $|\partial_{xy}f|\le \gamma|\partial_{yy}f|$ on $\Sigma^\circ$;
\end{itemize}
then Assumption \ref{ass:Lip-c} holds.
\end{proposition}
The proof of the proposition uses the next lemma concerning the optimal stopping time defined in \eqref{eq:tau*}. We move its slightly technical proof to the appendix. 
\begin{lemma}\label{lem:ST}
The mapping $(t,x,y)\mapsto \tau_*(t,x,y)$ is $\P$-almost surely continuous on $[0,T]\times\Sigma^\circ$ with $\tau_*(t,x,y)=0$, $\P$-a.s.~for $(t,x,y)\in\partial\cC$.
\end{lemma}
\begin{proof}[Proof of Proposition \ref{prop:Lip}]
The proof combines ideas from \cite{deangelisStabile} and \cite{deangelisPeskir}. First, for $\delta>0$ we define
\[
c_\delta(t,x):=\inf\{y\in[0,1]: u(t,x,y)<c_0-\delta\}
\]
with $\inf\varnothing=1$. Then it is clear that $c_\delta(\,\cdot\,)>c_{\delta'}(\,\cdot\,)>c(\,\cdot\,)$ for all $0<\delta'<\delta$ by monotonicity of $y\mapsto u(t,x,y)$. Since $u$ is continuous then
\[
\lim_{\delta\downarrow 0}c_\delta(t,x)=c(t,x)\quad(t,x)\in[0,T]\times\R
\]
and if we can prove that $x\mapsto c_\delta(t,x)$ is Lipschitz with a constant independent of $\delta$ we can conclude. By continuity of $u$ we know that 
\[
u\big(t,x,c_\delta(t,x)\big)=c_0-\delta
\]
so that by the implicit function theorem, whose use is justified in step 1 below, we have
\begin{align}\label{eq:IFth}
\partial_x c_\delta(t,x)=-\frac{\partial_{x}u\big(t,x,c_\delta(t,x)\big)}{\partial_y u\big(t,x,c_\delta(t,x)\big)},\quad(t,x)\in[0,T]\times\R.
\end{align}
Thanks to Corollary \ref{cor:reg-u} we have $\partial_x c_\delta(t,x)\ge 0$. In step 2 below we will find an upper bound so that $0\le \partial_x c_\delta\le \theta_c$ on $[0,T]\times\R$, for a suitable constant $\theta_c>0$. This concludes the proof.
\vspace{+4pt}

{\em Step 1:} (Gradient estimates). We fix an arbitrary $(t,x,y)\in[0,T]\times\Sigma^\circ$ and let $\tau_*=\tau_*(t,x,y)$. Then for any small $\eps>0$ we have
\begin{align*}
&u(t,x,y+\eps)-u(t,x,y)\\
&\le \E\left[\int_0^{\tau_*}e^{-rs}\Big(\partial_y f(X^{*;\,t,x}_{t+s},y+\eps)-\partial_y f(X^{*;\,t,x}_{t+s},y)\Big)\ud s\right]\\
&=\int_0^\eps\E\left[\int_0^{\tau_*}e^{-rs}\partial_{yy} f(X^{*;\,t,x}_{t+s},y+z)\ud s\right]\ud z,
\end{align*}
where we used Fubini's theorem in the final equality. Dividing by $\eps$, letting $\eps\to 0$ and using the integrability condition from Assumption \ref{ass:f3} we conclude
\[
\limsup_{\eps\to 0}\frac{u(t,x,y+\eps)-u(t,x,y)}{\eps}\le \E\left[\int_0^{\tau_*}e^{-rs}\partial_{yy} f(X^{*;\,t,x}_{t+s},y)\ud s\right].
\]
Taking $\tau^\eps_*:=\tau_*(t,x,y+\eps)$ in the first expression above we have
\begin{align*}
&u(t,x,y+\eps)-u(t,x,y)\\
&\ge \E\left[\int_0^{\tau^\eps_*}e^{-rs}\Big(\partial_y f(X^{*;\,t,x}_{t+s},y+\eps)-\partial_y f(X^{*;\,t,x}_{t+s},y)\Big)\ud s\right]\\
&=\int_0^\eps\E\left[\int_0^{\tau^\eps_*}e^{-rs}\partial_{yy} f(X^{*;\,t,x}_{t+s},y+z)\ud s\right]\ud z.
\end{align*}
Dividing again by $\eps>0$ and letting $\eps\to 0$, we can now invoke Lemma \ref{lem:ST} to justify that $\tau^\eps_*\to\tau_*$ and obtain
\[
\liminf_{\eps\to 0}\frac{u(t,x,y+\eps)-u(t,x,y)}{\eps}\ge \E\left[\int_0^{\tau_*}e^{-rs}\partial_{yy} f(X^{*;\,t,x}_{t+s},y)\ud s\right].
\]

So, in conclusion we have shown that $\partial_y u$ exists in $[0,T]\times\Sigma^\circ$ and it reads
\[
\partial_y u(t,x,y)=\E\left[\int_0^{\tau_*}e^{-rs}\partial_{yy} f(X^{*;\,t,x}_{t+s},y)\ud s\right].
\]
Further, in light of the fact that $(t,x,y)\mapsto\tau_*(t,x,y)$ and $(t,x,y)\mapsto (X^{*;\,t,x}_{t+s},y)$ are $\P$-a.s.~continuous, we deduce that $\partial_y u$ is also continuous on $[0,T]\times\Sigma^\circ$, by dominated convergence and Assumption \ref{ass:f3}. Finally, since $\partial_{yy}f<0$ (Assumption \ref{ass:f}-(ii)), we have that 
\begin{align}\label{eq:uy<0}
\partial_{y} u\big(t,x,c_\delta(t,x)\big)<0,\quad \text{for all $(t,x)\in[0,T)\times \R$,}
\end{align}
for which $\big(t,x,c_\delta(t,x)\big)\in[0,T)\times\Sigma^\circ$, since $\tau_*>0$ at those points.

Next we obtain a similar result for $\partial_x u$. With the same notation as above we have
\begin{align*}
&u(t,x+\eps,y)-u(t,x,y)\\
&\le 
\E\left[\int_0^{\tau_*}e^{-r s}\Big(\partial_y f(X^{*;\,t,x+\eps}_{t+s},y)-\partial_y f(X^{*;\,t,x}_{t+s},y)\Big)\ud s\right]\\
&=\int_x^{x+\eps}\E\left[\int_0^{\tau_*}e^{-r s}\partial_{xy} f(X^{*;\,t,\eta}_{t+s},y)Z^{t,\eta}_{t+s}\ud s\right]\ud \eta.
\end{align*}
Dividing by $\eps$ and letting $\eps\to 0$, we use dominated convergence (Assumption \ref{ass:f3}) and continuity of the flows $x\mapsto\big(X^{*;\,t,x},Z^{t,x}\big)$ to conclude 
\[
\limsup_{\eps\to 0}\frac{u(t,x+\eps,y)-u(t,x,y)}{\eps}\le \E\left[\int_0^{\tau_*}e^{-r s}\partial_{xy} f(X^{*;\,t,x}_{t+s},y)Z^{t,x}_{t+s}\ud s\right].
\]
By a symmetric argument and continuity of the optimal stopping time we also obtain the reverse inequality and therefore conclude
\[
\partial_x u(t,x,y)= \E\left[\int_0^{\tau_*}e^{-r s}\partial_{xy} f(X^{*;\,t,x}_{t+s},y)Z^{t,x}_{t+s}\ud s\right].
\]
Also in this case continuity of $(t,x,y)\mapsto \tau_*(t,x,y)$, due to Lemma \ref{lem:ST}, and $(t,x)\mapsto\big(X^{*;\,t,x},Z^{t,x}\big)$, combined with dominated convergence, imply that $\partial_x u$ is continuous on $[0,T]\times \Sigma^\circ$.

Since $\partial_yu$ and $\partial_xu$ are continuous and \eqref{eq:uy<0} holds, the equation in \eqref{eq:IFth} is fully justified as an application of the implicit function theorem.
\vspace{+4pt}

{\em Step 2:} (Bound on $\partial_x c_\delta$). In this step we use the probabilistic representations of $\partial_x u$ and $\partial_y u$ to obtain an upper bound on $\partial_x c_\delta$. First we recall the change of measure induced by $Z$ (see \eqref{eq:Q}) and we use it to write
\[
\partial_x u(t,x,y)= \E^{\Q}_{t,x}\left[\int_0^{\tau_*}e^{-r s+\int_0^s a(X^*_{t+u},m^*(t+u))\ud u}\partial_{xy} f(X^{*}_{t+s},y)\ud s\right].
\]
We want to find an upper bound for $\partial_x u$ in terms of the process under the original measure $\P$.
Under the measure $\Q$ we have, by Girsanov theorem, that $X^*$ evolves according to
\[
\ud X^*_{t+s}=\big[a(X^*_{t+s},m^*(t+s))+\sigma(X^*_{t+s})\partial_x\sigma(X^*_{t+s})\big]\ud s+\sigma(X^*_{t+s})\ud W^\Q_{t+s},
\]
where $W^\Q_{t+s}=W_{t+s}-\int_0^s\partial_x\sigma(X^*_{t+u})\ud u $ defines a Brownian motion under $\Q$.
Analogously, under the original measure $\P$ we can define a process $\bar X$ with the same dynamics, i.e.,
\[
\ud \bar X_{t+s}=\big[a(\bar X_{t+s},m^*(t+s))+\sigma(\bar X_{t+s})\partial_x\sigma(\bar X_{t+s})\big]\ud s+\sigma(\bar X_{t+s})\ud W_{t+s},
\]
and denote
\[
\bar\tau_*:=\inf\{s\in[0,T-t]: c(t+s,\bar X_{t+s})\ge y\}.
\]
Then we have that the processes and stopping times are equal in law, i.e.
\[
\text{Law}^\Q\big(X^*,\tau_*\big)=\text{Law}^\P\big(\bar X,\bar \tau_*\big)
\]
and we can express $\partial_x u$ in terms of the original measure as
\begin{align}\label{eq:lip-p}
\partial_x u(t,x,y)= \E_{t,x}\left[\int_0^{\bar \tau_*}e^{-r s+\int_0^s a(\bar X_{t+u},m^*(t+u))\ud u}\partial_{xy} f(\bar X_{t+s},y)\ud s\right].
\end{align}

Let us first consider $x\mapsto \sigma(x)$ not constant. By comparison principles we have $\bar X\ge X^*$, $\P$-a.s., since  $\partial_x \sigma\ge 0$ (Assumption \ref{ass:dX}). Therefore $\partial_{xy} f(\bar X,y)\le \partial_{xy} f(X^*,y)$, $\P$-a.s., by Assumption \ref{ass:f3}. Since $x\mapsto c(t,x)$ is non-decreasing as pointwise limit of non-decreasing functions (recall Proposition \ref{prop:OSboundary}), then $c(t+s,\bar X_{t+s})\ge c(t+s,X^*_{t+s})$, hence implying $\bar \tau_*\le \tau_*$, $\P$-a.s. Recalling that $\partial_{xy}f>0$ from Assumption \ref{ass:f} and combining these few facts we have
\[
\partial_x u(t,x,y)\le e^{\bar a(T-t)}\E_{t,x}\left[\int_0^{\tau_*}e^{-r s}\partial_{xy} f(X^*_{t+s},y)\ud s\right],
\]
where we also used $\partial_x a\le \bar a$ (Assumption \ref{ass:dX}). If instead $\sigma(x)=\sigma$ is constant then $X^*=\bar X$ by uniqueness of the SDE and therefore the above estimate follows directly from \eqref{eq:lip-p}. Plugging this bound into \eqref{eq:IFth} and recalling that $\partial_{yy}f<0$ we obtain
\[
0\le \partial_xc_\delta(t,x)\le e^{\bar a(T-t)}\frac{\E_{t,x}\left[\int_0^{\tau_*}e^{-r s}\partial_{xy} f\big(X^*_{t+s},c_\delta(t,x)\big)\ud s\right]}{\E_{t,x}\left[\int_0^{\tau_*}e^{-r s}\big|\partial_{yy} f\big(X^*_{t+s},c_\delta(t,x)\big)\big|\ud s\right]}.
\]
Now, if condition (i) holds we obtain
\[
0\le \partial_xc_\delta(t,x)\le e^{\bar a(T-t)}\left(\frac{\gamma}{\alpha}+\gamma\right), \quad\text{for all $(t,x)\in[0,T]\times\R$ and any $\delta>0$,}
\]
whereas if condition (ii) holds we obtain
\[
0\le \partial_xc_\delta(t,x)\le e^{\bar a(T-t)}\gamma, \quad\text{for all $(t,x)\in[0,T]\times\R$ and any $\delta>0$}.
\]
So in the first case Assumption \ref{ass:Lip-c} holds with 
\[
\theta_c=e^{\bar a T}\left(\frac{\gamma}{\alpha}+\gamma\right),
\]
and in the second case with $\theta_c=\gamma \exp (\bar a T)$.
\end{proof}
Next we provide a couple of examples meeting the requirements of Proposition \ref{prop:Lip}.

\begin{example}[Ornstein-Uhlenbeck dynamics with exponential-Cobb-Douglas profit]
Let $a(x,m):= \alpha (m-x)$ for some $\alpha>0$ and $\sigma(x)\equiv \sigma$ for some $\sigma>0$. Given a Borel function $m:[0,T]\to[0,1]$ the dynamics of $X$ from \eqref{eq:dynamics_mfg_01} reads 
\begin{equation}
X_t=X_0+\int_0^t\alpha (m(s)-X_s)\ud s+\sigma W_t,\quad t\in[0,T].
\end{equation}
Let $f(x,y):= e^xy^\beta$ for some $\beta\in(0,1)$ and for all $(x,y)\in\Sigma$. Finally assume that $\E[\exp (qX_0)]<\infty$ for some $q\ge 1$.

We check the assumptions of Proposition \ref{prop:Lip}. Assumptions \ref{ass:SDE} and Assumption \ref{ass:dX} on the coefficients of the SDE are trivially satisfied.
The profit function $f$ has the monotonicity required by Assumption \ref{ass:f}-(i) and it is strictly concave (Assumption \ref{ass:f}-(ii)).  Also, $\partial_{xy}f(x,y)= \beta e^xy^{\beta-1}>0$ (Assumption \ref{ass:f}-(iii)) and \eqref{ass:partial-y-f} is satisfied since
\begin{eqnarray*}
\lim_{x\to -\infty} \frac{\beta e^x}{y^{1-\beta}}=0< rc_0<\lim_{x\to \infty} \frac{\beta e^x}{y^{1-\beta}}=+\infty
\end{eqnarray*}
for any $y\in(0,1]$ fixed. The integrability Assumption \ref{ass:int} is satisfied by the Ornstein-Uhlenbeck dynamics with initial condition as above. Assumption \ref{ass:Z} is trivially satisfied since $\partial_x \sigma\equiv 0$. Assumption \ref{ass:f3} holds because $\sigma$ is constant and the integrability condition is not difficult to check.
Finally Assumption (ii) in Proposition \ref{prop:Lip} is satisfied with any $\gamma\ge \frac{1}{1-\beta}$ since
\begin{equation*}
|\partial_{xy}f(x,y)|=\frac{\beta e^x}{y^{1-\beta}}\quad\text{and}\quad
|\partial_{yy}f(x,y)|=\frac{\beta(1-\beta)e^x}{y^{2-\beta}}.
\nonumber
\end{equation*}
We also notice that Assumption \ref{ass:f2} holds so that Theorem \ref{teo:approximation} can be applied.
\end{example}

\begin{example}[GBM dynamics with linear-Cobb-Douglas profit]
Let $a(x,m):= \alpha m x $ for some $\alpha>0$ and $\sigma(x):=\sigma x$ for some $\sigma>0$. Given a Borel function $m:[0,T]\to[0,1]$, the dynamics of $X$ from \eqref{eq:dynamics_mfg_01} reads 
\begin{equation}
X_t=X_0+\int_0^t\alpha X_s m(s)\ud s+\int_0^t\sigma X_s \ud W_s,\quad t\in[0,T].
\end{equation}
Non-negativity of the trajectories reduces the state space $\Sigma$ to $(0,\infty)\times[0,1]$ (see Remark \ref{rem:statespace}). Let $f(x,y):= (1+x)(1+y)^\beta$ for some $\beta\in(0,1)$ and for all $(x,y)\in\Sigma$. Finally assume that $\nu\in\mathcal{P}_2(\Sigma)$ and that $rc_0>\beta$.

Let us check the assumptions of Proposition \ref{prop:Lip}. Assumptions \ref{ass:SDE} and Assumption \ref{ass:dX} on the coefficients of the SDE are trivially satisfied. The profit function $f$ has the monotonicity required by Assumption \ref{ass:f}-(i) and it is strictly concave (Assumption \ref{ass:f}-(ii)). Also, $\partial_{xy}f(x,y)= \beta(1+y)^{\beta-1}>0$ (Assumption \ref{ass:f}-(iii)) and Eq.\ \eqref{ass:partial-y-f} is satisfied since
\begin{eqnarray*}
\lim_{x\to 0} \frac{\beta(1+x)}{(1+y)^{1-\beta}}=\frac{\beta}{(1+y)^{1-\beta}}\le\beta< rc_0<\lim_{x\to \infty} \frac{\beta(1+x)}{(1+y)^{1-\beta}}=+\infty
\end{eqnarray*}
for any $y\in[0,1]$ fixed. The integrability Assumption \ref{ass:int} is satisfied with $p=2$ (or higher provided the initial condition has finite $p$-th moment) thanks to sub-linearity of the logarithm and standard estimates on the GBM dynamics. Assumption \ref{ass:Z} is another integrability assumption that reduces to finiteness of the second moment of the exponential martingale $Z$ (which is satisfied by boundedness of $\partial_x a$ and $\partial_x\sigma$).
Assumption \ref{ass:f3} holds because $x\mapsto\partial_{xy}f(x,y)$ is non-increasing.  Finally Assumption (ii) in Proposition \ref{prop:Lip} is satisfied with any $\gamma\ge \frac{2}{1-\beta}$ since
\begin{equation*}
|\partial_{xy}f(x,y)|=\frac{\beta}{(1+y)^{1-\beta}}\quad\text{and}\quad
|\partial_{yy}f(x,y)|=\frac{\beta(1-\beta)(1+x)}{(1+y)^{2-\beta}}.
\nonumber
\end{equation*}
We also notice that Assumption \ref{ass:f2} holds so that Theorem \ref{teo:approximation} can be applied.
\end{example}

We would like to emphasise that the conditions of Proposition \ref{prop:Lip} are far from being necessary. While it would be overly complicated to state a general theorem in this sense, we provide below an example with a clear economic interpretation for which Proposition \ref{prop:Lip} is not directly applicable.

\subsection{A model with Geometric Brownian Motion}\label{subsec:gbm}
Let us assume that $a(x,m)=(\mu + m)x$ and $\sigma(x)=\sigma\,x$ for some $\mu\in\R$ and $\sigma\in\R_+$. Let us also assume that $f(x,y)= x g(y)$ with $g\in C^2([0,1])$, $g> 0$, strictly concave and strictly increasing. This specification corresponds to the classical model of the goodwill problem in which firms produce a good whose price evolves as a geometric Brownian motion and revenues are linear in the price of the good and increasing and concave in the amount of investment that goes towards advertising.

On the one hand, Assumptions \ref{ass:SDE}--\ref{ass:int} are easily verified  and Theorem \ref{teo:existenceSolMFG} holds (i.e., our construction of the solution to the MFG holds). On the other hand, neither (i) or (ii) in Proposition \ref{prop:Lip} hold, so we cannot apply directly the result on Lipschitz continuity of the boundary which is needed for the approximation result in Theorem \ref{teo:approximation}. However, we shall now see how an alternative argument of proof can be applied to prove that Assumption \ref{ass:Lip-c} remains valid.

First of all we change our coordinates by letting $\psi:=\ln x$, so that the value function of the optimal stopping problem can be written as
\[
\widetilde u(t,\psi,y):=u(t,e^\psi,y)=\inf_{\tau\in\T_t}\E_{t,\psi}\left[\int_0^\tau e^{-rs}g'(y)e^{\Psi_{t+s}}\ud s+e^{-r\tau}c_0\right],
\]
where $\Psi_{t+s}:=\ln X^*_{t+s}$ is just a Brownian motion with drift, i.e., 
\[
\Psi^{t,\psi}_{t+s}=\psi+\int_0^s\big(\mu-\sigma^2/2+m^*(t+u)\big)\ud u+\sigma (W_{t+s}-W_t)
\]
The optimal boundary can also be expressed in terms of $(t,\psi)$ by putting $\widetilde c(t,\psi)=c(t,e^\psi)$. Then the mean-field optimal control reads
\[
\xi^*_{t}=\sup_{0\le s\le t}\Big(\widetilde c(s,\Psi_s)-y\Big)^+,\quad t\in[0,T]
\]
whereas the optimal stopping time for the value $\widetilde u (t,\psi,y)$ reads
\[
\tau_*=\inf\{s\in[0,T-t]:\, \widetilde c\,(t+s,\Psi_{t+s})\ge y\}.
\]
Now we show that the optimal boundary $\widetilde c(\,\cdot\,)$ is indeed Lipschitz with respect to $\psi$ and therefore the proof of Theorem \ref{teo:approximation} can be repeated with $\Psi$ instead of $X^*$ so that the theorem holds as stated. Since $\partial_\psi \Psi^{t,\psi}_{t+s}\equiv 1$ for $s\in[0,T-t]$ and Assumption \ref{ass:f3} holds, we can use the same arguments as in step 1 of the proof of Proposition \ref{prop:Lip} to obtain
\[
\partial_y \widetilde u(t,\psi,y)=g''(y)\E_{t,\psi}\left[\int_0^{\tau_*}e^{-rs +\Psi_{t+s}}\ud s\right]
\]
and
\[
\partial_\psi \widetilde u(t,\psi,y)=g'(y)\E_{t,\psi}\left[\int_0^{\tau_*}e^{-rs +\Psi_{t+s}}\ud s\right].
\]
Then, by the same arguments as in step 2 of the proof of Proposition \ref{prop:Lip} we obtain
\begin{align}\label{eq:ctilde}
\partial_\psi \widetilde c_\delta(t,\psi)=-\frac{\partial_x\widetilde w\big(t,\psi,\widetilde c_\delta(t,\psi)\big)}{\partial_y\widetilde w\big(t,\psi,\widetilde c_\delta(t,\psi)\big)}=\frac{g'\big(\widetilde c_\delta(t,\psi)\big)}{|g''\big(\widetilde c_\delta(t,\psi)\big)|}\le \kappa, 
\end{align}
for some $\kappa>0$, where the final inequality holds because $g\in C^2([0,1])$ and strictly concave. Therefore for the optimal boundary we have
\[
\sup_{0\le t\le T}|\widetilde c(t,\psi_1)-\widetilde c(t,\psi_2)|\le \kappa|\psi_1-\psi_2|,\quad\psi_1,\psi_2\in\R,
\]
as needed. In conclusion, the result of Theorem \ref{teo:approximation} remains valid, even though the optimal boundary in the original parameterisation of the problem is not uniformly Lipschitz.

\section{Appendix}
In this appendix we collect a number of technical results used in the paper.
\vspace{+3pt}

\noindent{\bf Proof of Proposition \ref{th:theoremConnection}}.
Fix $(t,x,y) \in [0,T] \times \Sigma$. Take any admissible control $\zeta\in\Xi_{t,x}(y)$ and define, for $q \geq 0$, its right-continuous inverse (see, e.g., \cite[Ch.~0, Sec.~4]{revuz2013continuous}) $\tau^{\zeta}(q):=\inf\{s \in [t,T): \,\,\zeta_{s} > q\} \wedge T$. The process $\tau^{\zeta} :=(\tau^{\zeta}(q))_{q \geq 0}$ has increasing right-continuous sample paths, hence it admits left limits $\tau_{-}^{\zeta}(q):=\inf\left\{s \in [t,T): \,\zeta_{s} \geq q\right\} \wedge T$, for $q \geq 0$. It can be shown that both $\tau^{\zeta}(q)$ and $\tau^{\zeta}_{-}(q)$ are $(\mathcal{F}_{t+s})$-stopping times for any $q \geq 0$.

Let now $q = z - y$ for $z \geq y$ and consider the function $w$ defined as 
\begin{align}
w(t, x,y)&:=\Phi_n(t, x)  -\int_{y}^{1} u_n(t, x, z)\,\ud z.
\end{align}
Since $\tau^\zeta(z-y)$ is admissible for $u_n(t,x,z)$ we have
\begin{align*}
&w(t, x,y) - \Phi_n(t, x)
&\geq - \int_{y}^{1} \E_{t,x}\left[ c_0  e^{-r  \tau^{\zeta}(z-y)}+\int_{t}^{\tau^{\zeta}(z-y)} e^{-r s}\partial_y f(X^{[n]}_s,z)\ud s\right]\ud z.
\end{align*}
In order to compute the integral with respect to $\ud z$ we observe that for $t\le s<T$ we have 
\[
\{\zeta_{s} < z- y \}\subseteq\{ s < \tau^{\zeta}(z-y)\} \subseteq \{ \zeta_{s}\leq z - y\}
\]
by right-continuity and monotonicity of the process $s\mapsto \zeta_{s}$. The left-most and right-most events above are the same up to $\ud z$-null sets. Then, applying Fubini's theorem more than once we obtain
\begin{align*}
&w(t, x,y) - \Phi_n(t, x)  \nonumber\\
&\geq \E_{t,x}\left[ - \int_{y}^{1}  e^{-r  \tau^{\zeta}(z-y)} c_0 \,\ud z  -  \int_{t}^{T}  e^{-r s}\int_{y}^{1}\partial_y f(X^{[n]}_s, z){\mathbf 1}_{\{s<\tau^{\zeta}(z-y)\}}\,\ud z\,\ud s \right] \nonumber\\
& =  \E_{t, x}\left[- \int_{y}^{1} e^{-r  \tau^{\zeta}(z-y)} c_0 \,\ud z -   \int_{t}^{T}  e^{-r s}\int_{y}^{1}\partial_y f(X^{[n]}_s, z) {\mathbf 1}_{\left\{\zeta_{s} < z - y\right\}}\ud z\,\ud s \right] \nonumber\\
& =  \E_{t, x}\left[- \int_{y}^{1} e^{-r  \tau^{\zeta}(z-y)} c_0 \,\ud z -   \int_{t}^{T}   e^{-r s}[f(X^{[n]}_s,1) - f(X^{[n]}_s, y + \zeta_{s})]\,\ud s \right] \nonumber\\
& =  J_n(t, x, y; \zeta) - \Phi_n(t, x),
\nonumber
\end{align*}
where the final equality uses the well-known change of variable formula \citep[see, e.g.,][Ch. 0, Proposition 4.9]{revuz2013continuous}
\[
\int_{y}^{1} e^{-r  \tau^{\zeta}(z-y)}\ud z=\int_{[t,T]}e^{-rs}\ud\zeta_{s}.
\]
By the arbitrariness of $\zeta\in\Xi_{t,x}(y)$ we conclude $w_n(t, x,y)\geq v_n(t,x,y)$.

For the reverse inequality we take $\zeta_{t+s}=\xi^{[n]*}_{t+s}$ as defined in Lemma \ref{lem:SK}. Recall that 
\[
\tau^{[n]}_*(t,x,z) = \inf\big\{s \in [0,T-T]:\, z \leq c_n(t+s, X_{t+s}^{[n];\,t,x}) \big\}.
\]
and since $s\mapsto c_n(s, X_s^{[n];\,t, x})-z$ is upper semi-continuous, it attains a maximum over any compact interval in $[t,T)$. In particular, for $s\in[t,T)$
\begin{align*}
\text{$\tau^{[n]}_*(t,x,z)\leq s \iff $there exists $\theta\in[t,t+s]$ such that $c_n(\theta, X^{[n];\,t,x}_{\theta})\ge z$}.
\end{align*}
For any $y<z$, the latter is also equivalent to 
\begin{align*}
\xi^{[n]*}_{t+s}= \sup_{0\leq u \leq s}\left(c_n(t+u, X^{[n];\,t,x}_{t+u})- y\right)^{+}  \geq z-y
\end{align*}
and, therefore, it is also equivalent to $\tau_{-}^{\xi^{[n]*}}(z-y) \leq s$. Since $s\in[t,T]$ was arbitrary the chain of equivalences implies that $\tau_{-}^{\xi^{[n]*}}(z-y) = \tau^{[n]}_*(t,x,z)$, $\P$-a.s.~for any $z>y$. However, we have already observed that for a.e.~$z > y$ it must be $\tau^{\xi^{[n]*}}_{-}(z-y) = \tau^{\xi^{[n]*}}(z-y)$, $\P$-a.s., hence $\tau^{\xi^{[n]*}}(z-y) = \tau^{[n]}_*(t,x,z)$ as well. The latter, in particular implies
\begin{align*}
&w(t, x,y) - \Phi_n(t, x)
=- \int_{y}^{1} \E_{t,x}\Big[ c_0  e^{-r  \tau^{\xi^{[n]*}}(z-y)}+\int_{t}^{\tau^{\xi^{[n]*}}(z-y)} e^{-r s}\partial_y f(X^{[n]}_s,z)\ud s\Big]\ud z,
\end{align*}
by optimality of $\tau^{[n]}_*(t,x,z)$ in $u_n(t,x,z)$. Repeating the same steps as above we then find
\[
w(t,x,y)=J_n(t,x,y;\xi^{[n]*}),
\]
which combined with $v_n\le w$ concludes the proof.\hfill$\square$
\vspace{+5pt}

\noindent{\bf Proof of Lemma \ref{lem:tau-n}}.
We have $\cC^{[n]}\subset \cC^{[n+1]}\subset \cC$ because the sequence $(c_n)_{n\ge 0}$ is decreasing. Then, the sequence $(\tau^{[n]}_*)_{n\ge 0}$ is increasing and $\lim_{n\to\infty}\tau^{[n]}_*\le \tau_*$, $\P_{t,x,y}$-a.s.~for any $(t,x,y)\in[0,T]\times\Sigma$. In order to prove the reverse inequality, first we observe that $t\mapsto X^{[k]}_t(\omega)$ is continuous for all $\omega\in\Omega\setminus N_k$ with $\P(N_k)=0$, for all $k\ge 0$. Moreover, $t\mapsto \widetilde X_t(\omega)$ is continuous for all $\omega\in\Omega\setminus N$ with $\P(N)=0$. Let us set $N_0:=(\cup_{k}N_k)\cup N$ and $\Omega_0:=\Omega\setminus N_0$ so that $\P(\Omega_0)=1$. Fix $(t,x,y)\in[0,T]\times\Sigma$ and $\omega\in\Omega_0$. Let $\delta>0$ be such that $\tau_*(\omega)>\delta$ (if no such $\delta$ exists, then $\tau_*(\omega)=0$ and $\tau^{[n]}_*(\omega)\ge \tau_*(\omega)$ for all $n\ge 0$). Then, since $s\mapsto u(t+s,\widetilde X_{t+s}(\omega),y)$ is continuous, there must exist $\eps>0$ such that
\begin{align}\label{eq:st0}
\sup_{0\le s\le \delta}\left(u(t+s,\widetilde X_{t+s}(\omega),y)-c_0\right)\le -\eps.
\end{align}
At the same time we also notice that $s\mapsto u_n(t+s,X^{[n]}_{t+s}(\omega),y)$ is continuous and moreover
\[
u_n(t+s,X^{[n]}_{t+s}(\omega),y)\ge u_{n+1}(t+s,X^{[n]}_{t+s}(\omega),y)\ge u_{n+1}(t+s,X^{[n+1]}_{t+s}(\omega),y)
\]
by monotonicity of the sequences $(u_n)_{n\ge 0}$ and $(X^{[n]})_{n\ge 0}$ and of the map $x\mapsto u_n(t,x,y)$. So we have that $u_n(t+\cdot,X^{[n]}_\cdot(\omega),y)$ is a decreasing sequence of continuous functions of time and since the limit is also continuous, the convergence is uniform on $[0,\delta]$. Then, there exists $n_0\ge 0$ sufficiently large that
\[
\sup_{0\le s\le \delta}\left|u(t+s,\widetilde X_{t+s}(\omega),y)-u_n(t+s,X^{[n]}_{t+s}(\omega),y)\right|\le -\frac{\eps}{2}, \quad\text{for $n\ge n_0$}.
\]
Using this fact and \eqref{eq:st0} we have
\begin{align*}
\sup_{0\le s\le \delta}\left(u_n(t+s,X^{[n]}_{t+s}(\omega),y)-c_0\right)\le -\frac{\eps}{2}
\end{align*}
and $\tau^{[n]}_*(\omega)>\delta$, for all $n\ge n_0$. Since $\delta>0$ was arbitrary, we obtain
\[
\lim_{n\to\infty}\tau^{[n]}_*(\omega)\ge \tau_*(\omega).
\]
Recalling that $\omega\in\Omega_0$ was also arbitrary we obtain the desired result.\hfill$\square$
\vspace{+5pt}

\noindent{\bf Proof of Lemma \ref{lem:ST}}.
The proof is divided into two steps: first we show that $(t,x,y)\mapsto \tau_*(t,x,y)$ is lower semi-continuous and then that it is upper semi-continuous. Both parts of the proof rely on continuity of the flow $(t,x,s)\mapsto X^{*;\,t,x}_{t+s}(\omega)$. The latter holds for all $\omega\in\Omega\setminus N$ where $N$ is a universal set with $\P(N)=0$. For simplicity, in the rest of the proof we just write $X$ instead of $X^*$.
\vspace{+5pt}

{\em Step 1.} (Lower semi-continuity). This part of the proof is similar to that of Lemma \ref{lem:tau-n}. Fix $(t,x,y)\in[0,T]\times\Sigma$ and take a sequence $(t_n,x_n,y_n)_{n\ge 1}$ that converges to $(t,x,y)$ as $n\to \infty$. Denote $\tau_*=\tau_*(t,x,y)$ and $\tau_n:=\tau_*(t_n,x_n,y_n)$ and fix an arbitrary $\omega\in\Omega\setminus N$. If $(t,x,y)\in \cS$ then $\tau_*(\omega)=0$ and $\liminf_n \tau_n(\omega)\ge \tau_*(\omega)$ trivially. Let $\delta>0$ be such that $\tau_*(\omega)>\delta$. Then by continuity of the value function $u$ and of the trajectory $s\mapsto X^{t,x}_{t+s}(\omega)$ there must exist $\eps>0$ such that 
\[
\sup_{0\le s\le \delta}\big(u(t+s,X^{t,x}_{t+s}(\omega),y)-c_0\big)\le -\eps.
\]
Thanks to continuity of the stochastic flow there is no loss of generality in assuming that $(t_n+s,X^{t_n,x_n}_{t_n+s}(\omega),y_n)$ lies in a compact $K$ for all $n\ge 1$ and $s\le \delta$. Then there must exits $n_\eps>0$ such that 
\[
\sup_{0\le s\le \delta}\big|u(t+s,X^{t,x}_{t+s}(\omega),y)-u(t_n+s,X^{t_n,x_n}_{t_n+s}(\omega),y)\big|\le \eps/2
\]
for all $n\ge n_\eps$ (by uniform continuity). Combining the above we get 
\[
\sup_{n\ge n_\eps}\sup_{0\le s\le \delta}\big(u(t_n+s,X^{t_n,x_n}_{t_n+s}(\omega),y_n)-c_0\big)\le -\eps/2,
\]
which implies $\tau_n(\omega)>\delta$ for all $n\ge n_\eps$. Hence $\liminf_n\tau_n(\omega)>\delta$ and since $\delta$ and $\omega$ were arbitrary we conclude $\liminf_n\tau_n(\omega)\ge\tau_*(\omega)$, for all $\omega\in\Omega\setminus N$.
\vspace{+4pt}

{\em Step 2.} (Upper semi-continuity). For this part of the proof we need to introduce the {\em hitting time} $\sigma_*^\circ$ to the interior of the stopping set $\cS^\circ:=\text{int}(\cS)$ (which is not empty thanks to the argument of proof of Lemma \ref{lem:limitOS}), i.e.,
\[
\sigma^\circ_*(t,x,y):=\inf\{s\in(0,T-t]\,:\, (t+s,X^{t,x}_{t+s},y)\in\cS^\circ\}.
\]
Assume for a moment that
\begin{align}\label{usc}
\P_{t,x,y}(\tau_*=\sigma^\circ_*)=1\quad\text{for all $(t,x,y)\in[0,T]\times\Sigma^\circ$}.
\end{align} 
Then we can invoke Lemma 4 in \cite{deangelisPeskir} (see Eq.~(3.7) therein) to conclude that $(t,x,y)\mapsto \sigma^\circ_*(t,x,y)$ is upper semi-continuous. Hence, given $(t,x,y)\in[0,T]\times\Sigma^\circ$ and any sequence $(t_n,x_n,y_n)_{n\ge 1}$ converging to $(t,x,y)$ as $n\to\infty$, setting $\tau_n=\tau_*(t_n,x_n,y_n)$ and $\sigma_n^\circ=\sigma^\circ_*(t_n,x_n,y_n)$, we have $\tau_n(\omega)=\sigma^\circ_n(\omega)$ for all $\omega\in\Omega_n$ with $\P(\Omega_n)=1$ for each $n\ge 1$; therefore letting $\bar\Omega:=\cap_{n\ge 1}\Omega_n$ we have $\P(\bar \Omega)=1$ and 
\[
\limsup_{n\to\infty}\tau_n(\omega)=\limsup_{n\to\infty}\sigma^\circ_n(\omega)\le \sigma_*^\circ(\omega)=\tau_*(\omega),
\]
with $\tau_*=\tau_*(t,x,y)$ and $\sigma^\circ_*=\sigma^\circ_*(t,x,y)$, for all $\omega\in\bar\Omega\cap\{\tau_*=\sigma^\circ_*\}$ where $\P(\bar\Omega\cap\{\tau_*=\sigma^\circ_*\})=1$.

Let us now prove \eqref{usc}. We introduce the generalised left-continuous inverse of $x\mapsto c(t,x)$, i.e.
\[
b(t,y)=\sup\{x\in\R: c(t,x)<y\}.
\]
Then it is easy to check that $t\mapsto b(t,y)$ is non-increasing. This implies that $\P_{t,x,y}(\tau_*=\sigma^\circ_*)=1$ for all $(t,x,y)\in\cS^\circ$ by continuity of the paths of $X$. Moreover, for $(t,x,y)\in\cC$ we have $\sigma^\circ_*=\tau_*+\sigma^\circ_*\circ\theta_{\tau_*}$, where $\{\theta_t,\,t\ge 0\}$ is the shift operator, i.e., $(t,X_t(\omega))\circ\theta_s=(t+s,X_{t+s}(\omega))$.  Then, $\tau_*=\sigma^\circ_*$ if and only if $\sigma^\circ_*\circ\theta_{\tau_*}=0$. Since $\sigma^\circ_*\circ\theta_{\tau_*}$ is the hitting time to $\cS^\circ$ after the process $(t,X,y)$ has reached the boundary $\partial\cC$ of the continuation set, the previous condition is implied by $\P_{t,x,y}(\sigma^\circ_*=0)=1$ for $(t,x,y)\in\partial\cC$. So we now focus on proving the latter.

We claim that
\[
\cS^\circ=\{(t,x,y) \in(0,T)\times\Sigma^\circ:\,x>b(t,y)\}
\]
and will give a proof of this fact in Lemma \ref{lem:app} below.
Then by the law of iterated logarithm and non-increasing $t\mapsto b(t,y)$ we immediately obtain $\P_{t,x,y}(\sigma^\circ_*=0)=1$ for $(t,x,y)\in\partial\cC$ because $(t,x,y)\in\partial\cC \cap[(0,T)\times\Sigma^\circ]$ if and only if $x\ge b(t,y)$.\hfill$\square$

\begin{lemma}\label{lem:app}
We have 
\begin{align}\label{claim1}
\cS^\circ=\{(t,x,y) \in(0,T)\times\Sigma^\circ:\,x>b(t,y)\}.
\end{align}
\end{lemma}
\begin{proof}
While the claim may seem obvious, since $y\mapsto b(t,y)$ is non-decreasing, one should notice that for it to hold we must rule out the case $b(t,y)<b(t,y+)$ for all $(t,y)\in[0,T]\times[0,1)$. Indeed, if the latter occurs for some $(t_0,y_0) \in(0,T)\times(0,1)$ we have $\{t_0\}\times (b(t_0,y_0),b(t_0,y_0+))\times\{y_0\}\in\partial\cC$ and \eqref{claim1} fails. 

We proceed in two steps.
\vspace{+4pt}

{\em Step 1}. (A PDE for the value function). Since $(t,x)\mapsto a(x,m^*(t))$ is not continuous in general we cannot immediately apply standard PDE arguments that guarantee that 
\begin{align}\label{eq:PDE-C}
\partial_t u+\frac{\sigma^2(\,\cdot\,)}{2}\partial_{xx} u + a(\,\cdot,m^*(\,\cdot\,))\partial_x u-r u =-\partial_y f,\quad\text{for $(t,x,y)\in\cC$}
\end{align}
\citep[see][Chapter III]{peskirShyriyaev}. However, we show that for any $y\in[0,1]$, $u(\cdot,y)$ solves \eqref{eq:PDE-C} in the a.e.\ sense. Fix $(t,x,y)\in\cC$ and let $\cO$ be an open rectangle in $[0,T]\times\R$ containing $(t,x)$ with parabolic boundary $\partial_P\cO$ such that $\overline{\cO}\times\{y\}\subset \cC$. Let
\begin{align}\label{eq:tauO}
\quad\tau_{\cO}=\inf\{s\ge 0:(t+s,X^*_{t+s})\notin\cO\}\ \ \text{and}\ \ \tau_{\overline\cO}=\inf\{s\ge 0:(t+s,X^*_{t+s})\notin\overline\cO\}.
\end{align} 
Since $\inf_{(t,x)\in\overline\cO}\sigma(x)>0$, then it is not difficult to verify that $\P_{t,x}(\tau_{\cO}=\tau_{\overline \cO})=1$ for all $(t,x)\in\cO$.

Let us now regularise the drift of $X^*$ by introducing
\[
a^n(x,t)=n\int_{t-1/n}^t a(x,m^*(s))\ud s.
\]
Notice that $|a^n(x,t)-a^n(x',t)|\le L|x-x'|$ with the same constant $L$ as in Assumption \ref{ass:SDE}. Moreover, $a^n$ is locally bounded on any compact by a constant independent of $n$.  
By Lipschitz continuity of $a$ it is immediate to verify
\begin{align}\label{eq:an}
\sup_{x\in\R}\big|a^n(x,t)-a(x,m^*(t))\big|\le L\ n\int_{t-1/n}^t\big|m^*(s)-m^*(t)\big|\ud s,
\end{align}
so that $a^n(\cdot,t)\to a(\cdot,m^*(t))$ when $n \to \infty$, for each $t$ and uniformly in $x$, by left-continuity of $m^*$. Denoting by $X^n$ the unique strong solution of \eqref{eq:dynamics_mfg_03} with the drift $a(\cdot)$ replaced by $a^n(\cdot)$, thanks to \eqref{eq:an} we have 
\begin{align*}
\lim_{n\to\infty}\E_{t,x}\Big[\sup_{0\le s\le T-t}\big|X^n_{t+s}-X^*_{t+s}\big|^2\Big]=0.
\end{align*} 
Then, there is a subsequence $(n_j)_{j\in\N}$ for which 
\begin{align}\label{eq:XnX}
\lim_{j\to\infty}\sup_{0\le s\le T-t}\big|X^{n_j}_{t+s}-X^*_{t+s}\big|^2=0,\quad \P_{t,x}-a.s.
\end{align}
Letting $\tau^n_{\cO}$ and $\tau^n_{\overline\cO}$ be the analogues of \eqref{eq:tauO} but with $X^n$ instead of $X^*$, we have 
\begin{align}\label{eq:taunO}
\tau_\cO\le \liminf_{j\to\infty}\tau^{n_j}_\cO\le \limsup_{j\to\infty}\tau^{n_j}_{\overline\cO}\le \tau_{\overline\cO}=\tau_{\cO},\quad\P_{t,x}-a.s.
\end{align}
thanks to \eqref{eq:XnX} and arguments as in the proof of Lemma \ref{lem:ST} (step 1 for lower semi-continuity of $\tau^{n_j}_\cO$ and step 2 for upper semi-continuity of $\tau^{n_j}_{\overline\cO}$). 

Thanks to the improved regularity of $a^n$, for each $y\in[0,1]$ there exists a unique solution $u^n(\cdot,y)\in C^{1,2}(\cO)\cap C(\overline\cO)$ of the boundary value problem
\begin{equation}\label{eq:PDEn}
\begin{split}
&\big(\partial_tu^n+\tfrac{\sigma^2}{2}(\cdot)\partial_{xx}u^n+a^n(\cdot)\partial_xu^n -ru^n \big)(t,x,y)=-\partial_y f(t,x,y),\quad (t,x)\in\cO,\\
&u^n(t,x,y)=u(t,x,y),\quad(t,x)\in\partial_P\cO.
\end{split}
\end{equation}
By Feynman-Kac formula and \eqref{eq:PDEn} we have
\[
u^n(t,x,y)=\E_{t,x}\Big[e^{-r\tau^n_{\cO}}u(t+\tau^n_{\cO},X^n_{t+\tau^n_{\cO}},y)+\int_0^{\tau^n_{\cO}}e^{-rs}\partial_y f(X^n_{t+s},y)\ud s\Big].
\]
Taking limits along the subsequence $(n_j)_{j\in\N}$, using \eqref{eq:XnX}, \eqref{eq:taunO} and dominated convergence, we obtain
\begin{equation}\label{eq:limunj-u}
\begin{split}
&\lim_{j\to\infty}u^{n_j}(t,x,y)\\
&=\E_{t,x}\Big[e^{-r\tau_{\cO}}u(t+\tau_{\cO},X^*_{t+\tau_{\cO}},y)+\int_0^{\tau_{\cO}}e^{-rs}\partial_y f(X^*_{t+s},y)\ud s\Big]\\
&=u(t,x,y),
\end{split}
\end{equation}
where the final equality holds because $u$ is an harmonic function in $\cO$ by standard optimal stopping theory.

Thanks to \cite[Thm.\ 5, Ch.\ 5, Sec.\ 2]{krylov2008lectures}, for any compact $K\subset\cO$ and $p\in[2,\infty)$ there is a constant $c>0$ depending on $K$ and $p$, on $L$ from Assumption \ref{ass:SDE} and on $\|u(\cdot,y)\|_{L^\infty(K)}$, $\|a\|_{L^\infty(K)}$ and $\|\sigma\|_{L^\infty(K)}$, such that $\|u^n(\cdot,y)\|_{W^{1,2;p}(K)}\le c$ for all $n\in\N$. Here $W^{1,2;p}(K)$ is the Sobolev space of functions $\varphi\in L^p(K)$ whose weak derivatives $\partial_t\varphi$, $\partial_x\varphi$, $\partial_{xx}\varphi$ also belong to $L^p(K)$. Then, $u^{n_j}(\cdot,y)\to \bar u(\cdot,y)$ weakly in $W^{1,2;p}(K)$ as $j\to\infty$. By uniqueness of the limit and \eqref{eq:limunj-u} we have $\bar u(\cdot,y)=u(\cdot,y)\in W^{1,2;p}(K)$. Thus, by arbitrariness of $\cO$ and $K$ we have $u(\cdot,y)\in W^{1,2;p}_{\ell oc}(\cC)$. Passing to the (weak) limit in \eqref{eq:PDEn} and using \eqref{eq:an} we have 
\begin{align}\label{eq:PDE}
\big(\partial_t u+\tfrac{\sigma^2(\,\cdot\,)}{2}\partial_{xx} u + a(\,\cdot,m^*(\,\cdot\,))\partial_x u-r u\big)(\cdot,y) =-\partial_y f(\cdot,y),\quad\text{a.e.\ in $\cO$}.
\end{align} 

Finally, we notice the Sobolev compact embedding $W^{1,2;p}_{\ell oc}(\cC)\hookrightarrow C^{0,1;\alpha}_{\ell oc}(\cC)$ for $\alpha=1-3/p$ with $p>3$, where $C^{0,1;\alpha}$ is the class of $\alpha$-H\"older continuous function with H\"older continuous spatial derivative. We will later use that $\partial_x u(\cdot,y)$ is continuous in $\cC$.
\vspace{+4pt}

{\em Step 2}. (Contradiction). Here we use an argument by contradiction inspired to \cite{de2015note}. Assume that there exists $(t_0,y_0)\in(0,T)\times(0,1)$ such that $x^0_1:=b(t_0,y_0)<b(t_0,y_0+)=:x^0_2$. Let $\cO\times(y_0,y_0+\delta)\subset\cC$ and $\{t_0\}\times(x^0_1,x^0_2)\subset\cO$.
Pick an arbitrary $\varphi\in C^{\infty}_c((x^0_1,x^0_2))$, $\varphi \geq 0$ and multiply \eqref{eq:PDE} by $\varphi$. Since $t\mapsto u(t,x,y)$ is non-decreasing we have $\partial_t u(\,\cdot\,,y)\ge 0$ a.e.\ on $\cO$ and therefore, for each $y\in(y_0,y_0+\delta)$ we have 
\begin{align*}
\varphi(\cdot)\left[\frac{\sigma^2(\cdot)}{2}\partial_{xx}u(\cdot,y)  + a(\cdot, m^*(\cdot))\partial_{x}u(\cdot, y)-r u(\cdot, y)\right]\leq -\varphi(\cdot)\partial_{y} f(\cdot, y),\ \text{a.e.\ in $\cO$.}
\end{align*}
By monotonicity of $t\mapsto b(t,y)$, we have $[0,t_0)\times(x^0_1,x^0_2)\times(y_0,y_0+\delta)\subset\cC$. Then, for arbitrary $0\le t< t_0$ and $0< h<t_0-t$, using integration by parts we obtain\footnote{Notice that the argument seems to require $\sigma\in C^2(\R)$. However, given that the final estimate \eqref{final} does not depend on $\sigma$ we can apply a smoothing of $\sigma$, if necessary, and then pass to the limit at the end.}
\begin{align*}
&\frac{1}{h}\int_{t}^{t+h}\int_{x^0_1}^{x^0_2} \left( \tfrac{1}{2}\partial_{xx}\big[\sigma^2(x)\varphi (x)\big]  + a(x, m^*(s))\varphi(x)\partial_x-r\varphi(x) \right)u(s, x, y) \ud x \ud s \\
& \leq - \frac{1}{h}\int_{t}^{t+h}\int_{x^0_1}^{x^0_2} \partial_y f(x, y)\varphi(x)\ud x \ud s. 
\end{align*}
Thanks to continuity of $u(\cdot,y)$, $\partial_x u(\cdot,y)$ and $a(\cdot)$, and right-continuity of $m^*(\cdot)$, letting $h\downarrow 0$ we obtain
\begin{align*}
&\int_{x^0_1}^{x^0_2} \left( \tfrac{1}{2}\partial_{xx}\big[\sigma^2(x)\varphi (x)\big]  + a(x, m^*(t))\varphi(x)\partial_x-r\varphi(x) \right)u(t, x, y) \ud x \ud s \\
& \leq - \int_{x^0_1}^{x^0_2} \partial_y f(x, y)\varphi(x)\ud x \ud s. 
\end{align*}
Recall that $\{t_0\}\times(x^0_1,x^0_2)\times(y_0,y_0+\delta)\subset\cC$. Then we can let $t\uparrow t_0$ and integrate by parts the term with $\partial_x u$ to obtain
\begin{align*}
&\int_{x^0_1}^{x^0_2} \left( \tfrac{1}{2}\partial_{xx}\big[\sigma^2(x)\varphi (x)\big]  - \partial_x\big[a(x, m^*(t_0-))\varphi(x)\big]-r\varphi(x) \right)u(t_0, x, y) \ud x \\
& \leq - \int_{x^0_1}^{x^0_2} \partial_y f(x, y)\varphi(x)\ud x,
\end{align*}
where $m^*(t_0-)=\lim_{t\uparrow t_0}m^*(t)$.
Now, letting $y\downarrow y_0$, using dominated convergence and $u(t_0,x,y_0)=c_0$ for $x\in(x^0_1,x^0_2)$, and undoing the integration by parts we obtain
\begin{equation}\label{final}
\int_{x^0_1}^{x^0_2}(\partial_yf(x,y_0)-rc_0)\varphi(x)\ud x\leq 0.
\end{equation}
Hence $\partial_yf(x,y_0)-rc_0\leq 0$ for all $x\in(x^0_1,x^0_1)$ by arbitrariness of $\varphi\geq 0$ and continuity of $x\mapsto\partial_yf(x,y_0)-rc_0$. However, since $\cS\subseteq [0,T]\times(\Sigma\setminus\cH)$ (recall \eqref{def-H}), then it must be $\partial_yf(x,y_0)-rc_0= 0$ for all $x\in(x^0_1,x^0_1)$, which contradicts $\partial_{xy}f>0$ (Assumption \ref{ass:f}-(iii)).
\end{proof}
{
\bibliographystyle{Chicago}
\bibliography{goodwillBib}

\begin{thebibliography}{}

\bibitem[\protect\citeauthoryear{Baldi}{Baldi}{2017}]{baldi}
Baldi, P. (2017).
\newblock {\em Stochastic {C}alculus: {A}n {I}ntroduction {T}hrough {T}heory
  and {E}xercises}.
\newblock Springer {I}nternational {P}ublishing.

\bibitem[\protect\citeauthoryear{Baldursson and Karatzas}{Baldursson and
  Karatzas}{1996}]{baldursson1996irreversible}
Baldursson, F.~M. and I.~Karatzas (1996).
\newblock Irreversible investment and industry equilibrium.
\newblock {\em Finance Stoch.\/}~{\em 1\/}(1), 69--89.

\bibitem[\protect\citeauthoryear{Basei, Cao, and Guo}{Basei
  et~al.}{2019}]{basei2019nonzero}
Basei, M., H.~Cao, and X.~Guo (2019).
\newblock Nonzero-sum stochastic games with impulse controls.
\newblock {\em arXiv:1901.08085\/}.

\bibitem[\protect\citeauthoryear{Bather and Chernoff}{Bather and
  Chernoff}{1967}]{bather1967sequential}
Bather, J. and H.~Chernoff (1967).
\newblock Sequential decisions in the control of a spaceship.
\newblock In {\em Fifth Berkeley Symposium on Mathematical Statistics and
  Probability}, Volume~3, pp.\  181--207.

\bibitem[\protect\citeauthoryear{Bene\v{s}, Shepp, and Witsenhausen}{Bene\v{s}
  et~al.}{1980}]{benevs1980some}
Bene\v{s}, V., L.~Shepp, and H.~Witsenhausen (1980).
\newblock Some solvable stochastic control problems.
\newblock {\em Stochastics\/}~{\em 4\/}(1), 39--83.

\bibitem[\protect\citeauthoryear{Bertucci}{Bertucci}{2020}]{bertucci2020fokker}
Bertucci, C. (2020).
\newblock Fokker-{P}lanck equations of jumping particles and mean field games
  of impulse control.
\newblock {\em Ann.\ Inst.\ {H}.\ {P}oincar\'e {A}nal.\ Non Lin\'eaire\/}~{\em
  37\/}(5), 1211--1244.

\bibitem[\protect\citeauthoryear{Buratto and Viscolani}{Buratto and
  Viscolani}{2002}]{buratto2002new}
Buratto, A. and B.~Viscolani (2002).
\newblock New product introduction: goodwill, time and advertising cost.
\newblock {\em Math. {M}ethods {O}per. {R}es.\/}~{\em 55\/}(1), 55--68.

\bibitem[\protect\citeauthoryear{Caffarelli and Salsa}{Caffarelli and
  Salsa}{2005}]{caffarelli2005geometric}
Caffarelli, L.~A. and S.~Salsa (2005).
\newblock {\em A geometric approach to free boundary problems}, Volume~68.
\newblock American Mathematical Soc.

\bibitem[\protect\citeauthoryear{Cao and Guo}{Cao and Guo}{2020}]{cao2020mfgs}
Cao, H. and X.~Guo (2020).
\newblock {MFG}s for partially reversible investment.
\newblock {\em Stoch.\ {P}rocess.\ {A}ppl. (In Press)\/}.

\bibitem[\protect\citeauthoryear{Cao, Guo, and Lee}{Cao
  et~al.}{2017}]{cao2017approximation}
Cao, H., X.~Guo, and J.~S. Lee (2017).
\newblock Approximation of mean field games to {N}-player stochastic games,
  with singular controls.
\newblock {\em arXiv:1703.04437\/}.

\bibitem[\protect\citeauthoryear{Cardaliaguet}{Cardaliaguet}{2012}]{cardaliaguet2012notes}
Cardaliaguet, P. (2012).
\newblock Notes from {P}.-{L}. {L}ions' lectures at the {C}oll{\`e}ge de
  {F}rance.
\newblock Technical report.

\bibitem[\protect\citeauthoryear{Cardaliaguet and Hadikhanloo}{Cardaliaguet and
  Hadikhanloo}{2017}]{cardaliaguet2017learning}
Cardaliaguet, P. and S.~Hadikhanloo (2017).
\newblock Learning in mean field games: the fictitious play.
\newblock {\em ESAIM {C}ontrol {O}ptim. {C}alc. {V}ar.\/}~{\em 23\/}(2),
  569--591.

\bibitem[\protect\citeauthoryear{Cardaliaguet and Lehalle}{Cardaliaguet and
  Lehalle}{2018}]{cardaliaguet2018mean}
Cardaliaguet, P. and C.-A. Lehalle (2018).
\newblock Mean field game of controls and an application to trade crowding.
\newblock {\em Math. {F}inanc. {E}con.\/}~{\em 12\/}(3), 335--363.

\bibitem[\protect\citeauthoryear{Carmona and Delarue}{Carmona and
  Delarue}{2013a}]{carmona2013mean}
Carmona, R. and F.~Delarue (2013a).
\newblock Mean field forward-backward stochastic differential equations.
\newblock {\em Electron.\ {C}omm.\ {P}robab.\/}~{\em 18\/}(68), 1--15.

\bibitem[\protect\citeauthoryear{Carmona and Delarue}{Carmona and
  Delarue}{2013b}]{carmona2013probabilistic}
Carmona, R. and F.~Delarue (2013b).
\newblock Probabilistic analysis of mean-field games.
\newblock {\em SIAM {J}. {C}ontrol {O}ptim.\/}~{\em 51\/}(4), 2705--2734.

\bibitem[\protect\citeauthoryear{Carmona and Delarue}{Carmona and
  Delarue}{2018}]{carmona2018probabilistic}
Carmona, R. and F.~Delarue (2018).
\newblock {\em Probabilistic Theory of Mean Field Games with Applications
  I-II}, Volume~83.
\newblock Springer {I}nternational {P}ublishing.

\bibitem[\protect\citeauthoryear{Carmona, Delarue, and Lacker}{Carmona
  et~al.}{2016}]{carmona2016mean}
Carmona, R., F.~Delarue, and D.~Lacker (2016).
\newblock Mean field games with common noise.
\newblock {\em Ann. {A}ppl. {P}robab.\/}~{\em 44\/}(6), 3740--3803.

\bibitem[\protect\citeauthoryear{Carmona and Lacker}{Carmona and
  Lacker}{2015}]{carmona2015probabilistic}
Carmona, R. and D.~Lacker (2015).
\newblock A probabilistic weak formulation of mean field games and
  applications.
\newblock {\em Ann. {A}ppl. {P}robab.\/}~{\em 25\/}(3), 1189--1231.

\bibitem[\protect\citeauthoryear{De~Angelis}{De~Angelis}{2015}]{de2015note}
De~Angelis, T. (2015).
\newblock A note on the continuity of free-boundaries in finite-horizon optimal
  stopping problems for one-dimensional diffusions.
\newblock {\em SIAM {J}. {C}ontrol {O}ptim.\/}~{\em 53\/}(1), 167--184.

\bibitem[\protect\citeauthoryear{De~Angelis, Federico, and Ferrari}{De~Angelis
  et~al.}{2017}]{de2017optimal}
De~Angelis, T., S.~Federico, and G.~Ferrari (2017).
\newblock Optimal boundary surface for irreversible investment with stochastic
  costs.
\newblock {\em Math. {O}per. {R}es.\/}~{\em 42\/}(4), 1135--1161.

\bibitem[\protect\citeauthoryear{De~Angelis and Peskir}{De~Angelis and
  Peskir}{2020}]{deangelisPeskir}
De~Angelis, T. and G.~Peskir (2020).
\newblock Global ${C}^{1}$ regularity of the value function in optimal stopping
  problems.
\newblock {\em Ann. {A}ppl. {P}robab.\/}~{\em 30\/}(3), 1007--1031.

\bibitem[\protect\citeauthoryear{De~Angelis and Stabile}{De~Angelis and
  Stabile}{2019}]{deangelisStabile}
De~Angelis, T. and G.~Stabile (2019).
\newblock On {L}ipschitz continuous optimal stopping boundaries.
\newblock {\em SIAM J. {C}ontrol {O}ptim.\/}~{\em 57\/}(1), 402--436.

\bibitem[\protect\citeauthoryear{Dianetti, Ferrari, Fischer, and
  Nendel}{Dianetti et~al.}{2020}]{dianetti2019submodular}
Dianetti, J., G.~Ferrari, M.~Fischer, and M.~Nendel (2020).
\newblock Submodular mean field games: {E}xistence and approximation of
  solutions.
\newblock {\em Ann.\ {A}ppl.\ {P}robab. (arXiv:1907.10968)\/}.

\bibitem[\protect\citeauthoryear{El~Karoui and Karatzas}{El~Karoui and
  Karatzas}{1988}]{karoui1988}
El~Karoui, N. and I.~Karatzas (1988).
\newblock Probabilistic aspects of finite-fuel, reflected follower problems.
\newblock {\em Acta {A}ppl. {M}ath.\/}~{\em 11\/}(3), 223--258.

\bibitem[\protect\citeauthoryear{El~Karoui and Karatzas}{El~Karoui and
  Karatzas}{1991}]{karoui1991new}
El~Karoui, N. and I.~Karatzas (1991).
\newblock A new approach to the {S}korohod problem, and its applications.
\newblock {\em Stochastics and Stochastic Reports\/}~{\em 34\/}(1-2), 57--82.

\bibitem[\protect\citeauthoryear{Ferrari, Riedel, and Steg}{Ferrari
  et~al.}{2017}]{ferrari2017continuous}
Ferrari, G., F.~Riedel, and J.-H. Steg (2017).
\newblock Continuous-time public good contribution under uncertainty: a
  stochastic control approach.
\newblock {\em Appl. {M}ath. Optim.\/}~{\em 75\/}(3), 429--470.

\bibitem[\protect\citeauthoryear{Friedman}{Friedman}{1988}]{F88}
Friedman, A. (1988).
\newblock {\em Variational principles and free-boundary problems\/} (Second
  ed.).
\newblock Robert E. Krieger Publishing Co., Inc., Malabar, FL.

\bibitem[\protect\citeauthoryear{Fu}{Fu}{2019}]{fu2019extended}
Fu, G. (2019).
\newblock Extended mean field games with singular controls.
\newblock {\em arXiv:1909.04154\/}.

\bibitem[\protect\citeauthoryear{Fu and Horst}{Fu and Horst}{2017}]{fu2017mean}
Fu, G. and U.~Horst (2017).
\newblock Mean field games with singular controls.
\newblock {\em SIAM {J}. {C}ontrol {O}ptim.\/}~{\em 55\/}(6), 3833--3868.

\bibitem[\protect\citeauthoryear{Gu{\'e}ant}{Gu{\'e}ant}{2012}]{gueant2012mean}
Gu{\'e}ant, O. (2012).
\newblock Mean field games equations with quadratic {H}amiltonian: a specific
  approach.
\newblock {\em Math. {M}odels {M}ethods {A}ppl. {S}ci.\/}~{\em 22\/}(09),
  1250022.

\bibitem[\protect\citeauthoryear{Guo, Tang, and Xu}{Guo
  et~al.}{2018}]{guo2018class}
Guo, X., W.~Tang, and R.~Xu (2018).
\newblock A class of stochastic games and moving free boundary problems.
\newblock {\em arXiv preprint arXiv:1809.03459\/}.

\bibitem[\protect\citeauthoryear{Guo and Xu}{Guo and
  Xu}{2019}]{guo2019stochastic}
Guo, X. and R.~Xu (2019).
\newblock Stochastic {G}ames for {F}uel {F}ollower {P}roblem: $n$ versus {M}ean
  {F}ield {G}ame.
\newblock {\em SIAM {J}. {C}ontrol {O}ptim.\/}~{\em 57\/}(1), 659--692.

\bibitem[\protect\citeauthoryear{Hu, {\O}ksendal, and Sulem}{Hu
  et~al.}{2014}]{hu2014singular}
Hu, Y., B.~{\O}ksendal, and A.~Sulem (2014).
\newblock Singular mean-field control games with applications to optimal
  harvesting and investment problems.
\newblock {\em arXiv:1406.1863\/}.

\bibitem[\protect\citeauthoryear{Huang, Malham{\'e}, and Caines}{Huang
  et~al.}{2006}]{huang2006large}
Huang, M., R.~P. Malham{\'e}, and P.~E. Caines (2006).
\newblock Large population stochastic dynamic games: closed-loop
  {M}c{K}ean-{V}lasov systems and the {N}ash certainty equivalence principle.
\newblock {\em Commun. {I}nf. {S}yst.\/}~{\em 6\/}(3), 221--252.

\bibitem[\protect\citeauthoryear{Jack, Johnson, and Zervos}{Jack
  et~al.}{2008}]{jack2008singular}
Jack, A., T.~C. Johnson, and M.~Zervos (2008).
\newblock A singular control model with application to the goodwill problem.
\newblock {\em Stoch.\ {P}rocess.\ {A}ppl.\/}~{\em 118\/}(11), 2098--2124.

\bibitem[\protect\citeauthoryear{Karatzas}{Karatzas}{1985}]{karatzas1985}
Karatzas, I. (1985).
\newblock Probabilistic aspects of finite-fuel stochastic control.
\newblock {\em Proc.\ Natl.\ Acad.\ Sci.\ USA\/}~{\em 82\/}(17), 5579--5581.

\bibitem[\protect\citeauthoryear{Karatzas and Shreve}{Karatzas and
  Shreve}{1984}]{karatzaShreve1984}
Karatzas, I. and S.~E. Shreve (1984).
\newblock Connections between optimal stopping and singular stochastic control
  {I}. {Monotone} follower problems.
\newblock {\em SIAM {J}. {C}ontrol {O}ptim.\/}~{\em 22\/}(6), 856--877.

\bibitem[\protect\citeauthoryear{Karatzas and Shreve}{Karatzas and
  Shreve}{1985}]{karatzasShreve1985}
Karatzas, I. and S.~E. Shreve (1985).
\newblock Connections between optimal stopping and singular stochastic control
  {II}. {Reflected} follower problems.
\newblock {\em SIAM {J}. {C}ontrol {O}ptim.\/}~{\em 23\/}(3), 433--451.

\bibitem[\protect\citeauthoryear{Karatzas and Shreve}{Karatzas and
  Shreve}{1998}]{karatzasShreve}
Karatzas, I. and S.~E. Shreve (1998).
\newblock {\em Brownian {M}otion and {S}tochastic {C}alculus}, Volume 113.
\newblock Springer-Verlag New York.

\bibitem[\protect\citeauthoryear{Kotlow}{Kotlow}{1973}]{K74}
Kotlow, D.~B. (1973).
\newblock A free boundary problem connected with the optimal stopping problem
  for diffusion processes.
\newblock {\em Trans. Amer. Math. Soc.\/}~{\em 184}, 457--478 (1974).

\bibitem[\protect\citeauthoryear{Kruse and Deely}{Kruse and
  Deely}{1969}]{kruse}
Kruse, R. and J.~Deely (1969).
\newblock Joint continuity of monotonic functions.
\newblock {\em Amer. Math. Monthly\/}~{\em 76\/}(1), 74--76.

\bibitem[\protect\citeauthoryear{Krylov}{Krylov}{2008}]{krylov2008lectures}
Krylov, N.~V. (2008).
\newblock {\em Lectures on elliptic and parabolic equations in {S}obolev
  spaces}, Volume~96.
\newblock American Mathematical Soc.

\bibitem[\protect\citeauthoryear{Lacker and Zariphopoulou}{Lacker and
  Zariphopoulou}{2019}]{lacker2019mean}
Lacker, D. and T.~Zariphopoulou (2019).
\newblock Mean field and {N}-agent games for optimal investment under relative
  performance criteria.
\newblock {\em Math. {F}inance\/}~{\em 29\/}(4), 1003--1038.

\bibitem[\protect\citeauthoryear{Lasry and Lions}{Lasry and
  Lions}{2007}]{lasry2007mean}
Lasry, J.-M. and P.-L. Lions (2007).
\newblock Mean field games.
\newblock {\em Japanese {J}. {M}ath.\/}~{\em 2\/}(1), 229--260.

\bibitem[\protect\citeauthoryear{Marinelli}{Marinelli}{2007}]{marinelli2007stochastic}
Marinelli, C. (2007).
\newblock The stochastic goodwill problem.
\newblock {\em European J.\ Oper.\ Res.\/}~{\em 176\/}(1), 389--404.

\bibitem[\protect\citeauthoryear{Peskir and Shiryaev}{Peskir and
  Shiryaev}{2006}]{peskirShyriyaev}
Peskir, G. and A.~Shiryaev (2006, 01).
\newblock {\em Optimal {S}topping and {F}ree-{B}oundary {P}roblems}.
\newblock Birkh\"{a}user: Basel.

\bibitem[\protect\citeauthoryear{Protter}{Protter}{1990}]{protter2005stochastic}
Protter, P.~E. (1990).
\newblock {\em Stochastic integration and differential equations}, Volume~21 of
  {\em Applications of Mathematics (New York)}.
\newblock Springer-Verlag, Berlin.

\bibitem[\protect\citeauthoryear{Revuz and Yor}{Revuz and
  Yor}{1999}]{revuz2013continuous}
Revuz, D. and M.~Yor (1999).
\newblock {\em Continuous {M}artingales and {B}rownian {M}otion\/} (Third ed.),
  Volume 293 of {\em Grundlehren der mathematischen Wissenschaften}.
\newblock Springer-Verlag, Berlin.

\bibitem[\protect\citeauthoryear{Soner and Shreve}{Soner and
  Shreve}{1991a}]{soner1989}
Soner, H. and S.~Shreve (1991a).
\newblock A free boundary problem related to singular stochastic control.
\newblock In {\em Applied stochastic analysis ({L}ondon, 1989)}, Volume~5 of
  {\em Stochastics Monogr.}, pp.\  265--301. Gordon and Breach, New York.

\bibitem[\protect\citeauthoryear{Soner and Shreve}{Soner and
  Shreve}{1991b}]{soner1991free}
Soner, H.~M. and S.~E. Shreve (1991b).
\newblock A free boundary problem related to singular stochastic control: the
  parabolic case.
\newblock {\em Comm.\ {P}artial {D}ifferential {E}quations\/}~{\em 16\/}(2-3),
  373--424.

\bibitem[\protect\citeauthoryear{Van~Moerbeke}{Van~Moerbeke}{1975}]{VM75}
Van~Moerbeke, P. (1975).
\newblock On optimal stopping and free boundary problems.
\newblock {\em Arch. Rational Mech. Anal.\/}~{\em 60\/}(2), 101--148.

\bibitem[\protect\citeauthoryear{Williams, Chow, and Menaldi}{Williams
  et~al.}{1994}]{williams1994regularity}
Williams, S.~A., P.~Chow, and J.-L. Menaldi (1994).
\newblock Regularity of the free boundary in singular stochastic control.
\newblock {\em J.\ Differential Equations\/}~{\em 111\/}(1), 175--201.

\bibitem[\protect\citeauthoryear{Zhang}{Zhang}{2012}]{zhang2012relaxed}
Zhang, L. (2012).
\newblock The relaxed stochastic maximum principle in the mean-field singular
  controls.
\newblock {\em arXiv:1202.4129\/}.

\bibitem[\protect\citeauthoryear{Zhou and Huang}{Zhou and
  Huang}{2017}]{zhou2017mean}
Zhou, M. and M.~Huang (2017).
\newblock Mean field games with {P}oisson point processes and impulse control.
\newblock In {\em 2017 IEEE 56th Annual Conference on Decision and Control
  (CDC)}, pp.\  3152--3157. IEEE.

\end{thebibliography}
}
\end{document}